\documentclass[a4paper]{amsart}

\usepackage{array}              
\usepackage{amssymb}              
\usepackage{lmodern,bm}              
\usepackage{longtable}
\usepackage{tikz}
\usepackage[all]{xy}
\usepackage{dsfont}
\usepackage{pdflscape}
\usepackage{bigdelim}
\usepackage{multirow}
\usepackage{wasysym}

\CompileMatrices

\newtheorem{theorem}{Theorem}[section]
\newtheorem{introthm}{Theorem}
\newtheorem{lemma}[theorem]{Lemma}
\newtheorem{proposition}[theorem]{Proposition}
\newtheorem{corollary}[theorem]{Corollary}

\theoremstyle{definition}

\newtheorem{introconstr}{Construction}
\newtheorem{definition}[theorem]{Definition}
\newtheorem{construction}[theorem]{Construction}

\newtheorem{example}[theorem]{Example}
\newtheorem{remark}[theorem]{Remark}

\newtheorem{introex}{Example}

\numberwithin{equation}{theorem}

\def\vector2#1#2{\left(\begin{array}{c} #1 \\ #2 \end{array}\right)}
\def\Cl{{\rm Cl}}

\def\CC{{\mathbb C}}

\def\TT{{\mathbb T}}
\def\ZZ{{\mathbb Z}}

\def\QQ{{\mathbb Q}}
\def\PP{{\mathbb P}}
\def\XX{{\mathbb X}}

\def\Chi{{\mathbb X}}
\def\Tau{{\rm T}}

\def\div{{\rm div}}
\def\quot{/\!\!/}
\def\mal{\! \cdot \!}

\def\im{{\rm im}}

\def\bangle#1{{\langle #1 \rangle}}
\def\rq#1{\widehat{#1}}
\def\t#1{\widetilde{#1}}
\def\b#1{\overline{#1}}

\def\Aut{{\rm Aut}}

\def\GL{{\rm GL}}
\def\SL{{\rm SL}}

\def\Lin{{\rm Lin}}

\def\GL{{\rm GL}}

\def\Supp{{\rm Supp}}
\def\Spec{{\rm Spec}}
\def\Proj{{\rm Proj}}

\def\cone{{\rm cone}}

\def\lcm{{\rm lcm}}
\def\im{{\rm im}}

\def\trop{{\rm trop}}
\def\lin{{\rm lin}}

\title[Log terminal singularities, platonic tuples, iteration of Cox rings]%
{Log terminal singularities, platonic tuples \\ and iteration of Cox rings}

\author[I.~Arzhantsev, L.~Braun, J.~Hausen and M.~Wrobel]{Ivan Arzhantsev, Lukas Braun, J\"urgen~Hausen and Milena Wrobel}

\address{National Research University Higher School of Economics, 
Faculty of Computer Science, Kochnovskiy Proezd 3, Moscow, 125319 Russia} 

\email{arjantsev@hse.ru}

\address{Mathematisches Institut, Universit\"at T\"ubingen,
Auf der Morgenstelle 10, 72076 T\"ubingen, Germany}
\email{juergen.hausen@uni-tuebingen.de}

\address{Mathematisches Institut, Universit\"at T\"ubingen,
Auf der Morgenstelle 10, 72076 T\"ubingen, Germany}
\email{braun@math.uni-tuebingen.de}

\address{Mathematisches Institut, Universit\"at T\"ubingen,
Auf der Morgenstelle 10, 72076 T\"ubingen, Germany}
\email{wrobel@math.uni-tuebingen.de}


\begin{document}

\begin{abstract}
Looking at the well understood case of log terminal 
surface singularities,
one observes that each of them is the quotient 
of a factorial one by a finite solvable group. 
The derived series of this group reflects an iteration of 
Cox rings of surface singularities. 
We extend this picture to log terminal singularities in
any dimension coming with a torus action of 
complexity one. In this setting, the previously finite
groups become solvable torus extensions.
As explicit examples, we investigate compound du Val 
threefold singularities.
We give a complete classification and exhibit all 
the possible chains of iterated Cox rings.
\end{abstract}

\maketitle

\section{Introduction}

We begin with a brief discussion of the well known 
surface case~\cite{Ar,Br,dV}.
The two-dimensional log terminal singularities 
are exactly the quotient singularities 
$\CC^2 / G$, where $G$ is a finite subgroup 
of the general linear group $\GL(2)$.
The particular case that  $G$ is a subgroup 
of $\SL(2)$ leads to the du Val singularities
$A_n$, $D_n$, $E_6$, $E_7$ and~$E_8$, named 
according to their resolution graphs.
They are precisely the rational double points,
and are also characterized by being the canonical 
surface singularities.
The du Val singularities fill the middle row 
of the following commutative diagram involving 
\emph{all} 
two-dimensional log terminal singularities:

\medskip

$$ 
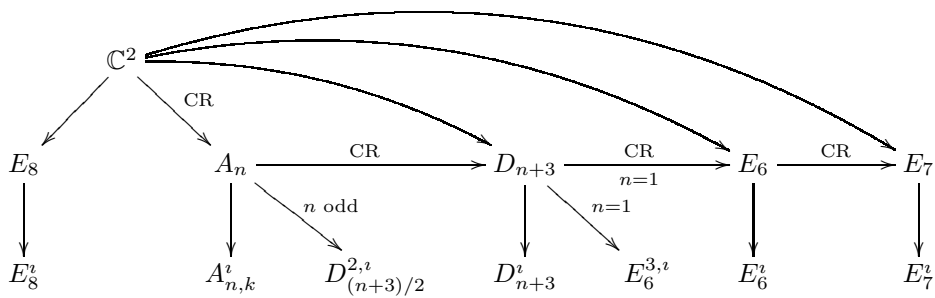
\xymatrix@C-5pt{
&
{\CC^2}
\ar[dl]
\ar[dr]^{\mathrm{CR}}
\ar@/^1pc/[drrr]
\ar@/^2pc/[drrrrr]
\ar@/^3pc/[drrrrrrr]
&
&
&
& 
&
&
&
\\
E_8
\ar[d]
&
&
A_n
\ar[rr]^{\mathrm{CR} \quad}
\ar[d]
\ar[dr]^{n \text{ odd}}
&
&
D_{n+3}
\ar[rr]^{\mathrm{CR}}_{n=1}
\ar[d]
\ar[dr]^{n=1}
&
&
E_6
\ar[rr]^{\mathrm{CR}}
\ar[d]
&
&
E_7
\ar[d]  
\\
E_8^{\imath}
&
&
A_{n,k}^{\imath}
&
D_{(n+3)/2}^{2,\imath}
&
D_{n+3}^{\imath}
&
E_6^{3,\imath}
&
E_6^{\imath}
&
&
E_7^{\imath}
}
$$
Here, all arrows indicate quotients by finite groups.
The label ``CR'' tells us that this quotient represents
a Cox ring of a du Val surface singularity; 
recall that the Cox rings of (the resolutions of) these
have been computed in~\cite{DB,FGL}, 
see also the example given below.
So, $E_6$ is the spectrum of the Cox ring of~$E_7$ etc..
In fact, the chain of Cox rings reflects the derived 
series of the binary octahedral group $\t{S}_4 \subseteq \SL(2)$, 
producing the $E_7$ singularity:
$$ 
\t{S}_4 
\ \supseteq \ 
\t{A}_4 
\ \supseteq \ 
\t{D}_4
\ \supseteq \ 
\{\pm I_2\}
\ \supseteq \
\{I_2\},
$$
where $\t{A}_4$ is the binary tetrahedral group,
$\t{D}_4$ the binary dihedral group, 
and $I_2$ stands for the $2 \times 2$ unit matrix.
The respective CR labelled arrows stand for 
quotients by the factors of this derived series.
The arrows passing from the middle to the lower row
indicate index-one covers: 
the upper surface is Gorenstein, one divides 
by a cyclic group of order $\imath$ and the lower
surface is of Gorenstein index~$\imath$. 
Finally, the superscripts~$2$ in~$D_{(n+3)/2}^{2,\imath}$ 
and~$3$ in~$E_6^{3,\imath}$ denote the ``canonical
multiplicity'' of the singularity, 
generalizing the ``exponent'' discussed 
in~\cite{Do,Eb}; see~\ref{def:canonmult}.
For a discussion of the surface case, including 
the determination of all Cox rings, based 
on the methods provided in this article, see 
Example~\ref{ex::lt-surfaces}.

\goodbreak

Another feature of the log terminal surface 
singularities is that, as quotients $\CC^2 / G$
by a finite subgroup $G \subseteq \GL(2)$, they all 
come with a non-trivial $\CC^*$-action, induced by scalar 
multiplication on $\CC^2$.
The higher dimensional analogue of  $\CC^*$-surfaces 
are $T$-varieties~$X$ of complexity one, that means 
varieties $X$ with an effective action of an algebraic torus 
$T$ which is of dimension one less than $X$.
The notion of log terminality is defined in general
via discrepancies in the ramification formula;
see Section~\ref{sec:acancomp} for a brief reminder.
In higher dimensions, log terminal singularities 
form a larger class than the quotient singularities 
$\CC^n/G$ with $G$ a finite subgroup of $\GL(n)$.
Our aim is, however, to extend the picture drawn at 
the beginning for the surface case to log terminal 
singularities with a torus action of complexity one
in any dimension.

We use the Cox ring based approach developed 
in~\cite{HaHe,HaSu,HaWr}.
Recall that the Cox ring of a normal variety $X$ 
with finitely generated divisor class group $\Cl(X)$ 
and only constant globally invertible functions is
$$ 
\mathcal{R}(X) 
\ := \ 
\bigoplus_{\Cl(X)} \Gamma(X,\mathcal{O}_X(D)),
$$
where we refer to~\cite{ArDeHaLa} for the 
necessary background.
If $X$ comes with a torus action of complexity 
one, then the Cox ring $\mathcal{R}(X)$ admits 
an explicit description in terms of generators 
and very specific trinomial relations. 
Vice versa, one can abstractly write down all
rings that arise as the Cox ring of some $T$-variety 
$X$ of complexity one. 
Let us briefly summarize the procedure; 
see~Section~\ref{sec:rat-tvar} and~\cite{HaHe,HaWr} 
for the details.

\begin{introconstr}
\label{constr:ringR1}
Fix integers $m \ge 0$, $\iota \in \left\{0,1\right\}$ 
and $r,n >0$ and a partition $n=n_\iota + \dots + n_r$. 
For every $i = \iota, \ldots, r$ let 
$l_i := (l_{i1}, \ldots , l_{in_i}) \in \ZZ_{>0}^{n_i}$
with 
$l_{i1} \ge \ldots \ge l_{in_i}$
and 
$l_{\iota 1} \ge \ldots \ge l_{r1}$
and define a monomial 
$$
T_{i}^{l_{i}}
\ := \
T_{i1}^{l_{i1}} \cdots T_{in_{i}}^{l_{in_{i}}}.
$$
Denote the polynomial ring 
$\CC[T_{ij},S_k; \, i = \iota, \ldots, r, \, 
j = 1, \ldots, n_i, \, k = 1, \ldots, m]$
for short by $\CC[T_{ij},S_k]$.
We distinguish two types of rings:

\medskip
\noindent
\emph{Type 1.}
Take $\iota =1$ and pairwise different scalars 
$\theta_1 = 1, \theta_{2}, \ldots, \theta_{r-1} \in \CC^*$ 
and define for each $i = 1, \ldots, r-1$ a trinomial 
$$
g_i  
\ := \   
T_{i}^{l_{i}} - T_{i+1}^{l_{i+1}} - \theta_i.
$$
Then we obtain a factor ring
$$ 
R  
\ = \ 
\CC[T_{ij},S_k] / \bangle{g_1, \ldots, g_{r-1}}.
$$

\medskip
\noindent
\emph{Type 2.}
Take $\iota =0$ and pairwise different scalars 
$\theta_0 = 1, \theta_{1}, \ldots, \theta_{r-2} \in \CC^*$ 
and define for each $i = 0, \ldots, r-2$ a trinomial 
$$
g_i  
\ := \   
\theta_i T_{i}^{l_{i}} + T_{i+1}^{l_{i+1}} + T_{i+2}^{l_{i+2}}.
$$
Then we obtain a factor ring
$$ 
R  
\ = \ 
\CC[T_{ij},S_k] / \bangle{g_0, \ldots, g_{r-2}}.
$$
\end{introconstr}

\goodbreak

As we explain later, the rings $R$ come 
with a natural grading by a finitely generated 
abelian group $K_0$ and suitable downgradings 
$K_0 \to K$ give us Cox rings of rational, 
normal varieties $X$ with $\Cl(X) = K$ that 
come with a torus action of complexity one. 
More geometrically, $X$ arises as a quotient 
of an open set $\rq{X} \subseteq \b{X}$ 
of the total coordinate space $\b{X} = \Spec \, R$ 
by the quasitorus $H$ having $K$ 
as its character group.
Conversely, basically every rational, normal
variety $X$ with a torus action of complexity 
one can be presented this way.

Geometrically speaking, Type~1 leads to 
the $T$-varieties of complexity one that admit 
non-constant global invariant functions 
and Type~2 to those having only constant global 
invariant functions, or, equivalently, having 
an attractive fixed point.
The varieties of Type~1 turn out to be locally 
isomorphic to toric varieties.
In particular, they are all log terminal 
and the study of their singularities is essentially 
toric geometry, see Corollary~\ref{cor:sing-type1} 
for a precise formulation.
We therefore mainly concentrate on Type~2. There, 
the true non-toric phenomena occur, as for instance 
the singularities $D_n$, $E_6$, $E_7$ and $E_8$ 
in the surface case.

Characterizing log terminality for a $T$-variety 
of complexity one of Type~2 involves platonic 
triples, that means, triples of the form 
$$
(5,3,2),\quad
(4,3,2),\quad
(3,3,2),\quad
(x,2,2),\quad
(x,y,1),\quad
$$
where $x \ge y \in \ZZ_{\ge 1}$. 
We say that positive integers 
$a_0, \ldots, a_r$ form a \emph{platonic tuple}
if, after reordering decreasingly, the first
three numbers are a platonic triple and all
others equal one.
Moreover, in the setting of Construction~\ref{constr:ringR1},
we say that a ring $R$ of Type~2 is \emph{platonic} 
if every $(l_{0j_0}, \ldots, l_{rj_r})$ is a platonic 
tuple.

\begin{introex}
The platonic rings of Type~2 in dimension two are the 
polynomial ring $\CC[T_1,T_2]$ and the factor rings 
$\CC[T_1,T_2,T_3] / \bangle{f}$, where $f$ is one of
$$ 
T_1^y + T_2^2 + T_3^2, \ y \in \ZZ_{>1},
\quad
T_1^3 + T_2^3 + T_3^2,
\quad
T_1^4 + T_2^3 + T_3^2,
\quad
T_1^5 + T_2^3 + T_3^2.
$$
Endowed with a suitable grading, 
$\CC[T_1,T_2]$ is the Cox ring of $A_n$,
and the other rings, according to the above order of listing, 
of $D_{y-2}$, $E_6$, $E_7$ and $E_8$. 
\end{introex}

Our first result says that a rational, normal variety $X$ 
with a torus action of complexity one of Type~2
has at most log terminal singularities if and only if 
there occur enough platonic tuples 
$(l_{0j_0}, \ldots, l_{rn_r})$ in the Cox ring~$R$; 
see Theorem~\ref{thm:logtermchar} for the precise 
meaning of ``enough''. 
In the affine case, the result specializes to the 
following; compare also~\cite[Ex.~2.20]{FlZa} for 
an earlier result in a particular case 
and~\cite[Cor.~5.8]{LiSu} for a 
related characterization.

\begin{introthm}
\label{thm:affltcharintro}
An affine, normal, $\QQ$-Gorenstein, rational 
variety $X$ with torus action of complexity one 
of Type~2 has at most log terminal singularities 
if and only if its Cox ring~$R$ is a platonic ring.
\end{introthm}

Set for the moment $\mathfrak{l}_i := \gcd(l_{i1}, \ldots, l_{in_i})$.
Then, by~\cite{HaWr}, a ring $R$ of
Type~1 is factorial if and only if $\mathfrak{l}_i=1$ 
holds for all $i= 1, \dots, r$.
Moreover, a ring $R$ of Type~2 is factorial if and 
only if the~$\mathfrak{l}_i$ are pairwise coprime for 
$i = 0, \ldots, r$, 
see~\cite[Thm.~1.1]{HaHe}.

\begin{introex}
In dimension two, the factorial platonic rings $R$ 
of Type~2 are the polynomial ring $\CC[T_1,T_2]$ and 
the ring $\CC[T_1,T_2,T_3] / \bangle{T_1^5 + T_2^3 + T_3^2}$.
\end{introex}

To extend the iteration of Cox rings 
$\CC^2 \to A_1 \to D_4 \to E_6 \to E_7$ observed 
in the surface case to higher dimensions,
we have to allow instead of only finite abelian groups
also non-finite abelian groups 
in the respective quotients.

\goodbreak

\begin{introthm}
\label{theorem::CoxIteration}
Let $X_1$ be a rational, normal, affine variety 
with a torus action of complexity one of Type~2
and at most log terminal singularities.
Then there is a unique chain of quotients
$$
\xymatrix{ 
X_p 
\ar[r]^{ \quot H_{p-1} \ } 
&
X_{p-1} 
\ar[r]^{\quot H_{p-2} \ } 
&
\quad \dots \quad 
\ar[r]^{ \quot H_3 \ } 
&
X_3 
\ar[r]^{\quot H_2 \ }
&
X_2 
\ar[r]^{\quot H_1 \ }
&
X_1 },
$$ 
where $X_i= \Spec(R_i)$ holds with a platonic 
ring $R_i$ for $i \ge 2$, 
the ring $R_p$ is factorial 
and each $X_{i}\rightarrow X_{i-1}$ is the total 
coordinate space.
\end{introthm}

Note that iteration of Cox rings requires
in each step finite generation of the divisor 
class group $\Cl(\b{X})$ of the total coordinate 
space of $X$.
The latter merely means that the curve $Y$ with 
function field $\CC(\b{X})^{H_0^0}$ is of genus zero,
where $H_0^0 \subseteq H_0$ is the unit component
of the quasitorus $H_0$ with character group $\Cl(X)$.
In Theorem~\ref{theo::curve}, we establish a formula for 
the genus of $Y$ in terms of the entries~$l_{ij}$ of 
the defining matrix $P$ of $R = \mathcal{R}(X)$,
generalizing the case of $\CC^*$-surfaces
settled in~\cite[Prop.~3, p.~64]{OrWa}.
This allows us to conclude that for log terminal 
affine $X$, the total coordinate space is always 
rational.
Together with the fact that the total coordinate space
of a log terminal affine $X$ is canonical, 
see Proposition~\ref{lemm::GorenstCan}, 
we obtain that Cox ring iteration is possible
in the log terminal case; 
see Remark~\ref{rem:rattcs} for a discussion of 
a non log terminal example with rational Cox ring.
The final step in proving Theorem~\ref{theorem::CoxIteration}
is to show that the Cox ring iteration even stops 
after finitely many steps. 
For this, we compute explicitly in 
Proposition~\ref{prop::isotropy} the equations 
of the iterated Cox ring.
It seems to be interesting to study Cox ring iteration 
also more generally; note that a $\QQ$-factorial variety 
has a log terminal Cox ring if and only if it is 
log Fano~\cite{BrM,Go}.

The next result shows that, in a large sense, the 
log terminal singularities with torus action of 
complexity one still can be regarded as quotient 
singularities: the affine plane $\CC^2$ and the 
finite group $G \subseteq \GL(2)$ of the surface
case have to be replaced with a factorial affine 
$T$-variety of complexity one and a solvable reductive 
group.

\begin{introthm}
\label{theorem::quotient}
Let $X$ be a rational, normal, affine variety of 
Type~2 with a torus action of complexity one
and at most log terminal singularities.
\begin{enumerate}
\item
$X$ is a quotient $X = X' \quot G$ 
of a factorial affine variety $X' := \Spec(R')$ 
by a solvable reductive group $G$, 
where $R'$ is a factorial platonic ring.  
\item
The presentation of Theorem~\ref{theorem::CoxIteration} 
is regained by $H_i:= G^{(i-1)}/G^{(i)}$ and $X_i:=X'/G^{(i-1)}$, 
where $G^{(i)}$ is the $i$-th derived subgroup of~$G$.
\end{enumerate}
\end{introthm}

\begin{introex}
Every log terminal affine $\CC^*$-surface 
is a quotient of $\CC^2$ or the $E_8$-singular 
surface 
$V(T_1^5 + T_2^3 + T_3^2) \subseteq \CC^3$
by a finite solvable group.   
\end{introex}

A natural three-dimensional generalization of du Val 
singularities are the \emph{compound du Val singularities},
introduced in~\cite{Re}:
these are normal, canonical Gorenstein threefold singularities 
$x \in X$ such that a general hypersurface section 
through~$x$ has a du Val (surface) singularity at $x$.
The isolated compound du Val singularities are precisely 
the terminal Gorenstein singularities.
If a threefold $X$ admits at most compound du Val 
singularities,
then, for a given singular point $x \in X$, we have 
possible one-dimensional irreducible components 
$C_1, \ldots, C_r$ of the singular locus that 
contain $x$.
The compound du Val singularity type (cDV-type) 
of $x$ is denoted by 
$S(x_1), \ldots, S(x_r) \to cS(x)$, where $S(x_i)$ 
stands for the type of the du Val surface singularity 
obtained by a general hypersurface section 
through a general point of $x_i \in C_i$ and 
$S(x)$ for that through $x$; the $c$ just 
indicates compound du Val.
The following result goes one step beyond 
the known~\cite{Da} case of toric compound 
du Val singularities.

\goodbreak

\begin{introthm}
\label{thm:cdv-class-intro}
The following table provides the equations 
for the affine threefolds with at most 
compound du Val singularities which are 
toric (nos. 1 -- 3) or non-toric with a
torus action of complexity one of Type~2 
(nos. 4 -- 18).

\goodbreak

\renewcommand{\arraystretch}{1.8}

\begin{longtable}{c|c|c}
No. 
& 
cDV-type 
& 
equation in $\mathbb{C}^4$   
\\
\hline 
1 
& 
$A_{l} \times \mathbb{C}$ 
& 
$T_1T_2+T_3^{l+1}$ 
\\
\hline 
2 
& 
$A_{l_1-1}, A_{l_2-1} \to cA_{l_1+l_2-1}$ 
& 
$T_1T_2+T_3^{l_1}T_4^{l_2}$   
\\
\hline 
3 
& 
$A_1,A_1,A_1 \to  cD_4$ 
& 
$T_1^2+T_2T_3T_4$  
\\
\hline 
4 
& 
$D_{l+3} \times \mathbb{C}$
& 
$T_1^2 + T_2^2T_3 +T_3^{l+2}$  
\\
\hline
5 
& 
$A_1, A_{l-1} \to cD_{l+4}$ 
& 
$T_1^2 + T_2^2T_3 +T_3T_4^{l+2}$   
\\
\hline
6 
& 
$E_{6}\times\mathbb{C}$
& 
$T_1^2 + T_2^3 +T_3^4$  
\\
\hline
7 
& 
$E_{7}\times\mathbb{C}$
& 
$T_1^2 + T_2^3 +T_2T_3^3$  
\\
\hline
8 
& 
$E_{8}\times\mathbb{C}$
& 
$T_1^2 + T_2^3 +T_3^5$  
\\
\hline
9a
& 
$A_{l-1} \to cA_{L}$ 
& \tiny{
$
\begin{array}{l}
T_1T_2+\left(T_3^{L_1+1} + T_4^{L_2+1}\right)^{l},  
\\
L= \min(L_1+1,L_2+1)l - 1
\end{array}
$}
\\
\hline
9b
& 
$A_{l_j-1} \to cA_{L}$ 
& \tiny{
$
\begin{array}{l}
T_1T_2 + \prod_{j=1}^{r-1}\left(jT_3^{L_1+1} + (2j-1)T_4^{L_2+1}\right)^{l_j},    
\\
L=\min(L_i+1)\sum l_j -1
\end{array} 
$}
\\
\hline
9c
& 
$A_{L_3-1}, A_{l_j-1} \to cA_{L}$ 
&  \tiny{
$
\begin{array}{l}
T_1T_2 + T_3^{L_3}\prod_{j=1}^{r-1}\left(jT_3^{L_1} + (2j-1)T_4^{L_2+1}\right)^{l_j},
\\
L=\min(L_3+L_1\sum l_j -1,L_2 \sum l_j -1) 
\end{array}
$}
\\
\hline
9d
& 
$A_{L_3-1},A_{L_4-1}, A_{l_j-1} \to cA_{L}$ 
& \tiny{
$
\begin{array}{l}
T_1T_2 + T_3^{L_3}T_4^{L_4}\prod_{j=1}^{r-1}\left(jT_3^{L_1} 
+ (2j-1)T_4^{L_2}\right)^{l_j},    
\\
L=\min_{k=3,4}(L_k + l_{k-2}\sum l_j -1) 
\end{array}
$}
\\
\hline
10
& 
$A_{l+1} \to cD_{l+3}$ 
& 
$T_1^2 + T_2^2T_3+T_4^{l+2}$  
\\
\hline
11
& 
$A_{2l+1} \to cD_{2l+2}$ 
& 
$T_1^2 + T_2^2T_3+T_2T_4^{l+1}$   
\\
\hline
12
& 
$A_{l_2-1}, D_{l_1+2} \to cD_{l_1+l_2+2}$ 
& 
$T_1^2 + T_2^2T_3+T_3^{l_1+1}T_4^{l_2}$ 
\\
\hline
13
& 
$A_{1}, A_{1} \to cD_{l+3}$ 
& 
$T_1^2 + T_2T_3T_4+T_4^{l+2}$   
\\
\hline
14
& 
$A_1, A_{1}, A_2 \to cE_6$ 
& 
$T_1^2 + T_2^3+T_3^2T_4^2$  
\\
\hline
15
& 
$D_{4} \to cE_6, cE_7$ 
& 
$T_1^2 + T_2^3+T_3^3T_4$  
\\
\hline
16
& 
$A_1, D_4 \to cE_7$ 
& 
$T_1^2 + T_2^3+T_2T_3T_4^2$   
\\
\hline
17
& 
$A_2, D_4 \to cE_8$ 
& 
$T_1^2 + T_2^3 + T_3^2T_4^3$   
\\
\hline
18
& 
$E_6 \to cE_8$ 
& 
$T_1^2 + T_2^3 + T_3T_4^4$    
\end{longtable}

\renewcommand{\arraystretch}{1} 

\noindent
Here, parameters are integers greater than zero with the 
exponents containing $L_1$, $L_2$ in nos. 9a to 9d 
being coprime, $A_0$ means that there is no singularity 
and $D_l \cong A_l$ for $l \leq 3$. 
\end{introthm}

The defining data as toric or $T$-varieties of complexity one
for the varieties listed in Theorem~\ref{thm:cdv-class-intro}
are provided in Section~\ref{sec:cDV}.
Finally, we study the possible Cox ring 
iterations of the compound du Val singularities.

\begin{introthm}
\label{thm:cdv-graph-intro}
For the singularities from 
Theorem~\ref{thm:cdv-class-intro},
one has the following Cox ring iterations;
the respective total coordinate spaces are 
indicated by the downward arrows:

$$
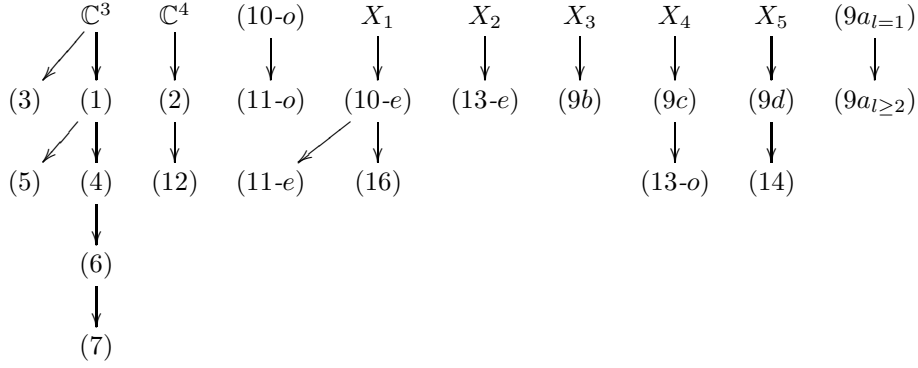
\xymatrix@C8pt@R15pt{
& 
\mathbb{C}^3 \ar[d] \ar[ld]
&
\mathbb{C}^4  \ar[d]  
& (10\text{-}o) \ar[d] 
& X_1 \ar[d] 
& X_2 \ar[d] 
& X_3 \ar[d] 
& X_4 \ar[d] 
& X_5 \ar[d] 
& (9a_{l=1}) \ar[d] 
\\
(3)
& (1) \ar[d] \ar[ld] 
& (2) \ar[d]  
& (11\text{-}o)  
& (10\text{-}e) \ar[ld] \ar[d] 
&   (13\text{-}e) 
& (9b) 
& (9c)  \ar[d] 
& (9d)  \ar[d] 
& (9a_{l\geq 2}) 
\\
(5) 
& (4)  \ar[d] 
&  (12) 
& (11\text{-}e) 
& (16) 
& & & (13\text{-}o) 
& (14) \\
&(6) \ar[d] 
\\
&(7)
}
$$

\goodbreak

\noindent
Here $10\text{-}e$ ($10\text{-}o$) denotes the singularity $10$ with even (odd) 
parameter; similarly in the other cases. 
Moreover with the respective parameters from Theorem~\ref{thm:cdv-class-intro}:
$$
X_1
\ = \
V(T_1^2T_2 + T_3^2T_4+T_5^{\frac{l+2}{2}}), \,
\qquad
X_2  
\ = \
V(T_1T_2+T_3T_4+T_5^{l-1}),
$$
\begin{align*}
X_5
\, = \,  
V(&
T_1^{L_1}T_2^{L_1}+T_3^{L_2}T_4^{L_2} + T_5T_6, 
\,
T_3^{L_2}T_4^{L_2} +2T_5T_6 + T_7T_8, 
\\
&T_5T_6 + 3T_7T_8 + T_9T_{10},
\, \ldots, \, 
T_{2r-3}T_{2r-2} +(r-1)T_{2r-1}T_{2r} + T_{2r+1}T_{2r+2} ).
\end{align*}
To obtain $X_4$, set $T_4=1$ and for $X_3$ in addition $T_2=1$ 
in the equations of $X_5$.
The singularities 8, 15, 17 and 18 are factorial. 
\end{introthm}

The varieties $X_1,\ldots,X_5$ in Theorem~\ref{thm:cdv-graph-intro} 
are of dimension four or higher. 
They enjoy a generalized compound du Val property in 
the sense that the hyperplane section
$X_i \cap V(T_4-T_3)$ has at most canonical singularities. 
For instance, for~$X_2$, the hyperplane section 
gives a compound du Val singularity of Type~9a.
The composition $\CC^3 \to (1) \to (5)$ is a quotient 
by the dihedral group $D_{2l+4}$, which is not a subgroup 
of $\SL(2)$.

We would like thank the referee for carefully reading 
the manuscript and for many helpful remarks.

\tableofcontents

\section{Rational varieties with torus 
action of complexity one}
\label{sec:rat-tvar}

We recall the basic concepts and facts on normal 
rational \emph{$T$-varieties $X$ of complexity one}, 
i.e., the variety $X$ is endowed with an effective  
action $T \times X \to X$ of an algebraic torus $T$ 
such that $\dim(T) = \dim(X)-1$ holds.
We work over the field~$\CC$ of complex numbers.
For the proofs and full details, we refer 
to~\cite{ArDeHaLa,HaHe,HaSu,HaWr}.

The approach follows the general philosophy 
behind~\cite[Chap.~3]{ArDeHaLa}: one starts 
with a Cox ring $R = \mathcal{R}(X)$
and then obtains $X$ as a quotient $X = \rq{X} \quot H$ 
of an open subset $\rq{X} \subseteq  \b{X}$ of the 
\emph{total coordinate space} 
$\b{X} = \Spec \, R$ by the action of the 
\emph{characteristic quasitorus} $H = \Spec \, \CC[K]$,
where $K \cong \Cl(X)$ is the divisor class 
group of~$X$. The quotient map $\rq{X} \to X$ is 
called the \emph{characteristic space} over $X$.
In our concrete case of $T$-varieties of complexity 
one, the total coordinate space $\b{X}$ will be acted on 
by a larger quasitorus $H_0 = \Spec \, \CC[K_0]$ 
containing the characteristic quasitorus $H$ as 
a closed subgroup and the torus action on 
$X = \rq{X} \quot H$ will be the induced action 
of $T = H_0/H$.

Our first step provides $K_0$-graded rings $R$, 
which after suitable downgrading become 
prospective Cox rings of our $T$-varieties.
The construction depends on continuous data $A$ and 
discrete data $P_0$ introduced below.
There are two types of input data $(A,P_0)$:
for Type~1, we will have the affine line as a
generic quotient of the action of $H_0$ on $\b{X}$ 
and Type~2 will lead to the projective line.

\begin{construction}
\label{constr:RAP0}
Fix integers $r,n > 0$, $m \ge 0$ and a partition 
$n = n_\iota+\ldots+n_r$ starting at $\iota \in \{0,1\}$.
For each $\iota \le i \le r$, fix a tuple
$l_{i} \in \ZZ_{> 0}^{n_{i}}$ and define a monomial
$$
T_{i}^{l_{i}}
\ := \
T_{i1}^{l_{i1}} \cdots T_{in_{i}}^{l_{in_{i}}}
\ \in \
\CC[T_{ij},S_{k}; \ \iota \le i \le r, \ 1 \le j \le n_{i}, \ 1 \le k \le m].
$$
We will also write $\CC[T_{ij},S_{k}]$ for the 
above polynomial ring. 
We distinguish two settings for the input data 
$A$ and $P_0$ of the graded $\CC$-algebra $R(A,P_0)$.

\medskip
\noindent
\emph{Type~1.} Take $\iota = 1$.
Let $A := (a_1, \ldots, a_r)$ be a list of pairwise 
different elements of~$\CC$.
Set $I := \left\{1, \ldots, r-1\right\}$
and define for every $i \in I$ a polynomial
$$
g_{i} 
\ := \ 
T_i^{l_i} - T_{i+1}^{l_{i+1}} - (a_{i+1}-a_i) 
\ \in \ 
\CC[T_{ij}, S_k].
$$
We build up an $r \times (n+m)$ matrix 
from the exponent vectors $l_1, \ldots, l_r$ of these 
polynomials:
$$
P_{0}
\ := \
\left[
\begin{array}{cccccc}
l_{1} &  & 0 & 0  &  \ldots & 0
\\
\vdots  & \ddots & \vdots & \vdots &  & \vdots
\\
 0 &  & l_{r} & 0  &  \ldots & 0
\end{array}
\right].
$$

\medskip
\noindent
\emph{Type~2. } Take $\iota = 0$.
Let $A:= (a_0, \ldots, a_r)$ be a $2 \times (r+1)$-matrix with pairwise 
linearly independent columns $a_i \in \CC^2$. 
Set $I := \left\{0, \ldots, r-2\right\}$ and for every $i \in I$
define
$$
g_{i}
\ :=  \
\det
\left[
\begin{array}{lll}
T_i^{l_i} & T_{i+1}^{l_{i+1}} & T_{i+2}^{l_{i+2}}
\\
a_i & a_{i+1}& a_{i+2}
\end{array}
\right]
\ \in \
\CC[T_{ij},S_{k}].
$$
We build up an $r \times (n+m)$ matrix 
from the exponent vectors $l_0, \ldots, l_r$ of these 
polynomials:
$$
P_{0}
\ := \
\left[
\begin{array}{ccccccc}
-l_{0} & l_{1} &  & 0 & 0  &  \ldots & 0
\\
\vdots & \vdots & \ddots & \vdots & \vdots &  & \vdots
\\
-l_{0} & 0 &  & l_{r} & 0  &  \ldots & 0
\end{array}
\right].
$$
We now define the ring $R(A,P_0)$ simultaneously 
for both types in terms of the data $A$ and $P_0$.
Denote by $P_0^*$ the transpose of $P_0$ and consider 
the projection
$$
Q \colon \ZZ^{n+m} 
\ \to \ 
K_{0} 
\ := \ 
\ZZ^{n+m}/\mathrm{im}(P_{0}^{*}).
$$
Denote by $e_{ij},e_{k} \in \ZZ^{n+m}$ the canonical
basis vectors corresponding to the variables 
$T_{ij}$, $S_{k}$.
Define a $K_0$-grading on $\CC[T_{ij},S_{k}]$ 
by setting
$$
\deg(T_{ij}) \ := \ Q(e_{ij}) \ \in \ K_{0},
\qquad
\deg(S_{k}) \ := \ Q(e_{k}) \ \in \ K_{0}.
$$
This is the coarsest possible grading of
$\CC[T_{ij},S_{k}]$ leaving the variables 
and the $g_i$ homogeneous.
In particular, we have  a $K_{0}$-graded 
factor algebra
$$
R(A,P_{0})
\ := \
\CC[T_{ij},S_{k}] / \bangle{g_{i}; \ i \in I}.
$$
\end{construction}

The  $\CC$-algebra $R(A,P_0)$ just constructed 
is an integral normal complete intersection of 
dimension $n+m+1 - r$ admitting only constant 
invertible homogeneous elements.
Moreover, $R(A,P_0)$ is $K_0$-factorial in the 
sense that every non-zero homogeneous non-unit
is a product of $K_0$-primes.
The latter merely means that on 
$\b{X} = \Spec\, R(A,P_0)$ every $H_0$-invariant 
divisor is the divisor of an $H_0$-homogeneous rational 
function. 
Moreover, every affine variety with a quasitorus 
action of complexity one having this property and 
admitting only constant invertible homogeneous functions
arises from Construction~\ref{constr:RAP0},
see~\cite[Sec.~4.4.2]{ArDeHaLa}.

In the second construction step, we introduce the 
downgradings $K_0 \to K$ that will turn 
$R(A,P_0)$ into a Cox ring.
More geometrically speaking, we figure out the 
possible characteristic quasitori $H \subseteq H_0$.
This is achieved by suitably enhancing the matrix 
$P_0$.

\begin{construction}
\label{constr:RAPdown}
Let integers $r$, $n = n_\iota + \ldots + n_r$, $m$ 
and data $A$ and $P_0$ of Type~1 or Type~2 as 
in Construction~\ref{constr:RAP0}. 
Fix $1 \le s \le n + m - r$, choose an integral
$s \times (n + m)$ matrix $d$ and build the 
$(r+s) \times (n + m)$ stack matrix
$$
P 
\ := \
\left[
\begin{array}{c}
P_0
\\
d
\end{array}
\right].
$$
We require the columns of $P$ to be pairwise 
different primitive vectors generating
$\QQ^{r+s}$ as a vector space. 
Let $P^*$ denote the transpose of $P$ and 
consider the projection
$$
Q \colon 
\ZZ^{n+m} 
\ \to \ 
K 
\ := \ 
\ZZ^{n+m} / \mathrm{im}(P^*).
$$
Denoting as before by $e_{ij}, e_k \in \ZZ^{n+m}$ the 
canonical basis vectors corresponding to
the variables $T_{ij}$ and $S_k$, we obtain a 
$K$-grading on $\CC[T_{ij}, S_k]$ by setting
$$
\deg(T_{ij}) \ := \ Q(e_{ij} ) \ \in \ K,
\qquad\qquad
\deg(S_k) \ := \ Q(e_k) \ \in \ K.
$$
This $K$-grading coarsens the $K_0$-grading of 
$\CC[T_{ij},S_k ]$ given in Construction~\ref{constr:RAP0}. 
In particular, we have the $K$-graded factor algebra
$$
R(A,P)
\ := \
\CC[T_{ij},S_{k}] / \bangle{g_{i}; \ i \in I}.
$$
\end{construction}

So, as algebras $R(A,P_0)$ and $R(A,P)$ coincide,
but the latter comes with the coarser $K$-grading. 
Again, $R(A,P)$ is $K$-factorial, i.e., 
for the action of $H = \Spec \, \CC[K]$ 
on $\b{X} = \Spec \, R(A,P)$, 
every $H$-invariant divisor is the divisor of 
an $H$-homogeneous function. 

\begin{remark}
\label{remark:admissibleops}
Consider the defining matrix~$P$ of a $K$-graded 
ring $R(A,P)$ as in Construction~\ref{constr:RAPdown}.
Write $v_{ij} = P(e_{ij})$ and $v_k = P(e_k)$ 
for the columns of $P$.
The $i$-th column block of $P$ is
$(v_{i1}, \ldots, v_{in_i})$ and by
the data of this block we mean $l_i$ and 
the $s \times n_i$ block $d_i$ 
of $d$.
We introduce \emph{admissible operations}
on $P$:
\begin{enumerate}
\item
swap two columns inside a block
$v_{i1}, \ldots, v_{in_i}$,
\item
exchange the data $l_{i_1},d_{i_1}$ and 
$l_{i_2},d_{i_2}$ of two column blocks,
\item
add multiples of the upper $r$ rows
to one of the last $s$ rows,
\item
any elementary row operation among the last $s$
rows,
\item
swapping among the last $m$ columns.
\end{enumerate}
The operations of type (iii) and (iv) do not change
the associated ring $R(A,P)$, whereas the 
types (i), (ii), (v)
correspond to certain renumberings of the variables
of $R(A,P)$ keeping the (graded) isomorphy type.
\end{remark}

\begin{remark}
\label{rem:redundant}
If $R(A,P)$ is not a polynomial ring, 
then we can always assume that~$P$ is 
\emph{irredundant} in the sense that 
$l_{i1} + \ldots + l_{in_i} > 1$ holds
for $i = 0, \ldots, r$.
Indeed, if $P$ is redundant, then we have 
$n_i = 1$ and $l_{i1} = 1$ for some $i$.
After an admissible operation of type~(ii),
we may assume $i=r$.
Now, erasing $v_{r1}$ and the $r$-th row 
of $P$ and the last column from $A$ 
produces new data defining a ring 
$R(A,P)$ isomorphic to the previous
one.
Iterating this procedure leads to 
an~$R(A,P)$ isomorphic to the 
initial one but with irredundant~$P$.   
\end{remark}

\begin{remark}
Construction~\ref{constr:RAPdown} allows 
more flexibility than the simpler version
presented in the introduction.
However, given any $R(A,P)$ as in
Construction~\ref{constr:RAPdown}, 
we can achieve $l_{i1} \ge \ldots \ge l_{in_i}$ 
for all $i$ and $l_{\iota 1} \ge \ldots \ge l_{r1}$
by means of admissible operations of type~(i) 
and~(ii).
Moreover, via suitable scalings of the variables~$T_{ij}$, 
we can turn the coefficients of the 
relations $g_{i}$ into those presented in the 
introduction.
\end{remark}

The algebras $R(A,P)$ will be our prospective 
Cox rings. 
The remaining task is to determine the 
open $H$-invariant sets
$\rq{X} \subseteq \b{X} = \Spec \, R(A,P)$ 
that give rise to suitable quotients 
$X = \rq{X} \quot H$.
This is done via geometric invariant theory:
the respective open sets $\rq{X} \subseteq \b{X}$ 
are in correspondence with ``bunches of cones'',
certain collections $\Phi$ of convex polyhedral cones 
in $K_{\QQ} := K \otimes_{\ZZ} \QQ$; 
we refer to~\cite[Sec.~3.2.1]{ArDeHaLa} for a detailed 
introduction.

\begin{construction}
\label{constr:RAPandPhi}
Let $R(A,P)$ be a $K$-graded ring as provided by 
Construction~\ref{constr:RAPdown} and 
$\mathfrak{F} = (T_{ij},S_k)$ the 
canonical system of generators.
Consider
$$
H \ := \ \Spec \, \CC[K],
\qquad\qquad
\b{X}(A,P) \ :=  \ \Spec \, R(A,P).
$$
Then $H$ is a quasitorus and the $K$-grading
of $R(A,P)$ defines an action of $H$ on~$\b{X}(A,P)$.
Any true $\mathfrak{F}$-bunch $\Phi$ defines 
an $H$-invariant open set and a good quotient
$$ 
\rq{X}(A,P,\Phi) \ \subseteq \ \b{X}(A,P),
\qquad\qquad
X(A,P,\Phi) \ := \ \rq{X}(A,P,\Phi) \quot H.
$$
The action of $H_0 = \Spec \, \CC[K_0]$ leaves 
$\rq{X}(A,P,\Phi)$ invariant and induces an 
action of the torus $T = \Spec \, \CC[\ZZ^s]$ 
on $X(A,P,\Phi)$.
\end{construction}

Recall from~\cite[Thm.~3.4.3.7]{ArDeHaLa}
that the resulting variety $X = X(A,P,\Phi)$ is 
rational, normal, admits only constant invertible 
functions and is of dimension $n+m+1-r-\dim(K_\QQ) = s+1$.
Moreover, the divisor class group of $X$ is 
isomorphic to $K$ and the Cox ring to 
$R(A,P)$.

\begin{remark}
In the important cases of affine or Fano varieties $X$, 
one may evade using the bunch of cones $\Phi$ due to 
the following observations:
\begin{enumerate}
\item
If $X$ is affine, then $\rq{X}(A,P,\Phi)  = \b{X}$ 
holds and we simply have $X = \b{X} \quot H$;
see also Proposition~\ref{prop:affchar} and the 
discussion thereafter.
\item
If $X$ is a Fano variety, then $\rq{X}(A,P,\Phi)$ 
equals the set of semistable points defined 
by the anticanonical class in the character group 
$K = \Cl(X)$ of $H$. 
\end{enumerate}
\end{remark}

The basic result of the approach via the data 
$A$, $P$ and $\Phi$ says that 
if $X$ is a rational, normal variety with a torus 
action of complexity one having only 
constant globally invertible functions and 
satisfies a certain maximality property with 
respect to embeddability into toric varieties,
then $X$ is equivariantly isomorphic to some 
$X(A,P,\Phi)$, see~\cite[Thm.~1.8]{HaWr}.

Toric embeddability is important in our subsequent 
considerations.
More specifically, there is even a canonical embedding 
$X \to Z$ into a toric variety such that~$X$ inherits 
many geometric properties from $Z$.
The construction makes use of the tropical variety 
of $X$.

\begin{construction}
\label{constr:tropvar}
Let $X = X(A,P,\Phi)$ be obtained from 
Construction~\ref{constr:RAPandPhi}.
The~\emph{tropical variety} of $X$ is 
the fan $\trop(X)$ in $\QQ^{r+s}$
consisting of the cones
$$ 
\lambda_i 
\ := \ 
\cone(v_{i1}) + \lin(e_{r+1}, \ldots, e_{r+s})
\text{ for } 
i \ = \ \iota, \ldots, r,
\qquad
\lambda 
\ := \ 
\lambda_\iota \cap \ldots \cap \lambda_r,
$$
where $v_{ij} \in \ZZ^{r+s}$ denote the first $n$ 
columns of $P$ and $e_k \in \ZZ^{r+s}$ the $k$-th 
canonical basis vector;
we call $\lambda_i$ a \emph{leaf} and $\lambda$ 
the \emph{lineality part} of $\trop(X)$. 

\begin{center}

\ \hfill
\begin{tikzpicture}[scale=0.4]
\draw[thick, draw=black, fill=gray!30!] (0,2) -- (2.4,3) -- (2.4,-1) -- (0,-2) -- cycle; 
\draw[thick, draw=black, fill=gray!60!, fill opacity=0.90] (0,2) -- (2.65,1) -- (2.65,-3) -- (0,-2) -- cycle; 
\node at (1.325,-4) {\tiny Type~1};
\end{tikzpicture}
\hfill 
\begin{tikzpicture}[scale=0.4]
\draw[thick, draw=black, fill=gray!90!] (0,2) -- (-3,2) -- (-3,-2) -- (0,-2) -- cycle; 
\draw[thick, draw=black, fill=gray!30!] (0,2) -- (2.4,3) -- (2.4,-1) -- (0,-2) -- cycle; 
\draw[thick, draw=black, fill=gray!60!, fill opacity=0.90] (0,2) -- (2.65,1) -- (2.65,-3) -- (0,-2) -- cycle; 
\node at (0,-4) {\tiny Type~2};
\end{tikzpicture}
\hfill \

\end{center}
\end{construction}

\begin{construction}
\label{constr:minamb}
Let $X = X(A,P,\Phi)$ be obtained from 
Construction~\ref{constr:RAPandPhi}.
For a face $\delta_0 \preceq \delta$ 
of the orthant $\delta \subseteq \QQ^{n+m}$,
let $\delta_0^* \preceq \delta$ denote the 
complementary face
and call $\delta_0$ \emph{relevant} if
\begin{itemize}
\item
the relative interior of $P(\delta_0)$ intersects $\trop(X)$, 
\item
the image $Q(\delta_0^*)$ comprises a cone of $\Phi$,
\end{itemize}
where  
$Q \colon \ZZ^{n+m} \to K = \ZZ^{n+m}/P^*(\ZZ^{r+s})$ 
is the projection. 
Then we obtain fans $\rq{\Sigma}$ in $\ZZ^{n+m}$ 
and~$\Sigma$ in $\ZZ^{r+s}$ of pointed cones 
by setting
$$ 
\rq{\Sigma}
\ := \ 
\{
\delta_1 \preceq \delta_0; \; 
\delta_0 \preceq \delta \text{ relevant} 
\},
\qquad
\Sigma 
\ := \ 
\{
\sigma \preceq P(\delta_0); \; 
\delta_0 \preceq \delta \text{ relevant} 
\}.
$$
The toric varieties $\rq{Z}$ and $Z$ 
associated with $\rq{\Sigma}$ and $\Sigma$,
respectively, and $\b{Z} = \CC^{n+m}$ 
fit into a commutative diagramm of 
characteristic spaces and total coordinate
spaces
$$ 
\xymatrix{
{\b{X}(A,P)}
\ar@{}[r]|{\qquad\subseteq}
\ar@{}[d]|{\rotatebox[origin=c]{90}{$\scriptstyle{\subseteq}$}}
&
{\b{Z}}
\ar@{}[d]|{\rotatebox[origin=c]{90}{$\scriptstyle{\subseteq}$}}
\\
{\rq{X}(A,P,\Phi)}
\ar@{}[r]|{\qquad\subseteq}
\ar[d]_{\quot H}
&
{\rq{Z}}
\ar[d]^{\quot H}
\\
X(A,P,\Phi)
\ar@{}[r]|{\qquad\subseteq}
&
Z
}
$$
The horizontal inclusions are $T$-equivariant closed 
embeddings, where $T$ acts on $Z$ 
as the subtorus of the $(r+s)$-torus 
corresponding to $0 \times \ZZ^{s} \subseteq  \ZZ^{r+s}$.
Moreover, $X(A,P,\Phi)$ intersects every closed toric 
orbit of $Z$.
\end{construction}

We call $Z$ from Construction~\ref{constr:minamb} 
the \emph{minimal toric ambient variety}
of $X = X(A,P,\Phi)$.
Observe that the rays of the fan $\Sigma$ of $Z$ 
have precisely the columns of the matrix $P$ as 
its primitive generators. 
In particular, every ray of $\Sigma$ lies on 
the tropical variety $\trop(X)$.
The minimal toric ambient variety is crucial for 
the resolution of singularities.
The following recipe for resolving singularities
directly generalizes~\cite[Thm.~3.4.4.9]{ArDeHaLa};
a related approach using polyhedral divisors is 
presented in~\cite{LiSu}.

\begin{construction}
\label{constr:ressing}
Let $X = X(A,P,\Phi)$ be obtained from 
Construction~\ref{constr:RAPandPhi}
and consider the canonical toric embedding
$X \subseteq Z$ and the defining fan $\Sigma$ 
of $Z$.
\begin{itemize}
\item
Let $\Sigma'  = \Sigma \sqcap \trop(X)$ be the 
coarsest common refinement.
\item
Let $\Sigma''$ be any regular subdivision of the 
fan $\Sigma'$.
\end{itemize}
Then $\Sigma'' \to \Sigma$ defines a proper toric 
morphism $Z'' \to Z$ and with the proper transform 
$X'' \subseteq Z''$ of $X \subseteq Z$, the morphism
$X'' \to X$ is a resolution of singularities.
\end{construction}

\begin{remark}
In the setting of Construction~\ref{constr:ressing},
the variety $X''$ has again a torus action 
of complexity one and thus is of the form 
$X'' = X(A'',P'',\Phi'')$.
We have $A'' = A$ and $P''$ is obtained from $P$ 
by inserting the primitive generators of $\Sigma''$ 
as new columns.
Moreover, $\Phi''$ is the Gale dual of $\Sigma''$, 
that means that with the corresponding projection $Q''$ 
and orthant $\delta''$ we have
$$ 
\Phi'' 
\ = \ 
\{
Q''(\delta_0^*); \; \delta_0 \preceq \delta''; 
\; P''(\delta_0) \in \Sigma''
\}.
$$
\end{remark}

\begin{proposition}
\label{prop:affchar}
Consider a variety $X = X(A,P,\Phi)$ of Type~2
as provided by Construction~\ref{constr:RAPandPhi}.
Then the following statements are equivalent.
\begin{enumerate}
\item
One has $\rq{X} = \b{X}$
\item
The variety $X$ is affine. 
\item
The minimal toric ambient variety $Z$ of $X$ is affine.
\item
One has $\rq{Z} = \b{Z} = \CC^{n+m}$.
\end{enumerate}
If one of these statements holds, 
then the columns of $P$ generate the extremal
rays of a full-dimensional cone 
$\sigma \subseteq \QQ^{r+s}$ and we have
$Z = \Spec \, \CC[\sigma^\vee \cap \ZZ^{r+s}]$.
\end{proposition}

\begin{proof}
Only for the implication ``(ii)$\Rightarrow$(iii)''
there is something to show.
As $X$ is of Type~2, we have 
$0 \in \b{X} \subseteq \b{Z} = \CC^{n+m}$.
Since $X$ is affine, we have
$\b{X} = \rq{X}$ and thus 
$0 \in \rq{Z}$. 
We conclude $\rq{Z} = \b{Z}$ and thus
$Z = \b{Z} \quot H$ is affine.
\end{proof}

The characterization~\ref{prop:affchar}~(i)
allows us to omit the bunch of cones $\Phi$ in 
the affine case:
we may just speak of the affine variety 
$X = X(A,P) := \b{X} \quot H$.

\begin{corollary}
\label{affQfact}
Let $X = X(A,P)$ be affine of Type~2.
Then the following statements are equivalent.
\begin{enumerate}
\item
The variety $X$ is $\QQ$-factorial.
\item
The variety $Z$ is $\QQ$-factorial.
\item
The columns of $P$ are linearly independent.
\end{enumerate}
\end{corollary}

\begin{proof}
The equivalence of~(i) and~(ii) 
is~\cite[Cor.~3.3.1.7]{ArDeHaLa},
The equivalence of~(ii) and~(iii) 
is~\cite[Thm.~3.1.19~(b)]{CLS}.
\end{proof}

\begin{corollary}
\label{cor:affpic}
Let $X = X(A,P)$ be affine of Type~2.
Then the Picard group of~$X$ is trivial.
\end{corollary}

\begin{proof}
Proposition~\ref{prop:affchar} says that 
the minimal toric ambient variety $Z$ is affine.
Thus, $Z$ has trivial Picard group;
see~\cite[Prop.~4.2.2]{CLS}.
According to~\cite[Cor.~3.3.1.12]{ArDeHaLa},
the Picard group of $X$ equals that of $Z$.
\end{proof}

More generally one can show that in fact 
every normal affine variety admitting a torus 
action with an attractive orbit has trivial 
Picard group: every bundle can be linearized 
and the non-vanishing loci of its homogeneous 
sections form an invariant trivializing open 
cover. As one of these covering sets contains 
the attractive fixed point, the bundle is 
trivial.

\section{The anticanonical complex and singularities}
\label{sec:acancomp}

First recall the basic singularity types arising 
in the minimal model programme.
Let $X$ be a $\QQ$-Gorenstein variety, i.e., some 
non-zero multiple of a canonical 
divisor $D_X$ on $X$ is an integral Cartier divisor.
Then, for any resolution of singularities 
$\varphi\colon X' \rightarrow X$, one has the 
ramification formula
$$
D_{X'} - \varphi^*(D_X) 
\ = \ 
\sum  a_i E_i,
$$
where the $E_i$ are the prime components of 
the exceptional divisors and the coefficients 
$a_i \in \QQ$ are the discrepancies of the 
resolution. 
The variety $X$ is said to have at most 
\emph{log terminal (canonical, terminal)} 
singularities, if for every resolution of 
singularities the discrepancies $a_i$ satisfy 
$a_i > -1$ ($a_i \geq 0$,  $a_i > 0$).

In~\cite{BeHaHuNi}, the ``anticanonical complex'' 
has been introduced for Fano varieties $X(A,P,\Phi)$
and served as a tool to study singularities of the 
above type. 
The purpose of this section is to extend this 
approach and to generalize results from~\cite{BeHaHuNi}
to the non-complete and non-$\QQ$-factorial 
cases.
As an application, we characterize log terminality
in Theorem~\ref{thm:logtermchar} via platonic 
triples occuring in the Cox ring. 
For the affine case, the result specializes to 
Theorem~\ref{thm:affltcharintro}.

Now, let $X = X(A,P,\Phi)$ be a rational $T$-variety 
of complexity one arising from 
Construction~\ref{constr:RAPandPhi}.
Consider the embedding $X \subseteq Z$ into the 
minimal toric ambient variety.
Then $X$ and $Z$ share the same divisor class group 
$$
K 
\ = \ 
\Cl(X) 
\ = \ 
\Cl(Z)
$$ 
and the same degree map $Q \colon \ZZ^{n+m} \to K$ 
for their Cox rings. 
Let $e_Z \in \ZZ^{n+m}$ denote the sum over 
the canonical basis vectors $e_{ij}$ and $e_k$ 
of $\ZZ^{n+m}$.
Then, with the defining relations $g_\iota, \ldots, g_{r-2}$ 
of the Cox ring $R(A,P)$, the canonical divisor 
classes of $Z$ and $X$ are given as
$$
\mathcal{K}_Z
\ = \ 
-Q(e_Z)
\ \in \ 
K,
\qquad\qquad
\mathcal{K}_{X}
\ = \ 
\sum_{i=\iota}^{r-2+\iota} \deg(g_i) + \mathcal{K}_{Z}
\ \in \ 
K,
$$
see~\cite[Prop.~3.3.3.2]{ArDeHaLa}.
Observe that if $X$ is of Type~$1$, then its canonical 
divisor class equals that of the minimal toric 
ambient variety~$Z$.
Define a (rational) polyhedron 
$$ 
B(-\mathcal{K}_X)
\ := \ 
Q^{-1}(-\mathcal{K}_X) \cap \QQ^{n+m}_{\ge 0}
\ \subseteq \ 
\QQ^{n+m}
$$
and let $B := B(g_\iota) + \ldots + B(g_{r-2+\iota})
\subseteq \QQ^{n+m}$ denote the 
Minkowski sum of the Newton polytopes $B(g_i)$
of the relations $g_\iota, \ldots, g_{r-2+\iota}$ of $R(A,P)$.

\begin{definition}
Let $X = X(A,P,\Phi)$ such that $-\mathcal{K}_X$ 
is ample and denote by $\Sigma$ the 
fan of the minimal toric ambient variety $Z$ of $X$.
\begin{enumerate}
\item
The \emph{anticanonical polyhedron} of $X$ 
is the dual polyhedron $A_X \subseteq \QQ^{r+s}$ 
of the polyhedron
$$
B_X
\ := \ 
(P^*)^{-1}(B(-\mathcal{K}_X) + B - e_{\Sigma}) 
\ \subseteq \ 
\QQ^{r+s}.
$$
\item
The \emph{anticanonical complex} of $X$ 
is the coarsest common refinement of polyhedral 
complexes
$$ 
A^c_X
\ := \ 
{\rm faces}(A_X) \sqcap \Sigma \sqcap \trop(X).
$$
\item
The \emph{relative interior} of $A_X^c$ 
is the interior of its support with respect 
to the intersection $\Supp(\Sigma) \cap \trop(X)$.
\item
The \emph{relative boundary} $\partial A_X^c$ 
is the complement of the relative interior of $A_X^c$ 
in $A_X^c$.
\end{enumerate}
\end{definition}

A first statement expresses the discrepancies of a given 
resolution of singularities via the anticanonical complex; 
the proof is a straightforward generalization 
of the one given in~\cite{BeHaHuNi} for the Fano case and 
will be made available elsewhere.

\begin{proposition}
\label{lemm::discrepancy}
Let $X = X(A,P,\Phi)$ such that 
$-\mathcal{K}_X$ is ample 
and $X'' \rightarrow X$ a resolution of 
singularities as in Construction~\ref{constr:ressing}.
For any ray $\varrho \in \Sigma''$, let 
$v_{\varrho}$ be its primitive generator, 
$v_{\varrho}'$ its leaving point of $A_X^c$
provided $\varrho \not \subseteq A_X^c$
and $D_{\varrho}$ the prime 
divisor on $X''$ obtained by intersecting 
$X''$ with the toric divisor of $Z''$ 
corresponding to $\varrho$.
Then the discrepancy $a_{\varrho}$ along 
$D_{\varrho}$ satisfies 
$$
a_{\varrho} 
\ = \ 
-1 + \frac{||v_{\varrho}||}{||v'_{\varrho}||}
\quad \text{if }
\varrho \not \subseteq A_X^c, 
\qquad  \qquad  
a_{\varrho} \leq -1 
\quad \text{if } \varrho \subseteq A_X^c.
$$
\end{proposition}

The next result characterizes the existence of at most 
log terminal (canonical, terminal) singularities in terms
of the anticanonical complex; again, this generalizes a result 
from~\cite{BeHaHuNi} and the proof will be made available 
elsewhere.

\begin{theorem}
\label{theo::singComplex}
Let $X = X(A,P,\Phi)$ be 
such that $-\mathcal{K}_X$ is ample.
Then the following statements hold.
\begin{enumerate}
\item
$A_X^c$ contains the origin in its relative
interior and all primitive generators of 
the fan $\Sigma$ are vertices of $A_X^c$.
\item 
$X$ has at most log terminal singularities 
if and only if the anticanonical complex 
$A_X^c$ is bounded.
\item
$X$ has at most canonical singularities
if and only if $0$ is the only 
lattice point in the relative interior of 
$A_X^c$.
\item
$X$ has at most terminal singularities
if and only if $0$ and the primitive generators 
$v_\varrho$ for $\varrho \in \Sigma^{(1)}$
are the only lattice points of $A_X^c$.
\end{enumerate}
\end{theorem}

We describe the structure of the anticanonical 
complex in more detail, which generalizes
in particular statements on the $\QQ$-factorial 
Fano case obtained in~\cite{BeHaHuNi}. 
For Type~1, the situation turns out to be simple,
whereas Type~2 is more involved.

\begin{proposition}
\label{prop:acancompstruct-type1}
Let $X = X(A,P,\Phi)$ be of Type~1 such that $-\mathcal{K}_X$ 
is ample.
Let~$\Sigma$ be the fan of the minimal toric ambient
variety of $X$ and
denote by $\lambda_0,\ldots, \lambda_r$ the leaves 
of $\trop(X)$.
\begin{enumerate}
\item
Every cone $\sigma \in \Sigma$ is contained 
in a leaf $\lambda_i \subseteq \trop(X)$.
In particular, $\Sigma \sqcap \trop(X)$
equals $\Sigma$.
\item 
The boundary of $A_X^c$ is the union of all 
faces of $A_X$ that are contained in 
$\Supp(\Sigma)$.
\item
The non-zero vertices of $A_X^c$ are the primitive 
generators of $\Sigma$, i.e.~the columns of~$P$. 
\end{enumerate}
\end{proposition}

\begin{corollary}
\label{cor:sing-type1}
Let $X = X(A,P,\Phi)$ be a $T$-variety of Type~1.
Then $X$ has at most log-terminal singularities.
Moreover, it has at most canonical (terminal)
singularities if and only if its minimal toric 
ambient variety $Z$ does so.
\end{corollary}

\begin{construction}
Let $X = X(A,P,\Phi)$ be of Type~2 and 
$\Sigma$ the fan of the minimal toric ambient
variety of $Z$.
Write $v_{ij} := P(e_{ij})$ and $v_k := P(e_k)$ 
for the columns of~$P$.
Consider a pointed cone of the form
$$ 
\tau 
\ = \ 
\cone(v_{0j_0}, \ldots, v_{rj_r})
\ \subseteq \ 
\QQ^{r+s},
$$
that means that $\tau$ contains exactly one 
$v_{ij}$ for every $i = 0,\ldots,r$.
We call such $\tau$ a \emph{$P$-elementary cone}
and associate the following numbers with $\tau$:
$$
\ell_{\tau,i} 
\ := \ 
\frac{l_{0j_0} \cdots l_{rj_r}}{l_{ij_i}}
\text{ for } i = 0, \ldots, r,
\qquad
\ell_{\tau}
\ := \ 
(1-r) l_{0j_0} \cdots l_{rj_r} + \sum_{i=0}^r \ell_{\tau, i}.
$$
Moreover, we set 
$$ 
v(\tau) 
\ := \ 
\ell_{\tau,0} v_{0j_0} + \ldots +  \ell_{\tau,r} v_{rj_r}
\ \in \ 
\ZZ^{r+s},
\qquad
\varrho(\tau) 
\ := \ 
\QQ_{\ge 0} \cdot v(\tau)
\ \in \ 
\QQ^{r+s}.
$$
We denote by $\Tau(A,P,\Phi)$ the set of all 
$P$-elementary cones $\tau \in \Sigma$.
For a given $\sigma \in \Sigma$, we denote by 
$\Tau(\sigma)$ the set of all $P$-elementary faces
of~$\sigma$.
\end{construction}

\begin{remark}
Let $X = X(A,P,\Phi)$ be of Type~2.
Let $\Sigma$ be the fan of the minimal toric ambient
variety of $X$
and $\lambda_0, \ldots, \lambda_r \subseteq \trop(X)$
the leaves of the tropical variety of $X$.
As in~\cite[Def.~4.1]{BeHaHuNi}, we say that
\begin{enumerate}
\item
a cone $\sigma \in \Sigma$ is a 
\emph{leaf cone} if $\sigma \subseteq \lambda_i$ 
holds for some $i = 0, \ldots, r$,
\item
a cone $\sigma \in \Sigma$ is called \emph{big} 
if $\sigma \cap \lambda_i^\circ \ne \emptyset$ 
holds for all $i = 0, \ldots, r$.
\end{enumerate}
Observe that a given cone $\sigma \in \Sigma$ is big if and only 
if $\sigma$ contains some $P$-elementary cone as a subset.
\end{remark}

\begin{proposition}
\label{prop:acancompstruct}
Let $X = X(A,P,\Phi)$ be of Type~2 such that $-\mathcal{K}_X$ 
is ample.
Let~$\Sigma$ be the fan of the minimal toric ambient
variety of $X$,
denote by $\lambda_0,\ldots, \lambda_r$ the leaves 
of $\trop(X)$ and by 
$\lambda = \lambda_0 \cap \ldots \cap \lambda_r$
its lineality part.
\begin{enumerate}
\item
The fan $\Sigma \sqcap \trop(X)$ consists of 
the cones $\sigma \cap \lambda$ and 
$\sigma \cap \lambda_i$, where $\sigma \in \Sigma$
and $i = 0,\ldots,r$.
Here, one always has
$\sigma \cap \lambda \preceq \sigma \cap \lambda_i$.
\item 
The fan $\Sigma \sqcap \trop(X)$ is a subfan of the 
normal fan of the polyhedron $B_X$.
In particular, for every cone $\sigma \cap \lambda_i$, 
there is a vertex $u_{\sigma,i} \in B_X$ with 
$$ 
\partial A_X^c \cap \sigma \cap \lambda_i
\ = \ 
\{v \in \sigma \cap \lambda_i; \; \bangle{u_{\sigma,i},v} = -1\}.
$$
\item
If a $P$-elementary cone $\tau$ is contained in some 
$\sigma \in \Sigma$, 
then $\tau$ is simplicial, $v(\tau) \in \tau^\circ$ holds,
$\varrho(\tau)$ is a ray,
$\varrho(\tau) = \tau \cap \lambda$ holds
as well as $\QQ \varrho(\tau) = \QQ\tau \cap \lambda$.
\item
Let $\sigma \in \Sigma$ be any cone.
Then, for every $i = 0, \ldots, r$, 
the set of extremal rays of
$\sigma \cap \lambda_i \in \Sigma \sqcap \trop(X)$
is given by
$$ 
\qquad\qquad
(\sigma \cap \lambda_i)^{(1)}
\ = \
\{\varrho(\sigma_0); \ \sigma_0 \in \Tau(\sigma) \}
\cup 
\{\varrho \in \sigma^{(1)}; \ \varrho \subseteq \lambda_i\}.
$$
\item
The set of rays of $\Sigma \sqcap \trop(X)$ consists of the 
rays of $\Sigma$ and the rays $\varrho(\sigma_0)$, 
where $\sigma_0 \in \Tau(A,P,\Phi)$. 
\item 
If a $P$-elementary cone $\tau$ is contained in some 
$\sigma \in \Sigma$, then the minimum value 
among all $\bangle{u,v(\tau)}$, where $u \in B_X$, equals 
$-\ell_\tau$.
\item
Let the $P$-elementary cone $\tau$ be contained in 
$\sigma \in \Sigma$.
Then $\varrho(\tau) \not\subseteq A_X^c$ holds 
if and only if $\ell_\tau > 0$ holds; 
in this case, $\varrho(\tau)$ leaves $A_X^c$ at
$v(\tau)' = \ell_{\tau}^{-1} v(\tau)$.
\item
The vertices of $A_X^c$ are the primitive generators of 
$\Sigma$, i.e.~the columns of~$P$, and the points
$v(\sigma_0)' = \ell_{\sigma_0}^{-1} v(\sigma_0)$, where 
$\sigma_0 \in \Tau(A,P,\Phi)$ and $\ell_{\sigma_0} > 0$. 
\end{enumerate}
\end{proposition}

\begin{proof}
Assertion~(i) holds more generally.
Indeed, the coarsest common refinement 
$\Sigma_1 \sqcap \Sigma_2$
of any two quasifans~$\Sigma_i$ in a common vector space 
consists of the intersections $\sigma_1 \cap \sigma_2$, 
where $\sigma_i \in \Sigma_i$.
Moreover, the faces of a given cone
$\sigma_1 \cap \sigma_2$ of $\Sigma_1 \sqcap \Sigma_2$
are precisely the cones $\sigma_1' \cap \sigma_2'$,
where $\sigma_i' \preceq \sigma_i$.

We show~(ii). Let $\Sigma'$ be the complete fan in 
$\QQ^{r+s}$ defined by the class $-\mathcal{K}_X \in K$.
Since $-\mathcal{K}_X$ is ample, the fan $\Sigma$ 
is a subfan of $\Sigma'$.
The preimage $P^{-1}(\Sigma')$ consists of
the cones $P^{-1}(\sigma')$, where 
$\sigma' \in \Sigma'$, and is the normal fan of 
$B(-\mathcal{K}_X) \subseteq \QQ^{n+m}$.
Moreover, $P^{-1}(\trop(X))$ turns out to be a 
subfan of the normal fan of $B \subseteq \QQ^{n+m}$.
It follows that $P^{-1}(\Sigma') \sqcap P^{-1}(\trop(X))$ 
is a subfan of the normal fan of 
$B(-\mathcal{K}_X) + B$.
Projecting the involved fans via $P$ to $\QQ^{r+s}$ 
gives the assertion. 

To obtain~(iii), consider first any $P$-elementary 
$\tau = \cone(v_{0j_0}, \ldots, v_{rj_r})$.
Then $v_{0j_0}, \ldots, v_{rj_r}$ is linearly 
dependent if and only if $v(\tau) = 0$ holds.
The latter is equivalent to $0$ being an inner point 
of $\tau$. 
Thus, if $\tau$ is contained in some 
$\sigma \in \Sigma$, then $\tau$ is 
pointed an thus must be simplicial.
The remaining part is then obvious; recall 
that the lineality part of $\trop(X)$ equals the vector 
subspace $0 \times \QQ^s \subseteq \QQ^{r+s}$.

We turn to~(iv). 
First, we claim that if $\sigma_0 \in \Sigma$ 
is big and $\varrho(\tau) = \varrho(\tau')$ 
holds for any two $P$-elementary cones
$\tau, \tau' \subseteq \sigma$,
then $\sigma_0$ is $P$-elementary.
Assume that $\sigma_0$ is not $P$-elementary.
Then we find some $1 \le t \le r$ and cones
$$
\tau 
\ = \ 
\cone(v_{0j_0}, \ldots, v_{tj_{t-1}}, 
v_{tj_t},  
v_{tj_{t+1}}, \ldots , v_{rj_r})
\ \subseteq \ 
\sigma_0,
$$
$$
\tau' 
\ = \ 
\cone(v_{0j_0}, \ldots, v_{tj_{t-1}}, 
v_{tj'_t},  
v_{tj_{t+1}}, \ldots , v_{rj_r})
\ \subseteq \ 
\sigma_0
$$
with $j_t \ne j'_t$ and thus $\tau \neq \tau'$.
Here, we may assume that 
$c_{\tau}^{-1}l_{tj_t} \ge c_{\tau'}^{-1}l_{tj_t'}$ 
holds with the greatest common divisors 
$c_\tau$ and $c_{\tau'}$ of the entries of
$v(\tau)$ and $v(\tau')$ respectively.
Then even 
$c_{\tau}^{-1}\ell_{\tau,i} \ge c_{\tau'}^{-1}\ell_{\tau',i}$ 
must hold for all $1 \le i \le r$.
Since, the rays $\varrho(\tau)$ and $\varrho(\tau')$ 
coincide, also their primitive generators 
$c_{\tau'}^{-1} v(\tau')$ and $c_{\tau}^{-1} v(\tau)$
coincide.
By the definition of $v(\tau)$ and $v(\tau')$, this 
implies
$$
c_{\tau'}^{-1}\ell_{\tau',t} v_{tj_t'}
\ = \ 
c_{\tau}^{-1}\ell_{\tau,k} v_{tj_t}
+
\sum_{i \ne t}
(c_{\tau}^{-1}\ell_{\tau,i} - c_{\tau'}^{-1}\ell_{\tau',i}) v_{ij_i}.
$$
We conclude $v_{tj_t'} \in \tau$. Since $v_{tj_t'}$ 
is an extremal ray of $\sigma_0$ and 
$\tau' \subseteq \sigma_0$ holds, 
$v_{tj_t'}$ generates an extremal ray of $\tau$.
This contradicts to the choice of $j_t'$
and the claim is verified.

Now, consider the equation of~(iv).
To verify ``$\subseteq$'', let $\varrho$ be an 
extremal ray of $\sigma \cap \lambda_i$.
We have to show that $\varrho = \varrho(\sigma_0)$
holds for some $\sigma_0 \in \Tau(\sigma)$ 
or that $\varrho$ is a ray of $\sigma$ with
$\varrho \subseteq \lambda_i$.
According to~(ii), there is a face $\sigma_\varrho \preceq \sigma$ 
such that $\varrho = \sigma_\varrho \cap \lambda$
or $\varrho = \sigma_\varrho \cap \lambda_i$ holds.
We choose $\sigma_\varrho$ minimal with respect to this
property, that means that we have
$\varrho^{\circ} \subseteq \sigma_\varrho^{\circ}$.
We distinguish the following cases.

\medskip

\noindent
\emph{Case 1.} We have $\varrho = \sigma_\varrho \cap \lambda$.
If $\sigma_\varrho \subseteq \lambda$ holds, then we
obtain $\varrho = \sigma_\varrho$ and thus 
$\varrho \subseteq \lambda_i$ is an extremal ray 
of $\sigma$.
So, assume that $\sigma_\varrho$ is not contained in $\lambda$.
Then, because of $\sigma_\varrho^{\circ} \cap \lambda \ne \emptyset$,
there is a $P$-elementary cone $\tau \subseteq \sigma_\varrho$. 
Using~(i), we obtain
$$ 
\varrho(\tau) 
\ = \ 
\tau \cap \lambda 
\ \subseteq \
\sigma_\varrho \cap \lambda
\ = \ 
\varrho
$$
and thus $\varrho = \varrho(\tau)$. 
As this does not depend on the particular choice of 
the $P$-elementary cone $\tau \subseteq \sigma_\varrho$, 
the above claim yields 
$\sigma_0 := \sigma_\varrho \in \Tau(\sigma)$
and $\varrho = \varrho(\sigma_0)$.

\medskip

\noindent
\emph{Case 2.}  
We don't have $\varrho = \sigma_\varrho \cap \lambda$.
Then $\varrho = \sigma_\varrho \cap \lambda_i$ and 
$\varrho^{\circ} \subseteq \lambda_i^{\circ}$ hold.
If $\sigma_\varrho \subseteq \lambda_i$ holds, then we
obtain $\varrho = \sigma_\varrho$ and thus 
$\varrho \subseteq \lambda_i$ is an extremal ray 
of $\sigma$.
So, assume that $\sigma_\varrho$ is not contained 
in $\lambda_i$.
Then $\sigma_\varrho \cap \lambda_j^{\circ}$ is 
non-empty for all $j = 0, \ldots, r$.
Thus, there is a $P$-elementary
cone $\tau \subseteq \sigma_\varrho$. 
Using~(i), we obtain
$$ 
\varrho(\tau) 
\ = \ 
\tau \cap \lambda 
\ \subseteq \
\sigma_\varrho \cap \lambda
\ = \ 
\varrho
$$
and thus $\varrho = \varrho(\tau)$. 
As this does not depend on the particular choice of 
the $P$-elementary cone $\tau \subseteq \sigma_\varrho$, 
the above claim yields 
$\sigma_0 := \sigma_\varrho \in \Tau(\sigma)$
and $\varrho = \varrho(\sigma_0)$.

We verify the inclusion ``$\supseteq$''.
Consider a face $\sigma_0 \in \Tau(\sigma)$.
As seen just before, the extremal rays of 
$\sigma_0 \cap \lambda_i$ are $\varrho(\sigma_0)$ 
and the rays of $\sigma_0$ that lie in $\lambda_i$.
Since $\sigma_0 \cap \lambda_i$ is a face of 
$\sigma \cap \lambda_i$, the ray $\varrho(\sigma_0)$ 
is an extremal ray of $\sigma \cap \lambda_i$.
Finally, consider an extremal ray 
$\varrho \preceq \sigma$ with $\varrho \subseteq \lambda_i$.
Then $\varrho = \varrho \cap \lambda_i$ is a face 
of $\sigma \cap \lambda_i$.

The proof of Assertion~(iv) is complete now. Assertion~(v)  
is a direct consequence of~(iv).

We turn to Assertions~(vi), (vii) and~(viii). 
Let $\rq{\tau} \preceq \rq{\sigma} \preceq \QQ_{\ge 0}^{n+m}$
be the faces with $P(\rq{\tau}) = \tau$ 
and $P(\rq{\sigma}) = \sigma$.
Moreover, let $e_\tau \in \rq{\tau}$ be the (unique) 
point with $P(e_\tau) = v(\tau)$. 
The minimum value $\bangle{u,v(\tau)}$ is attained at some 
vertex $u \in B_X$.
For this $u$, we find vertices 
$e_{\sigma} \in B(-\mathcal{K}_X)$ 
and $e_B \in B$ with
$$ 
u \ = \ (P^*)^{-1}(e_{\sigma} + e_B - e_Z).
$$
Here, $e_{\sigma}$ is any vertex 
of $B(-\mathcal{K}_X)$ such that 
$\rq{\sigma}$ is contained in the cone
of the normal fan of $B(-\mathcal{K}_X)$
associated with $e_\sigma$;
such $e_\sigma$ exists due to ampleness 
of $-\mathcal{K}_X$ 
and~$e_{\sigma}$ vanishes along $\rq{\sigma}$.
Together we have 
$$ 
e_\tau 
\ = \ 
\sum_{i = 0}^r l_{i j_i} e_{ij_i},
\qquad\qquad
\bangle{u,v(\tau)}
\ = \ 
\bangle{e_{\sigma} + e_B - e_Z,e_\tau}.
$$
As mentioned, $\bangle{e_{\sigma},e_{\tau}} = 0$
holds.
Moreover, $\bangle{e,e_\tau} = (r-1)l_{0j_0} \cdots l_{rj_r}$
holds for every $e \in B$.
We conclude $\bangle{u,v(\tau)} = - \ell_\tau$
and Assertion~(vi).
Moreover, Assertions~(vii) and~(viii) are direct consequences 
of~(vi) and~(ii).
\end{proof}

\begin{example}
Consider the $E_6$-singular affine surface 
$X = V(z_1^4 + z_2^3 + z_3^2) \subseteq \CC^3$.
It inherits a $\CC^*$-action from the action
$$
t \cdot (z_1,z_2,z_3)
\ = \
(t^{3}z_1,t^{4}z_2,t^{6}z_3)
$$
on $\CC^3$. The divisor class group and the Cox ring 
of the surface $X$ are explicitly given by 
$$ 
\Cl(X) \ = \ \ZZ/3\ZZ,
\qquad\qquad
\mathcal{R}(X) 
\ = \ 
\CC[T_{1},T_{2},T_{3}]/ \bangle{T_{1}^{3}+T_{2}^{3}+T_{3}^{2}},
$$
where the $\Cl(X)$-degrees of $T_1$, $T_2$, and $T_3$ 
are $\bar 1$, $ \bar 2$ and $\bar 0$.
The minimal toric ambient variety is
affine and corresponds to the cone
$$
\sigma 
\ = \ 
\cone((-3,-3,-2),(3,0,1),(0,2,1)).
$$
Denoting by $e_{i} \in \QQ^{3}$ the $i$-th canonical 
basis vector, the tropical variety $\trop(X)$ 
in $\QQ^{3}$ is given as
$$
\trop(X)
\ = \ 
\cone(e_{1},\pm e_{3})
\cup 
\cone(e_{2},\pm e_{3})
\cup 
\cone(-e_{1}-e_{2}, \pm e_{3}).
$$
The anticanonical polyhedron $A_X \subseteq \QQ^3$ is non bounded 
with recession cone generated by $(-1,-1,-1)$, $(1,0,0)$, 
$(0,1,0)$. The vertices  of $A_X$ are
$$
(-3,-3,-2),\
(3,0,1),\
(0,2,1),\
(0,0,1).
$$
The anticanonical complex $A_{X}^{c} = A_{X} \sqcap \Sigma \sqcap \trop(X)$ 
lives inside $\trop(X)$ and looks as follows.%
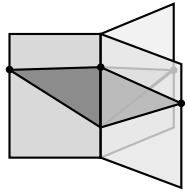
\begin{center}
\begin{tikzpicture}[scale=0.4]
\draw[thick, draw=black, fill=gray!30!] 
(0,2.1) -- (-3,2.1) -- (-3,-2) -- (0,-2) -- cycle; 
\draw[thick, draw=black, fill=gray!90!] 
(0,-1) -- (-3,0.91) -- (0,1) -- cycle; 
\draw[black, fill=black] (-3,0.92) circle (.7ex);
\draw[thick, draw=black, fill=gray!10!] 
(0,2.1) -- (2.4,3.1) -- (2.4,-1) -- (0,-2) -- cycle; 
\draw[thick, draw=black, fill=gray!40!] 
(0,-1) -- (2.4,0.9) -- (0,1) -- cycle; 
\draw[thick, draw=black] (0,-1) -- (2.4,0.95);
\draw[black, fill=black] (2.4,0.9) circle (.7ex);
\draw[thick, draw=black, fill=gray!20!, fill opacity=0.80] 
(0,2.1) -- (2.65,1.1) -- (2.65,-3) -- (0,-2) -- cycle; 
\draw[thick, draw=black, fill=gray!60!, fill opacity=0.80] 
(0,-1) -- (2.65,-0.2) -- (0,1) -- cycle; 
\draw[black, fill=black] (2.65,-0.2) circle (.7ex);
\draw[black, fill=black] (0,1) circle (.7ex);
\end{tikzpicture}
\end{center}
\end{example}

\begin{corollary}
\label{coro::disc}
Let $X = X(A,P,\Phi)$ be of Type~2 such that 
$-\mathcal{K}_X$ is ample.
Let $\tau$ be a $P$-elementary cone contained 
in some $\sigma \in \Sigma$.
Assume $\varrho(\tau) \not\subseteq A_X^c$ and denote
by $c_{\tau}$ the greatest common divisor of the 
entries of $v(\tau)$. 
Then, for any resolution of singularities 
$\varphi \colon X'' \to X$ provided by~\ref{constr:ressing}, 
the discrepancy along the prime divisor of $X''$ 
corresponding to $\varrho(\tau)$ equals 
$c_{\tau}^{-1}\ell_{\tau}-1$.
\end{corollary}

\begin{corollary}
\label{cor::sing}
Let $X = X(A,P,\Phi)$ be of Type~2 such that $-\mathcal{K}_X$ is ample
and let $\tau  = \mathrm{cone}(v_{0j_0}, \dots, v_{rj_r})$
be contained in some $\sigma \in \Sigma$.
\begin{enumerate}
\item 
If $X$ has at most log terminal singularities, 
then $l_{0j_0}^{-1} + \ldots + l_{rj_r}^{-1} > r-1$ holds.
\item 
If $X$ has at most canonical singularities, 
then $l_{0j_0}^{-1} + \ldots + l_{rj_r}^{-1} \ge
r-1 + c_\tau l_{0j_0}^{-1}  \cdots  l_{rj_r}^{-1}$ holds.
\item 
If $X$ has at most terminal singularities, 
then $l_{0j_0}^{-1} + \ldots + l_{rj_r}^{-1} >
r-1 + c_\tau l_{0j_0}^{-1}  \cdots  l_{rj_r}^{-1}$ holds.
\end{enumerate}
\end{corollary}

\begin{remark}
Let $a_0 , \ldots, a_r$ be positive integers.
Then $a_0^{-1} + \ldots + a_r^{-1} > r-1$ holds if and 
only if $(a_0, \ldots,a_r)$ is a platonic tuple.
\end{remark}

\begin{theorem}
\label{thm:logtermchar}
Let $X = X(A,P,\Phi)$ be of Type~2 such that 
$-\mathcal{K}_X$ is ample and let~$\Sigma$ be the 
fan of the minimal toric ambient variety of $X$.
Then the following statements are equivalent.
\begin{enumerate}
\item
The variety $X$ has at most log terminal 
singularities.
\item
For every $P$-elementary 
$\tau = \cone(v_{0j_0}, \ldots, v_{rj_r})$
contained in a cone of $\Sigma$, the 
exponents $l_{0j_0}, \ldots, l_{rj_r}$ 
form a platonic tuple.
\end{enumerate}
\end{theorem}

\begin{proof}
Assume that $X = X(A,P,\Phi)$ is log terminal.
Then Corollary~\ref{cor::sing}~(i) tells us 
that for every $P$-elementary 
$\tau = \cone(v_{0j_0}, \ldots, v_{rj_r})$ 
contained in a cone of $\Sigma$, the 
corresponding exponents $l_{0j_0}, \ldots, l_{rj_r}$
form a platonic tuple.

Now assume that~(ii) holds.
Then every $(l_{0j_0}, \ldots, l_{rj_r})$ is a 
platonic tuple.
Consequently, we have $\ell_\tau > 0$ for 
every $P$-elementary cone $\tau$.
Proposition~\ref{prop:acancompstruct} 
shows that $A_X^c$ is bounded for 
$X = X(A,P,\Phi)$.
Theorem~\ref{theo::singComplex}~(ii) tells 
us that $X$ is log terminal. 
\end{proof}

\begin{remark}
\label{rem:affcase}
Let $X = X(A,P,\Phi)$ be affine of Type~2
such that $\mathcal{K}_X$ is $\QQ$-Cartier.
Then $-\mathcal{K}_X$ is ample.
The fan $\Sigma$ of the minimal toric ambient 
variety $Z$ of $X$ consists of all 
the faces of the cone $\sigma$ generated 
by the columns of $P$.  
In particular, every $P$-elementary cone
is contained in $\sigma$.
Thus, Theorem~\ref{thm:affltcharintro}
follows from Theorem~\ref{thm:logtermchar}.
Moreover, the rays $\varrho(\sigma_0)$, where 
$\sigma_0 \in \Tau(A,P,\Phi)$, are precisely the 
extremal rays of the intersection of $\sigma$
and the lineality part of $\trop(X)$.
\end{remark}

\section{Gorenstein index and canonical multiplicity}
\label{sec:gorenstein}

If a normal variety $X$ is $\QQ$-Gorenstein, then,
by definition, some multiple of its canonical class 
$\mathcal{K}_X$ is Cartier. 
The \emph{Gorenstein index} of $X$ is the smallest positive
integer $\imath_X$ such that $\imath_X \mathcal{K}_X$ is
Cartier.
We attach another invariant to the canonical divisor 
of $X$.

\begin{remark}
\label{rem:etaX}
Let  $X = X(A,P)$ be a $\QQ$-Gorenstein, 
affine $T$-variety of Type~2.
We consider canonical divisors $D_X$ 
on $X$ that are of the 
following form, cf.~\cite[Prop.~3.3.3.2]{ArDeHaLa}:
\begin{equation}
\label{eqn:simplbdry}
- \sum_{i,j} D_{ij} - \sum_k E_k 
+
\sum_{\alpha = 1}^{r-1} \sum_{j = 0}^{n_{i_\alpha}}  l_{i_\alpha j} D_{i_\alpha j},
\qquad 
0 \le i_\alpha \le r.
\end{equation}
Corollary~\ref{cor:affpic} says that $\imath_X D_X$ 
is the divisor of a $T$-homogeneous rational function.
Any two $\imath_X D_X$ with $D_X$ of shape~(\ref{eqn:simplbdry})
differ by the divisor of a $T$-invariant rational function,
and thus, all the functions with divsors $\imath_X D_X$,
where $D_X$ as in~(\ref{eqn:simplbdry}), are homogeneous with 
respect to the same weight $\eta_X \in \Chi(T)$. 
\end{remark}

\begin{definition}
\label{def:canonmult}
Let $X = X(A,P)$ be a $\QQ$-Gorenstein, affine 
$T$-variety of Type~2. 
We call $\eta_X \in \XX(T)$ of Remark~\ref{rem:etaX}
the \emph{canonical weight} of $X$. 
The \emph{canonical multiplicity} of $X$ is the minimal 
non-negative integer $\zeta_X$ such that 
$\eta_X = \zeta_X \cdot \eta_X'$ holds with a primitive 
element $\eta_X' \in \Chi(T)$. 
\end{definition}

\begin{proposition}
\label{prop:zeta}
Let $X = X(A,P)$ be a $\QQ$-Gorenstein, affine 
$T$-variety of Type~2 with at most log terminal 
singularities. 
Then $\zeta_X > 0$ holds.
Moreover, for any positive integer $\imath$, 
the following statements are equivalent.
\begin{enumerate}
\item 
The variety $X$ is of Gorenstein index $\imath$.
\item
There exist integers $\mu_1,\ldots,\mu_r$ with 
$\gcd(\mu_1,\ldots,\mu_r,\zeta_X,\imath) = 1$ such that
with $\mu_0 := \imath(r-1) - \mu_1 - \ldots - \mu_r$
we obtain integral vectors
$$ 
\nu_i := (\nu_{i1}, \ldots, \nu_{in_i})
\text{ with }
\nu_{ij}
\ := \ 
\frac{\imath - \mu_il_{ij}}{\zeta_X},
$$
$$
\nu' := (\nu'_1, \ldots, \nu'_m)
\text{ with }
\nu'_k
\ := \ 
\frac{\imath}{\zeta_X}
$$
and by suitable elementary row operations on 
the $(d,d')$-block, the matrix~$P$ gains 
$(\nu_0, \ldots, \nu_r,\nu')$ as its last row, 
i.e., turns into the shape
$$ 
\tilde P
\ = \ 
\begin{pmatrix}
-l_{0} & l_{1}  & \ldots & 0 & 0
\\
\vdots & \vdots & \ddots & \vdots & \vdots
\\
-l_{0} & 0 & \ldots &  l_r & 0
\\
* & * &  \ldots & * & *
\\
\nu_0 & \nu_1 &  \ldots & \nu_r & \nu'
\end{pmatrix}.
$$
\end{enumerate}
\end{proposition}

\begin{proof}
We work with an anticanonical divisor $D_X$ on $X$ such 
that $-D_X$ is of the form~(\ref{eqn:simplbdry}):
$$ 
D_X
\ := \ 
\sum_{i,j} D_{ij} + \sum_k E_k - (r-1)\sum_{j = 1}^{n_0} l_{0j}D_{0j}.
$$ 
According to Corollary~\ref{cor:affpic}, the Picard 
group of $X$ is trivial.
Thus, $\imath_X D_X$ is the divisor of some toric 
character $\chi^u$, where 
$$
u 
\ = \ 
(\mu_1,\ldots,\mu_r,\eta_1,\ldots,\eta_s)
\ \in \ 
\ZZ^{r+s}.
$$
Note that $-(\eta_1,\ldots,\eta_s) \in \ZZ^s = \Chi(T)$ 
is the canonical weight $\eta_X$ of $X$.
Moreover, the divisor 
$\imath_X D_X = \div(\chi^{u})$ corresponds to 
the vector $P^* \cdot u \in \ZZ^{m+n}$ 
under the identification of toric divisors
with lattice points via $D_{ij} \mapsto e_{ij}$
and $E_k \mapsto e_k$. 

We claim that $\eta_X$ is non-trivial.
Otherwise, $\eta_1 = \ldots = \eta_s = 0$ holds.
As noted, the $ij$-th and $k$-th components 
of the vector $P^* \cdot u$ are the multiplicities 
of $D_{ij}$ and~$D_k$ in~$\imath_X D_X$, respectively.
More explicitly, this leads to the conditions
$$ 
m \ = \ 0,
\qquad
\imath_X((r-1) l_{0j} -1) 
\ = \ 
(\mu_1 + \ldots + \mu_r)l_{0j},
\qquad
\imath_X 
\ = \ 
\mu_i l_{ij}
$$
for all $i$ and $j$. 
Plugging the third into the second one, we obtain 
that $l_{0j_0}^{-1} + \ldots + l_{rj_r}^{-1}$ 
equals $r-1$ for any choice of $1 \le j_i \le n_i$.
According to Corollary~\ref{cor::sing}~(i), 
this contradicts to log terminality of $X$.
Knowing that $\eta_X$ is non-zero, we obtain 
that~$\zeta_X$ is non-zero.

Now, assume that~(i) holds, i.e., we have 
$\imath = \imath_X$.
Let $u \in \ZZ^{r+s}$ as above.
Then we have $\zeta_X = \gcd(\eta_1,\ldots,\eta_s)$
and $\div(\chi^{u}) = \imath D_X$ implies
$\gcd(\mu_1,\ldots,\mu_r,\zeta_X,\imath) = 1$. 
Next, choose a unimodular $s \times s$ matrix $\mathcal{B}$ 
with $\mathcal{B}^{-1} \cdot (\eta_1,\ldots,\eta_s) = (0,\ldots,0,\zeta_X)$. 
Consider $\tilde P := \mathrm{diag}(E_r,\mathcal{B}^*) \cdot P$ and
$$
\tilde u 
\ = \ 
(\mu_1,\ldots,\mu_r,0,\ldots,0,\zeta_X)
\ \in \ 
\ZZ^{r+s}.
$$
Observe that we have 
$P^* \cdot u = \tilde P^* \cdot \tilde u$.
Comparing the entries of $\tilde P^* \cdot \tilde u$
with the multiplicities of the prime divisors 
$D_{ij}$ and $D_k$ 
in~$\imath D_X$ shows that the last row of~$\tilde P$ is 
as claimed.

Conversely, if~(ii) holds, consider
$u := (\mu_1,\ldots,\mu_r,0, \ldots, 0,\zeta_X)$.
Then we obtain
$\imath D_X = \div(\chi^u)$.
Using $\gcd(\mu_1,\ldots,\mu_r,\zeta_X,\imath)=1$,
we conclude that $\imath$ is the Gorenstein index 
of $X$.
\end{proof}

\begin{remark}
\label{rem:computezeta}
Let $X = X(A,P)$ be a $\QQ$-Gorenstein, affine 
$T$-variety of Type~2 
and~$D_X$ a canonical divisor on $X$ 
as in~(\ref{eqn:simplbdry}). 
Then $\imath_X D_X$ is the divisor of some toric 
character $\chi^u$, where 
$$
u 
\ = \ 
(\mu_1,\ldots,\mu_r,\eta_1,\ldots,\eta_s)
\ \in \ 
\ZZ^{r+s}.
$$
In this situation, 
we have $\eta_X = (\eta_1,\ldots,\eta_s) \in \Chi(T)$
for the canonical weight of $X$ and the canonical multiplicity
of $X$ is given by  $\zeta_X = \gcd(\eta_1,\ldots,\eta_s)$.
If $P$ is in the shape 
of Proposition~\ref{prop:zeta},
then $\eta_X = (0,\ldots,0,\zeta_X)$ holds
and $-\mu_1,\ldots,-\mu_r$ satisfy the conditions
of~\ref{prop:zeta}~(ii).
\end{remark}

\begin{remark}
\label{rem:zetaPprops}
The defining matrix $P$ of a given $\QQ$-Gorenstein,
affine $T$-variety $X = X(A,P)$ is in the shape 
of Proposition~\ref{prop:zeta} 
if and only if for every $i = 0, \ldots, r$,
the numbers 
$\mu_i := (\imath_X - \zeta_X \nu_{i1})l_{i1}^{-1}$
satisfy
\begin{enumerate}
\item
$\zeta_X \nu_{ij} + \mu_il_{ij} = \imath_X$ 
for $i = 1, \ldots, r$ and $j = 1, \ldots, n_i$,
\item
$\zeta_X\nu_{0j} + \mu_0l_{0j} = \imath_X$,
for 
$\mu_0 := \imath_X(r-1) - \mu_1 - \ldots - \mu_r$ 
and 
$j = 1, \ldots, n_0$,
\item
$\gcd(\mu_1,\ldots,\mu_r,\zeta_X,\imath_X) = 1$,
\item
$\zeta_X\nu_k' = \imath_X$ for $k = 1, \ldots, m$.
\end{enumerate}
\end{remark}

\begin{corollary}
\label{cor::Gore}
Let $X = X(A,P)$ be a $\QQ$-Gorenstein, affine 
$T$-variety  of Type~$2$ 
with at most log terminal singularities.
Then, for every $\imath \in \ZZ_{\ge 1}$,
the following statements are equivalent.
\begin{enumerate}
\item
The variety $X$ is of Gorenstein index $\imath$ and 
of canonical multiplicity one.
\item
One can choose the defining matrix $P$ to be of 
the shape 
$$
\begin{pmatrix}
-l_{0} & l_{1}  & \dots & 0 & 0
\\
\vdots & \vdots & \ddots & \vdots & \vdots
\\
-l_{0} & 0 & \dots &  l_r & 0
\\
* & * & \ldots & * & *
\\
\bm{\imath}-\imath(r-1)l_0 & \bm{\imath} & \ldots & \bm{\imath} & \bm{\imath}
\end{pmatrix},
$$
where  $\bm{\imath}$ stands for a vector 
$(\imath, \ldots, \imath)$ of suitable 
length. 
\end{enumerate}
\end{corollary}

\begin{proof}
If~(i) holds, then we may assume $P$ to be as $\tilde P$  in
Proposition~\ref{prop:zeta}.
Adding the $\mu_i$-fold of the $i$-th row 
to the last row brings $P$ into the desired form.
If~(ii) holds, take $u = (0,\ldots, 0,-1) \in \ZZ^{r+s}$. 
Then $P^* \cdot u \in \ZZ^{n+m}$ defines a divisor 
$\imath D_X$ with $D_X$ a canonical divisor of 
shape~(\ref{eqn:simplbdry}) and we see 
$\zeta_X = 1$.
\end{proof}

\begin{proposition}
\label{prop:zetaexcep}
Let $X = X(A,P)$ be a $\QQ$-Gorenstein affine $T$-variety 
of Type $2$ with at most log terminal singularities
and canonical multiplicity $\zeta_X > 1$.
Then we can choose $P$ of shape~\ref{prop:zeta}~(ii)
such that $l_{ij}=1$ and $\nu_{ij}=0$ holds 
for $i = 3, \ldots, r$ and $j = 1,\ldots, n_i$
and, moreover, $P$ satisfies one of the following 
cases:

\renewcommand{\arraystretch}{1.8} 

\begin{longtable}{c|c|c|c|c}
Case 
& 
$(l_{01},l_{11},l_{21})$ 
& 
$(\nu_0,\nu_1,\nu_2)$ 
& 
$\zeta_X$ 
& 
$\imath_X$
\\
\hline
$(i)$ 
&
$(4,3,2)$ 
& 
$\frac{1}{2}
(\bm{\imath_X} + l_0,\, \bm{\imath_X} - \imath_Xl_1, \, \bm{\imath_X} - l_2)$  
&
$2$ 
&
$ 0 \mod 2$
\\
\hline
$(ii)$ 
&
$(3,3,2)$ 
& 
$
\frac{1}{3}
(\bm{\imath_X} - l_0,\,
\bm{\imath_X} + l_1, \, 
\bm{\imath_X} - \imath_Xl_2) 
$
&
$3$ 
&
$0 \mod 3$
\\
\hline
$(iii)$ 
&
$(2k+1,2,2)$ 
& 
$
\frac{1}{4}
(\bm{\imath_X} - \imath_Xl_0, \, 
\bm{\imath_X} - l_1, \,
\bm{\imath_X} + l_2) 
$  
&
$4$ 
&
$2 \mod 4$
\\
\hline
$(iv)$ 
&
$(2k,2,2)$ 
& 
$
\frac{1}{2}
(\bm{\imath_X} - l_0, \,
\bm{\imath_X} + l_1, \, 
\bm{\imath_X}- \imath_X l_2)
$  
&
$2$ 
&
$0 \mod 2$
\\
\hline
$(v)$ 
&
$(k,2,2)$ 
& 
$
\frac{1}{2}
(\bm{\imath_X}-\imath_Xl_0, \,
\bm{\imath_X} - l_1, \, 
\bm{\imath_X} + l_2)
$  
&
$2$ 
&
$0 \mod 2$
\\
\hline
$(vi)$ 
&
$(k_0,k_1,1)$ 
& 
$
(\nu_0, \, \nu_1, \, \zeta_X^{-1}(\bm{\imath_X}-\imath_Xl_2))
$ 
&
&
\end{longtable}

\noindent
where  $\bm{\imath_X}$ stands for a vector 
$(\imath_X, \ldots, \imath_X)$ of suitable 
length, and in Case (vi), all the numbers 
$(\imath_X-\nu_{0j_0}\zeta_X) / l_{0j_0}$
and 
$(\nu_{1j_1}\zeta_X-\imath_X) / l_{1j_1}$ 
are integral and coincide.
\end{proposition}

\begin{proof}
Since $X = X(A,P)$ has at most log terminal 
singularities, Theorem~\ref{thm:affltcharintro} 
guarantees that the Cox ring $\mathcal{R}(X) = R(A,P)$ 
is platonic.
Thus, suitably exchanging data column blocks, we 
achieve $l_{ij} = 1$ for all $i \ge 3$.
Next, we bring $P$ in to the form of 
Proposition~\ref{prop:zeta}~(ii).
Finally, subtracting the  $\nu_{ij}$-fold 
of the $i$-th row from the last one, 
we achieve
$\nu_{ij}=0$ for $i = 3, \ldots, r$.

Observe that our new matrix $P$ still satisfies 
the conditions of Remark~\ref{rem:zetaPprops}.
For the integers $\mu_i$ defined there, we have
\begin{equation}
\label{eq:mu}
\mu_0 + \mu_1 + \mu_2 
\ = \ 
\mu_3 
\ = \ 
\ldots 
\ =  \ 
\mu_r  
\ =  \ 
\imath_X.
\end{equation}
Moreover, for $i=0,1,2$ set $\ell_i := l_{01} l_{11} l_{21} / l_{i1}$.
Then, because of  
$\imath_X + \mu_il_{ij} = \nu_{ij} \zeta_X$,  
we obtain
\begin{equation}
\label{eq:ltau}
\gcd(\ell_{0},\ell_{1},\ell_{2})^{-1} 
\sum_{i=0}^{2} \ell_{i} (\imath_X - \mu_il_{ij}) 
\ = \
\alpha \zeta_X
\quad \text{for some } \alpha \in \ZZ.
\end{equation}
Finally, Remark~\ref{rem:zetaPprops} ensures
\begin{equation}
\label{eq:gcd}
1
\ = \
\gcd(\mu_1,\ldots,\mu_r,\zeta_X,\imath_X)
\ = \
\gcd(\mu_1,\mu_2,\zeta_X,\imath_X).
\end{equation}
We will now apply these conditions to 
establish the table of the assertion.
Since $(l_{01},l_{11},l_{21})$ is a platonic
triple, we have to discuss the following 
cases.

\medskip

\noindent
\emph{Case 1}: $(l_{01},l_{11},l_{21})$ equals $(5,3,2)$.
Our task is to rule out this case.
Using~(\ref{eq:mu}) and~(\ref{eq:ltau}), we see 
that~$\zeta_X$ divides
$$
\imath_X
\ = \
31\imath_X-30(\mu_0+\mu_1+\mu_2)
\ = \
6(\imath_X-5\mu_0)+10(\imath_X-3\mu_1)+15(\imath_X-2\mu_2).
$$
Consequently, (\ref{eq:gcd}) becomes 
$\gcd(\mu_1,\mu_2,\zeta_X) = 1$
and from $\imath_X - \mu_il_{ij} = \nu_{ij} \zeta_X$ 
we infer that $\zeta_X$ divides $5\mu_0$, $3\mu_1$ 
and $2\mu_2$. 
This leaves us with the three possibilities 
$\zeta_X = 2,3,6$.

If $\zeta_X = 2$ holds, then $\zeta_X$ divides $\mu_0$ 
and $\mu_1$ but not $\mu_2$;
if $\zeta_X = 3$ holds, then~$\zeta_X$ divides $\mu_0$ 
and $\mu_2$ but not $\mu_1$.
Both contradicts to the fact that 
$\zeta_X$ divides $\imath_X = \mu_0+\mu_1+\mu_2$.
Thus, only $\zeta_X = 6$ is left. In that case,
$\zeta_X$ must divide~$\mu_0$.
Since $\zeta_X$ divides $\imath_X = \mu_0+\mu_1+\mu_2$,
we see that $\zeta_X$ divides $\mu_1+\mu_2$.
Moreover, $\zeta_X \mid 3\mu_1$ gives 
$\mu_1 = 2 \mu_1'$ and $\zeta_X \mid 2\mu_2$ gives 
$\mu_2 = 3 \mu_2'$ with integers $\mu_1',\mu_2'$.
Now, as $\zeta_X = 6$ divides $2\mu_1' + 3\mu_2'$,
we obtain that $\mu_2'$ and hence $\mu_2$ are even.
This contradicts $\gcd(\mu_1,\mu_2,\zeta_X) = 1$.

\medskip

\noindent
\emph{Case 2}: $(l_{01},l_{11},l_{21})$ equals $(4,3,2)$. 
Similarly as in the preceding case, we apply~(\ref{eq:mu}) 
and~(\ref{eq:ltau}) to see that~$\zeta_X$ divides
$$
\imath_X
\ = \
13\imath_X-12(\mu_0+\mu_1+\mu_2)
\ = \
\frac{1}{2}
\bigl(
6(\imath_X-4\mu_0)+8(\imath_X-3\mu_1)+12(\imath_X-2\mu_2)
\bigr).
$$
As before, we conclude $\gcd(\mu_1,\mu_2,\zeta_X) = 1$
and obtain that $\zeta_X$ divides $4\mu_0$, $3\mu_1$ 
and $2\mu_2$. This reduces to $\zeta_X = 2,3,6$.

If $\zeta_X = 3$ holds, then~$\zeta_X$ divides $\mu_0$ 
and $\mu_2$ but not $\mu_1$, contradicting the fact that 
$\zeta_X$ divides $\imath_X = \mu_0+\mu_1+\mu_2$.
If $\zeta_X = 6$ holds, then we obtain 
$\mu_0 = 3\mu_0'$, $\mu_1 = 2\mu_1'$ and $\mu_2 = 3 \mu_2'$ 
with suitable integers $\mu_i'$.
Since $\zeta_X$ divides $\imath_X = \mu_0+\mu_1+\mu_2$,
we obtain that $\mu_2$ is divisible by $3$, contradicting 
$\gcd(\mu_1,\mu_2,\zeta_X) = 1$.

Thus, the only possibility left is $\zeta_X = 2$.
We show that this leads to Case~(i) of the assertion.
Observe that $\mu_1$ is even, $\mu_2$ is odd because 
of $\gcd(\mu_1,\mu_2,\zeta_X) = 1$ and $\mu_2$ is odd
because $\imath_X = \mu_0+\mu_1+\mu_2$ is even.   
Recall that the vectors $\nu_i$ in the last row of $P$
are given as 
$$ 
\nu_i 
\ = \ 
\frac{1}{\zeta_X}(\bm{\imath_X}-\mu_i l_i)
\ = \ 
\frac{1}{2}\bm{\imath_X} - \frac{\mu_i}{2} l_i.
$$
Thus, adding the $(-\mu_0-\mu_2) / 2 $-fold of the first 
row and the $(\mu_2-1) / 2$-fold of the second row 
to the last row brings $P$ into the shape of Case~(i).

\medskip

\noindent
\emph{Case 3}: $(l_{01},l_{11},l_{21})$ equals $(3,3,2)$. 
As in the two preceding cases, we infer from
(\ref{eq:mu}) and~(\ref{eq:ltau}) that~$\zeta_X$ divides
$$
\imath_X
\ = \
7\imath_X-6(\mu_0+\mu_1+\mu_2)
\ = \
\frac{1}{3}
\bigl(
6(\imath_X-3\mu_0)+6(\imath_X-3\mu_1)+9(\imath_X-2\mu_2)
\bigr).
$$
Since $\gcd(\mu_1,\mu_2,\zeta_X)=1$ 
and $\zeta_X$ divides $3\mu_0$, $3\mu_1$, $2\mu_2$,
we are left with $\zeta_X=2,3,6$.
If $\zeta_X=2$ or $\zeta_X=6$ holds, 
then $\mu_0$, $\mu_1$ and
$\imath_X=\mu_0+\mu_1+\mu_2$ must be even. 
Thus also $\mu_2$ must be even, contradicting
$\gcd(\mu_1,\mu_2,\zeta_X)=1$.

Let $\zeta_X=3$. 
We show that this leads to Case~(ii) of the 
assertion. 
First,~$3$ divides~$\mu_2$ and $\imath_X = \mu_0+\mu_1+\mu_2$,
hence also $\mu_0+\mu_1$.
Moreover, $3$ divides neither $\mu_0$ nor $\mu_1$ 
because of $\gcd(\mu_1,\mu_2,\zeta_X)=1$. 
Interchanging, if necessary, the data of the
column blocks no.~0 and~1,
we achieve that $3$ divides $\mu_0-1$ and $\mu_1+1$.
So, at the moment, the $\nu_i$ in the last row of $P$ 
are of the form
$$ 
\nu_i 
\ = \ 
\frac{1}{\zeta_X}(\bm{\imath_X}-\mu_i l_i)
\ = \ 
\frac{1}{3}\bm{\imath_X} - \frac{\mu_i}{3} l_i.
$$
Adding the $(\mu_1+1)/3$-fold of the first and the 
$(-\mu_0-\mu_1)/3$-fold of the second to the last 
row of~$P$, we arrive at Case~(ii).

\medskip

\noindent
\emph{Case 4}: $(l_{01},l_{11},l_{21})$ equals $(k,2,2)$ with 
$k \geq 3$ odd. 
Then~(\ref{eq:mu}) and~(\ref{eq:ltau}) show that $\zeta_X$ 
divides
$$
2\imath_X
\ = \
(2+2k)\imath_X - 2k(\mu_0+\mu_1+\mu_2)
\ = \
\frac{1}{2}(4(\imath_X - k\mu_0)+2k(\imath_X-2\mu_1)+2k(\imath_X-2\mu_2)).
$$

\medskip

\noindent
\emph{Case 4.1}: 
$\zeta_X$ doesn't divide $\imath_X$.
Then we have $2\imath_X = \alpha \zeta_X$
with $\alpha \in \ZZ$ odd.
Thus,~$\zeta_X$ is even and 
$2 \mu_i = \imath_X - \nu_{ij}\zeta_X$ 
implies that $4 \mu_i$ is an odd multiple
of $\zeta_X$ for $i = 1,2$.
In particular, $4$ divides $\zeta_X$.
Moreover,~(\ref{eq:gcd}) 
implies $\gcd(\mu_1,\mu_2,\zeta_X/2)=1$
and we obtain $\zeta_X=4$. 
That means $\imath_X \equiv 2 \mod 4$.
Since $\zeta_X=4$ divides $\imath_X-k\mu_0$ 
and $k$ is odd, we conclude 
$\mu_0 \equiv 2 \mod 4$.
Then $\mu_0+\mu_1+\mu_2 =\imath_X \equiv 2 \mod 4$
implies that $4$ divides $\mu_1+\mu_2$.
Interchanging, if necessary, the data of 
the column blocks no.~1 and~2, 
we can assume $\mu_1 \equiv - \mu_2 \equiv 1 \mod 4$.
Then, adding the $(\mu_1-1)/4$-fold of the first 
and the $(\mu_2+1)/4$-fold of the second 
to the last row of~$P$, we arrive at Case~(iii) 
of the assertion.

\medskip

\noindent
\emph{Case 4.2}: $\zeta_X$ divides $\imath_X$. 
Then~(\ref{eq:gcd}) becomes $\gcd(\mu_1,\mu_2,\zeta_X)=1$.
Since $\zeta_X$ divides $2\mu_1$ and $2\mu_2$,
we see that $\zeta = 2$ holds and $\mu_1$, $\mu_2$ are odd.
Adding the $(\mu_1-1)/2$-fold of the first and 
the $(\mu_2+1)/2$-fold of the second to the last row of~$P$
leads to Case~(v) of the assertion.

\medskip

\noindent
\emph{Case 5}: $(l_{01},l_{11},l_{21})$ equals $(k,2,2)$ 
with $k \geq 2$ even. 
Then~(\ref{eq:mu}) and~(\ref{eq:ltau}) show that
$\zeta_X$ divides
$$
\imath_X
\ = \
(k+1)\imath_X - k(\mu_0+\mu_1+\mu_2)
\ = \
\frac{1}{4}(4(\imath_X - k\mu_0)+2k(\imath_X-2\mu_1)+2k(\imath_X-2\mu_2)).
$$
As earlier, we conclude that $\zeta_X | 2\mu_i$ for $i=1,2$
and $\zeta_X = 2$. 
Since $\gcd(\mu_1,\mu_2,2)=1$ holds 
and $\mu_0+\mu_1+\mu_2=\imath_X$ is even, 
two of the $\mu_i$ are be odd and one is even.
If $\mu_1$ and $\mu_2$ are odd, then adding the 
$(\mu_1-1)/2$-fold of the first and the 
$(\mu_2+1)/2$-fold of the second to the 
last row of~$P$ leads to Case (v). 
Now, let $\mu_0$ be odd.
Interchanging, if necessary, the data of the
column blocks no.~1 and~2,
we achieve that $\mu_1$ is odd. 
Then we add the $(\mu_1+1)/2$-fold of the 
first and the $(-\mu_0-\mu_1)/2$-fold of the second 
to the last row of~$P$ and arrive at 
Case~(iv) of the assertion.

\medskip

\noindent
\emph{Case 6}. $(l_{01},l_{11},l_{21})$ equals $(k_0,k_1,1)$, 
where $k_0, k_1 \in \ZZ_{>0}$. 
We subtract the $\nu_{21}$-fold of the second row of~$P$ 
from the last one. 
Since $\nu_{21}=(\imath_X-\mu_2)/\zeta_X$ holds, we 
obtain $\nu_2 = \zeta_X^{-1}(\bm{\imath_X}-\imath_Xl_2)$. 
Moreover,~(\ref{eq:mu}) becomes $\mu_0+\mu_1=0$.
We arrive at Case (vi) of the assertion 
by observing
$$
(\imath_X-\nu_{0j_0}\zeta_X) / l_{0j_0}
\ = \
\mu_0
\ = \
-\mu_1
\ = \
(\nu_{1j_1}\zeta_X-\imath_X) / l_{1j_1}.
$$
\end{proof}

\begin{example}
\label{ex::lt-surfaces}
We discuss the rational affine 
$\CC^*$-surfaces $X$ 
with at most log terminal singularities.
First, the affine toric surfaces 
$X = \CC^2 / C_k$ show up here, where 
$C_k$ is the cyclic group of order $k$ 
acting diagonally.
In terms of toric geometry, these surfaces
are given as 
$$
X 
\ = \
\Spec \, \CC[\sigma^\vee \cap \ZZ^2],
\qquad 
\sigma = \cone((k, \imath), \, (\imath, k + m)),
$$
where $k,m \in \ZZ_{> 0}$ with 
$\gcd(k,\imath) = \gcd(k+m,\imath) =1$ and 
$\imath$ is the Gorenstein index of $X$;
see~\cite[Chap.~10]{CLS} for more background.
Now consider a non-toric $\CC^*$-surface 
$X = X(A,P)$ of Type~2.
As a quotient of $\CC^2$ by a finite group,
$X$ has finite divisor class group and thus
$P$ is a $3 \times 3$ matrix of the shape
$$ 
P
\ = \ 
\left[
\begin{array}{rrr}
-l_{01} & l_{11} & 0
\\
-l_{01} & 0 & l_{21}
\\
d_{01} & d_{11} & d_{21}
\end{array}
\right].
$$
Theorem~\ref{thm:affltcharintro} says that 
$(l_{01},l_{11},l_{21})$ is a platonic triple.
Moreover, Corollary~\ref{cor::Gore} and 
Proposition~\ref{prop:zetaexcep} 
provide us with constraints on the 
$d_{i1}$.
Having in mind that~$P$ is of rank three with 
primitive columns, 
one directly arrives at the following possibilities,
where $\zeta = \zeta_X$ is the canonical multiplicity
and $\imath = \imath_X$ the Gorenstein index:

\renewcommand{\arraystretch}{2} 

\begin{longtable}{c|c|c|c}
Type 
& 
$P$
&
\hbox{\quad $\zeta$ \quad}
&
$\imath$
\\
\hline
$D_n^{1,\imath}$
&
\renewcommand{\arraystretch}{1} 
{\tiny
$
\left[
\begin{array}{rrr}
-n+2 & 2 & 0
\\
-n+2 & 0 & 2 
\\
-n \imath +3 \imath & \imath & \imath
\\
\end{array}
\right]
$
}
\renewcommand{\arraystretch}{2} 
&
$1$
&
$\gcd(\imath,2n) = 1$
\\
\hline
$D_{2n+1}^{2,\imath}$
&
\renewcommand{\arraystretch}{1} 
{\tiny
$
\left[
\begin{array}{rrr}
-2n+1 & 2 & 0
\\
-2n+1 & 0 & 2 
\\
(1-n)\imath & \imath/2+1 & \imath/2-1
\\
\end{array}
\right]
$
}
\renewcommand{\arraystretch}{2} 
&
$2$
&
$\gcd(\imath,8n-4) = 4$
\\
\hline
$E_6^{1,\imath}$
&
\renewcommand{\arraystretch}{1} 
{\tiny
$
\left[
\begin{array}{rrr}
-3 & 3 & 0
\\
-3 & 0 & 2 
\\
-2 \imath & \imath & \imath
\\
\end{array}
\right]
$
}
\renewcommand{\arraystretch}{2} 
&
$1$
&
$\gcd(\imath,6) = 1$
\\
\hline
$E_6^{3,\imath}$
&
\renewcommand{\arraystretch}{1} 
{\tiny
$
\left[
\begin{array}{rrr}
-3 & 3 & 0
\\
-3 & 0 & 2 
\\
\imath/3-1 & \imath/3+1 & -\imath/3
\\
\end{array}
\right]
$
}
\renewcommand{\arraystretch}{2} 
&
$3$
&
$\gcd(\imath, 18) = 9$ 
\\
\hline
$E_7^{1,\imath}$
&
\renewcommand{\arraystretch}{1} 
{\tiny
$
\left[
\begin{array}{rrr}
-4 & 3 & 0
\\
-4 & 0 & 2 
\\
-3 \imath & \imath & \imath
\\
\end{array}
\right]
$
}
\renewcommand{\arraystretch}{2} 
&
$1$
&
$\gcd(\imath,6) = 1$
\\
\hline
$E_8^{1,\imath}$
&
{\tiny
\renewcommand{\arraystretch}{1} 
$
\left[
\begin{array}{rrr}
-5 & 3 & 0
\\
-5 & 0 & 2 
\\
-4 \imath & \imath & \imath
\end{array}
\right]
$
\renewcommand{\arraystretch}{2} 
}
&
$1$
&
$\gcd(\imath,30) = 1$
\end{longtable}

\noindent
For geometric details on these surfaces, 
we refer to the work of Brieskorn~\cite{Br},
and, in the context of the McKay Correspondence,
Wunram~\cite{Wun} and Wemyss~\cite{Wem}.
\end{example}

\section{Geometry of the total coordinate space}

We take a closer look at the geometry of the 
total coordinate space $\b{X}$ 
of a $T$-variety $X$ of complexity one.
The first result says in particular 
that $\b{X}$ is Gorenstein and canonical 
provided that $X$ is log terminal 
and affine.

\begin{proposition}
\label{lemm::GorenstCan}
Let $R(A,P_0)$ be a platonic ring of Type $2$. 
Then the affine variety 
\linebreak
$\b{X} = \Spec \, R(A,P_0)$ is Gorenstein 
and has at most canonical singularities.
\end{proposition}

\begin{proof}
Adding suitable rows, we complement the 
matrix $P_0$ to a square matrix $P$ of 
full rank with last row
$(\bm{1} - (r-1)l_0, \bm{1}, \ldots, \bm{1})$,
where $\bm{1}$ indicates vectors of length~$n_i$ 
with all entries equal to one;
this is possible, because the last row 
is not in the row space of~$P_0$.
Then $X = X(A,P)$ is a $\QQ$-factorial
affine $T$-variety.
Theorem~\ref{thm:affltcharintro}
tells us that $X$ has at most log terminal 
singularities and
Corollary~\ref{cor::Gore} ensures that~$X$ 
is Gorenstein. 
Thus, $X$ has at most canonical singularities.
Since $\b{X} \to X$ is finite with ramification
locus of codimension at least two, we can 
use~\cite[Thm.~6.2.9]{Is} to see that $\b{X}$ is 
Gorenstein with at most canonical 
singularities.
\end{proof}

Now we investigate the generic quotient 
$Y$ of $\b{X}$ by the action of the 
unit component $H_0^0 \subseteq H_0$,
in other words, the smooth projective 
curve $Y$ with function field 
$\CC(Y) = \CC(\b{X})^{H_0^0}$.
Note that the curve~$Y$ occurs also 
in~\cite{AlPe}, where it carries the 
polyhedral divisor of the Cox ring.

\begin{definition}
\label{def:variousl}
Consider the defining matrix $P_0$ of a ring $R(A,P_0)$
of Type~2 and the vectors $l_i = (l_{i1}, \ldots, l_{in_i})$ 
occuring in the rows of $P_0$. Set
$$ 
\mathfrak{l}_i  := \gcd(l_{i1}, \ldots, l_{in_i}),
\qquad
\mathfrak{l} := \gcd(\mathfrak{l}_0, \ldots, \mathfrak{l}_r),
\qquad
\mathfrak{l}_{ij} 
:= 
\gcd(\mathfrak{l}^{-1}\mathfrak{l}_i, \mathfrak{l}^{-1}\mathfrak{l}_j),
$$
$$
\overline{\mathfrak{l}}
:=
\lcm(\mathfrak{l}_0, \dots, \mathfrak{l}_r),
\qquad
b_i:= \mathfrak{l}_i^{-1} \overline{\mathfrak{l}},
\qquad
 b(i): = \gcd(b_j; \ j \neq i).
$$ 
\end{definition}

\goodbreak

\begin{theorem}
\label{theo::curve}
Let $R(A,P_0)$ be of Type $2$ and consider the 
action of the unit component $H_0^0 \subseteq H_0$
of the quasitorus $H_0 = \Spec \, \CC[K_0]$ on 
$\b{X} = \Spec \, R(A,P_0)$.
Then the smooth projective curve $Y$ with function 
field $\CC(Y) = \CC(\b{X})^{H_0^0}$ is of genus
$$ 
g(Y)
\ = \ 
\frac{\mathfrak{l}_0 \cdots \mathfrak{l}_r}{2 \overline{\mathfrak{l}}} 
\left(
(r-1) 
- 
\sum_{i=0}^r \frac{b(i)}{\mathfrak{l}_i}
\right) 
+ 1.
$$ 
\end{theorem}

\begin{lemma}
\label{lemma::amu}
Let $R(A,P_0)$ be of Type $2$, 
consider the degree $u := \deg(g_0) \in K_0$
of the defining relations and the subgroup 
$$
K_0(u) 
\ := \ 
\{w \in K_0; \, \alpha w \in \ZZ u 
\text{ for some } \alpha \in \ZZ_{> 0} \}
\ \subseteq \ 
K_0.
$$
Then the Veronese subalgebra $R(A,P_0)(u)$ of $R(A,P_0)$ 
associated with $K_0(u)$ of $K_0$
is generated by the monomials 
$T_0^{l_0/\mathfrak{l}_0}, \ldots, T_r^{l_r/\mathfrak{l}_r}$.
\end{lemma}

\begin{proof}
First, observe that every element of $R(A,P_0)(u)$
is a polynomial in the variables $T_{ij}$.
Now consider a monomial $T^l$ in the $T_{ij}$ 
of degree $w \in K_0(u)$, where $l \in \ZZ^{n+m}$. 
Then $\alpha w \in \beta_0 u$ holds for some 
$\alpha \in \ZZ_{> 0}$ and $\beta_0 \in \ZZ$.
Moreover, there are $\beta_1, \ldots, \beta_r \in \ZZ$
with
$$
\alpha l 
 = 
\beta_0 l_0' + \beta_1 (l_0'-l_1') + \ldots + \beta_r(l_0' -l_r'),
\text{ where }
l_i'  :=  l_{i1} e_{i1} + \ldots + l_{in_i} e_{in_i},
$$
reflecting the fact that $\alpha l - \beta_0 l_0'$
lies in the row space of $P_0$.
Consequently, we obtain
$l = \beta_0' l_0' + \ldots + \beta_r' l_r'$
for suitable $\beta_i' \in \QQ$. 
Since $l$ has only non-negative integer entries, 
we conclude that every $\beta_i'$ is a non-negative 
integral multiple of $\mathfrak{l}_i^{-1}$.
Thus, $T^l$ is a monomial in the $T_i^{l_i/\mathfrak{l}_i}$.
The assertion follows.
\end{proof}

\begin{proof}[Proof of Theorem~\ref{theo::curve}]
The curve $Y$ occurs as a GIT-quotient: 
$Y = \b{X}^{ss}(u^0) / H_0^0$, where 
$u^0 \in \Chi(H_0^0)$ represents the character 
induced by $u = \deg(g_0) \in K_0 = \Chi(H_0)$.
In other words, we have $Y = \Proj \, R(A,P_0)(u^0)$
with the Veronese subalgebra defined by $u^0$.
We may replace $u^0$ with 
$$
w^0 \ := \ 
\frac{1}{\overline{\mathfrak{l}}} u^0 
\ \in \ 
\Chi(H_0^0).
$$
Then $R(A,P_0)(u^0)$ is replaced with 
$R(A,P_0)(w^0)$ which in turn equals the Veronese 
subalgebra treated in Lemma~\ref{lemma::amu}.
Moreover, the generators 
$T_i^{l_i/\mathfrak{l}_i} \in R(A,P_0)(w^0)$ 
are of degree $b_iw^0 \in  \Chi(H_0^0)$,
respectively.  
We obtain a closed 
embedding into a weighted projective space
$$
Y 
\ = \ 
V(h_0, \ldots, h_{r-2}) 
\ \subseteq \ 
\PP(b_0, \dots, b_r),
\qquad
h_{i}
\ :=  \
\det
\left[
\begin{array}{lll}
T_i^{\mathfrak{l}_i}& T_{i+1}^{\mathfrak{l}_{i+1}}&T_{i+2}^{\mathfrak{l}_{i+2}}
\\
a_i & a_{i+1}& a_{i+2}
\end{array}
\right],
$$
where the $h_i$ generate the ideal of relations among the generators 
of the Veronese subalgebra $R(A,P_0)(w^0)$.
The idea is now to construct a ramified 
covering $Y' \to Y$ with a suitable curve $Y'$ 
and then to compute the genus of $Y$ via the Hurwitz
formula. Consider
$$ 
Y'
\ = \ 
V(h_0', \ldots, h_{r-2}') 
\ \subseteq \ 
\PP_r,
\qquad
h_{i}'
\ :=  \
\det
\left[
\begin{array}{lll}
T_i^{\b{\mathfrak{l}}} & T_{i+1}^{\b{\mathfrak{l}}} &T_{i+2}^{\b{\mathfrak{l}}}
\\
a_i & a_{i+1}& a_{i+2}
\end{array}
\right].
$$
The $Y' \subseteq \PP_r$ is a smooth complete 
intersection curve.
Computing the genus of $Y'$ according
to~\cite{Fu}, we obtain
$$
g(Y') 
\ = \ 
\frac{1}{2}((r-1)
\overline{\mathfrak{l}}^r - (r+1)\overline{\mathfrak{l}}^{r-1})+1.
$$
The morphism $\PP_r \to \PP(b_0, \ldots, b_r)$
sending $[z_0,\ldots,z_r]$ to $[z_0^{b_0},\ldots,z_r^{b_r}]$
restricts to a morphism $Y' \to Y$ 
of degree $b_0 \cdots b_r$.
The intersection $Y \cap U_i$ with the $i$-th coordinate
hyperplane $U_i \subseteq \PP_r$ contains   
precisely $\overline{\mathfrak{l}}^{r-1}$ points and 
each of these points has ramification 
order $b_i\cdot b(i)-1$.
Outside the $U_i$, the morphism $Y' \to Y$ is unramified.
The Hurwitz formula then gives $g(Y)$.
\end{proof}

We now use Theorem~\ref{theo::curve} to characterize 
rationality of $\b{X} = \Spec \, R(A,P_0)$.
For the special case of Pham-Brieskorn surfaces,
the following statement has been obtained in~\cite{BaKp}.

\begin{proposition}
\label{prop::Hyper}
Let $R(A,P_0)$ be of Type $2$ with $r=2$, that 
means that $\b{X} = \Spec \, R(A, P_0)$ is 
given as
$$ 
\b{X} 
\ \cong \ 
V(
T_{01}^{l_{01}} \cdots T_{0n_0}^{l_{0n_0}} 
\, + \, 
T_{11}^{l_{11}} \cdots T_{1n_1}^{l_{1n_1}}
\, + \, 
T_{21}^{l_{21}} \cdots T_{2n_2}^{l_{2n_2}} 
)
\ \subseteq \
\CC^n.
$$
Then the hypersurface $\b{X}$ is rational
if and only if one of the following conditions 
holds:
\begin{enumerate}
\item
there are pairwise coprime positive integers 
$c_0, c_1,c_2$ and a positive integer~$s$ 
such that, after suitable renumbering, one has
$$
\gcd(c_2,s)=1,
\quad
\mathfrak{l}_0=sc_0, 
\quad \mathfrak{l}_1=sc_1, 
\quad \mathfrak{l}_2=c_2;
$$
\item
there are pairwise coprime positive integers $c_0, c_1,c_2$ 
such that
$$
\mathfrak{l}_0=2c_0, 
\quad 
\mathfrak{l}_1=2c_1, 
\quad 
\mathfrak{l}_2=2c_2.
$$
\end{enumerate}
\end{proposition}

\begin{lemma}
\label{lem:variousl}
For $i=0,1,2$, let $l_i = (l_{i1}, \ldots, l_{in_i})$
be tuples of positive integers.
Define $\mathfrak{l}$, $\mathfrak{l}_i$ 
and $\mathfrak{l}_{ij}$ as in 
Definition~\ref{def:variousl} for $r=2$.
Then the following statements are equivalent.
\begin{enumerate}
\item
We have
$
\mathfrak{l}(\mathfrak{l}\mathfrak{l}_{01}\mathfrak{l}_{02}\mathfrak{l}_{12}
-
(\mathfrak{l}_{01}+\mathfrak{l}_{02}+\mathfrak{l}_{12}))
= 
-2
$.
\item
One of the following two conditions holds:
\begin{enumerate}
\item
there are pairwise coprime positive integers 
$c_0, c_1,c_2$ and a positive integer~$s$ 
such that, after suitable renumbering, one has
$$
\gcd(c_2,s)=1,
\quad
\mathfrak{l}_0=sc_0, 
\quad \mathfrak{l}_1=sc_1, 
\quad \mathfrak{l}_2=c_2;
$$
\item
there are pairwise coprime positive integers $c_0, c_1,c_2$ 
such that
$$
\mathfrak{l}_0=2c_0, 
\quad 
\mathfrak{l}_1=2c_1, 
\quad 
\mathfrak{l}_2=2c_2.
$$
\end{enumerate}
\end{enumerate}
\end{lemma}

\begin{proof}
If~(ii) holds, then a simple computation shows
that~(i) is valid.
Now, assume that~(i) holds. 
Then the following cases have to be 
considered.

\medskip

\noindent
\emph{Case 1}. We have $\mathfrak{l}=1$. 
Then 
$
\mathfrak{l}_{01}(\mathfrak{l}_{02}\mathfrak{l}_{12}-1)
=
\mathfrak{l}_{02}+\mathfrak{l}_{12}-2
$
holds.
From this we deduce 
\begin{eqnarray*}
\mathfrak{l}_{01}(\mathfrak{l}_{02}\mathfrak{l}_{12}-1)
& = & 
(\mathfrak{l}_{01}-1)(\mathfrak{l}_{02}\mathfrak{l}_{12}-1)
+
(\mathfrak{l}_{02}-1)(\mathfrak{l}_{12}-1)+\mathfrak{l}_{02}
+
\mathfrak{l}_{12}-2
\\
& \ge &
\mathfrak{l}_{02}+\mathfrak{l}_{12}-2,
\end{eqnarray*}
where equality holds if and only if at least two of 
$\mathfrak{l}_{01}, \mathfrak{l}_{02}, \mathfrak{l}_{12}$ 
equal one. 
So, we arrive at Condition~(a).

\medskip

\noindent
\emph{Case 2}. We have $\mathfrak{l}=2$. 
Then we have 
$
\mathfrak{l}_{01}(2\mathfrak{l}_{02}\mathfrak{l}_{12}-1)+1
=
\mathfrak{l}_{02}+\mathfrak{l}_{12}.
$
In this situation, we conclude
\begin{eqnarray*}
\mathfrak{l}_{01}(2\mathfrak{l}_{02}\mathfrak{l}_{12}-1)+1 
& = &
(\mathfrak{l}_{01}-1)(2\mathfrak{l}_{02}\mathfrak{l}_{12}-1)
+\mathfrak{l}_{02}\mathfrak{l}_{12}
\\
& & 
\quad +(\mathfrak{l}_{02}-1)(\mathfrak{l}_{12}-1)+\mathfrak{l}_{02}
+
\mathfrak{l}_{12}-1
\\
& \ge &
\mathfrak{l}_{02}+\mathfrak{l}_{12},
\end{eqnarray*}
where equality holds if and only if 
we have $\mathfrak{l}_{01}=\mathfrak{l}_{02}=\mathfrak{l}_{12}=1$.
Thus, we arrive at Condition~(b).
\end{proof}

\begin{proof}[Proof of Proposition~\ref{prop::Hyper}]
First, observe that $\b{X}$ is rational if and only 
if $Y$ is rational or, in other words, of genus zero.
For $r=2$, Theorem~\ref{theo::curve} gives
$$ 
g(Y) 
\ = \ 
\frac{\mathfrak{l}}{2}
(\mathfrak{l}\mathfrak{l}_{01}\mathfrak{l}_{02}\mathfrak{l}_{12}
- \mathfrak{l}_{01} -\mathfrak{l}_{02} -\mathfrak{l}_{12})
+1.
$$
Thus, according to Lemma~\ref{lem:variousl}, 
condition $g(Y) = 0$ holds if and only if~(i) 
or~(ii) of the proposition holds.
\end{proof}

\begin{remark}
If the defining polynomial in Proposition~\ref{prop::Hyper}
is classically homogeneous, then it defines a projective 
hypersurface $X' \subseteq \PP^{n-1}$ and the following 
statements are equivalent.
\begin{enumerate}
\item
$X'$ is rational.
\item
$\Cl(X')$ is finitely generated.
\item
Condition~\ref{prop::Hyper}~(i) or~(ii) holds.
\end{enumerate}
\end{remark}

\begin{corollary}
\label{cor:ratchar}
Let $R(A,P_0)$ be of Type $2$.
Then $\b{X} = \Spec \, R(A,P_0)$ is rational if 
and only if one of the following conditions
holds:
\begin{enumerate}
\item
We have $\gcd(\mathfrak{l}_i,\mathfrak{l}_j) = 1$
for all $0 \le i < j \le r$, in other words, $R(A,P_0)$ is factorial.
\item
There are $0 \le i < j \le r$ with 
$\gcd(\mathfrak{l}_i,\mathfrak{l}_j) > 1$
and $\gcd(\mathfrak{l}_u,\mathfrak{l}_v) = 1$ 
whenever $v \not\in \{i,j\}$.
\item
There are $0 \le i < j < k \le r$
with 
$
\gcd(\mathfrak{l}_i,\mathfrak{l}_j) = 
\gcd(\mathfrak{l}_i,\mathfrak{l}_k) =
\gcd(\mathfrak{l}_j,\mathfrak{l}_k) =
2
$ 
and $\gcd(\mathfrak{l}_u,\mathfrak{l}_v) = 1$ 
whenever $v \not\in \{i,j,k\}$.
\end{enumerate} 
\end{corollary}

\begin{lemma}
\label{lem:redtohysu}
Let $A,P_0$ be defining data of Type~2,
enhance $A$ to $A'$ by attaching a 
further column
and $P_0$ to $P'_0$ by attaching
$l_{r+1}$ to $l_0, \ldots, l_r$.
If $\gcd(\mathfrak{l}_i,\mathfrak{l}_{r+1}) = 1$
holds for $i = 0, \ldots, r$, then 
we have $g(Y) = g(Y')$ for the 
curves associated with $R(A,P)$ and 
$R(A',P_0')$ respectively.
\end{lemma}

\begin{proof}
Denote the numbers arising from $P'$ 
in the sense of Definition~\ref{def:variousl} 
by $\mathfrak{l}_i'$, $\mathfrak{l}'$ etc.
Then we have
$$ 
\b{\mathfrak{l}}' 
\ = \ 
\b{\mathfrak{l}}\mathfrak{l}_{r+1},
\quad
b'(i) 
\ = \ 
\gcd(\b{\mathfrak{l}}, \, \b{\mathfrak{l}}' / \b{\mathfrak{l}}_j; \; j \ne i)
\ = \ 
b(i),
\quad 
i = 0, \ldots, r,
$$
$$
b(r+1)
\ = \ 
\gcd(\b{\mathfrak{l}}' / \b{\mathfrak{l}}_0, 
\ldots, 
\b{\mathfrak{l}}' / \b{\mathfrak{l}}_r)
\ = \ 
\mathfrak{l}_{r+1}.
$$
Plugging these identities into the genus formula 
of Theorem~\ref{theo::curve},
we directly obtain $g(Y') = g(Y)$.  
\end{proof}

\begin{lemma}
\label{lem:ijk}
Let $R(A,P_0)$ be of Type $2$ and assume 
that the curve $Y$ associated with $R(A,P)$ 
is of genus zero.
Then there are $0 \le i \le j \le  k \le r$
with $\gcd(\mathfrak{l}_u,\mathfrak{l}_v) = 1$ 
whenever $v \not\in \{i,j,k\}$.
\end{lemma}

\begin{proof}
According to Theorem~\ref{theo::curve}, the fact 
that the curve $Y$ associated with~$R(A,P)$ 
is of genus zero implies
$$ 
\sum_{i=0}^r \frac{b(i)}{\mathfrak{l}_i}
\ = \ 
(r-1) 
+
\frac{2 \b{\mathfrak{l}}}{\mathfrak{l}_0 \cdots \mathfrak{l}_r}
\ > \ 
r-1.
$$
As $b(i)$ divides $\mathfrak{l}_i$, we see that 
$b(i) \ne \mathfrak{l}_i$ can happen at most three times.
Moreover, $b(i) = \mathfrak{l}_i$ is equivalent 
to $\gcd(\mathfrak{l}_i,\mathfrak{l}_j) = 1$ for 
all $j \ne i$.
\end{proof}

\begin{proof}[Proof of Corollary~\ref{cor:ratchar}]
We may assume that the indices $i,j$ and $k$ of 
Lemma~\ref{lem:ijk} are $0,1$ and $2$.
Then Lemma~\ref{lem:redtohysu} says that $\b{X}$ 
is rational if and only if the trinomial hypersurface
defined by the exponent vectors $l_0,l_1,l_2$ is rational.
Thus, Proposition~\ref{prop::Hyper} gives the assertion.
\end{proof}

\begin{corollary}
\label{cor:rat}
Let $R(A,P_0)$ be a platonic ring of Type $2$. 
Then $\b{X} = \Spec \, R(A,P_0)$ is rational.
\end{corollary}

\begin{remark}
\label{rem:rattcs}
It may happen that for a 
rational $T$-variety $X$
of complexity one, 
the total coordinate space $\b{X}$ 
is rational, 
but the total coordinate space 
of $\b{X}$ not any more.
For instance consider
$$
X_3 
\ := \ 
V(T_1^4 + T_2^4 + T_3^4) 
\ \subseteq \ 
\CC^3.
$$
Then, according to Proposition~\ref{prop::Hyper},
the surface $X_3$ is not rational.
Moreover, $X_3$ is the total coordinate space 
of an affine rational $\CC^*$-surface $X_2$ 
with defining matrix 
$$
P_2
\ = \ 
\left[
\begin{array}{rrr}
-4 & 4 & 0
\\
-4 & 0 & 4
\\
-3 & 1 & 1
\end{array}
\right].
$$
The divisor class group of $X_2$ is 
$\Cl(X_2) = \ZZ / 4 \ZZ \times \ZZ / 4\ZZ$ 
and the $\Cl(X_2)$-grading of the Cox ring 
$
\mathcal{R}(X_2) 
=  
\CC[T_1,T_2,T_3]/ \bangle{T_1^4 + T_2^4 + T_3^4}$
is given by
$$ 
\deg(T_1) \ = \ (\b{1},\b{1}),
\quad
\deg(T_2) \ = \ (\b{1},\b{2}),
\quad
\deg(T_3) \ = \ (\b{2},\b{1}).
$$
For an equation for $X_2$, compute 
the degree zero subalgebra 
of $\mathcal{R}(X_2)$: it has three generators $S_1,S_2,S_3$ 
and $S_1^3 + S_2^3 + S_3^4$ as defining relation.
Thus,  
$$
X_2 
\ \cong \ 
V(S_1^3 + S_2^3 + S_3^4) 
\ \subseteq \ 
\CC^3.
$$
To obtain a rational affine $\CC^*$-surface having 
$X_2$ as its total coordinate space, we take $X_1$,
defined by 
$$
P_1
\ := \ 
\left[
\begin{array}{rrr}
-3 & 3 & 0
\\
-3 & 0 & 4
\\
-2 & 1 & 1
\end{array}
\right].
$$
The divisor class group of $X_1$ is 
$\Cl(X_1) = \ZZ / 3 \ZZ$ 
and the $\Cl(X_1)$-grading of the Cox ring 
$
\mathcal{R}(X_1) 
=  
\CC[S_1,S_2,S_3]/ \bangle{S_1^3 + S_2^3 + S_3^4}$
is given by
$$ 
\deg(T_1) \ = \ \b{1},
\quad
\deg(T_2) \ = \ \b{2},
\quad
\deg(T_3) \ = \ \b{0}.
$$
We have constructed a chain of total coordinate spaces 
$X_3 \to X_2 \to X_1$, where $X_1$ is a rational 
affine $\CC^*$-surface, $X_2$ is rational and
$X_3$ not.
\end{remark}

Finally, we determine the factor group of the 
maximal quasitorus by its unit component acting 
on a given trinomial hypersurface; the proof 
is a direct consequence of the subsequent lemma.

\begin{proposition}
Let $R(A,P)$ be any ring of Type~2, where $r=2$.
Then, for the quasitorus $H_0$ acting on the 
corresponding trinomial hypersurface 
$$
\b{X} 
\ \cong \  
V(
T_{01}^{l_{01}} \cdots T_{0n_0}^{l_{0n_0}}
\, + \,
T_{11}^{l_{11}} \cdots T_{1n_1}^{l_{1n_1}}
\, + \,
T_{21}^{l_{21}} \cdots T_{2n_2}^{l_{2n_2}} 
)
\subseteq 
\CC^n,
$$
the factor group $H_0/H_0^0$ by the unit 
component $H_0^0 \subseteq H_0$
is isomorphic to the product of cyclic groups
$C(\mathfrak{l})
\times 
C(\mathfrak{l}\mathfrak{l}_{01}\mathfrak{l}_{02}\mathfrak{l}_{12})$.
\end{proposition}

\begin{lemma}
\label{lem:smith}
Consider a matrix $P_0$ with $m=0$ and $r=2$ 
as in Type~2 of Construction~\ref{constr:RAPdown}:
$$ 
P_0 
\ = \ 
\left[
\begin{array}{rrr}
-l_0 & l_1 & 0
\\
-l_0 & 0 & l_2
\end{array}
\right].
$$
As earlier, set $\mathfrak{l}_i = \gcd(l_{i1}, \ldots, l_{n_i})$.
Then, with 
$\mathfrak{l}_{ij} =  \gcd(\mathfrak{l}_i,\mathfrak{l}_j)$
and
$\mathfrak{l} =  \gcd(\mathfrak{l}_0,\mathfrak{l}_1,\mathfrak{l}_2)$,
we obtain 
$$
K_0^{\rm tors} 
\ = \ 
(\ZZ^{n} / {\rm im}(P_0^*))^{\rm tors} 
\ \cong \ 
C(\mathfrak{l})
\times 
C(\mathfrak{l}\mathfrak{l}_{01}\mathfrak{l}_{02}\mathfrak{l}_{12}).
$$
\end{lemma}

\begin{proof}
Suitable elementary column operations to $P_0$
transform the entries $l_i$ to 
$(\mathfrak{l}_i,0,\ldots, 0)$.
Thus, $K_0^{\rm tors} \cong (\ZZ^{3} / {\rm im}(P_1^*))^{\rm tors}$
holds with the $2 \times 3$ matrix
$$ 
P_1 
\ := \ 
\left[
\begin{array}{rrr}
-\mathfrak{l}_0 & \mathfrak{l}_1 & 0
\\
-\mathfrak{l}_0 & 0 & \mathfrak{l}_2
\end{array}
\right].
$$
The determinantal divisors of $P_0$ are 
$\gcd(\mathfrak{l}_0,\mathfrak{l}_1, \mathfrak{l}_2)$
and $\gcd(\mathfrak{l}_0\mathfrak{l}_1,\mathfrak{l}_0\mathfrak{l}_2,
\mathfrak{l}_1\mathfrak{l}_2)$.
Thus, the invariant factors of $P_0$ are
$\mathfrak{l}$
and 
$\mathfrak{l}\mathfrak{l}_{01}\mathfrak{l}_{02}\mathfrak{l}_{12}$;
see~\cite{Ne}.
\end{proof}

\section{Proof of Theorems~\ref{theorem::CoxIteration} 
and~\ref{theorem::quotient}}

We are ready to prove our first main results.
The proof of Theorem~\ref{theorem::CoxIteration}
will be in fact constructive in the sense 
that it allows to compute the defining equations 
of the Cox ring in every iteration step;
see Proposition~\ref{prop::isotropy}.

\begin{remark}
\label{rem:leadPlat}
Let $R(A,P)$ resp.~$R(A,P_0)$ be a ring of Type~2. 
Applying suitable admissible operations,
one achieves that $P$ resp.~$P_0$ (is \emph{ordered} 
in the sense that 
$l_{i1} \ge \ldots \ge l_{in_i}$
for all $i = 0, \ldots, r$ 
and ${l_{01} \ge \ldots \ge l_{r1}}$ 
hold. 
For an ordered $P$ resp.~$P_0$, 
the ring $R(A,P)$ resp.~$R(A,P_0)$
is platonic 
if and only if $(l_{01},l_{11},l_{21})$ 
is a platonic triple and $l_{i1} =1$ 
holds for $i \geq 3$.
\end{remark}

\begin{definition}
The \emph{leading platonic triple} of
a ring $R(A,P)$ resp.~$R(A,P_0)$
of Type~2 is the triple
$(l_{01},l_{11},l_{21})$ obtained after 
ordering $P$ resp.~$P_0$.
\end{definition}

\begin{lemma}
\label{lemma::Ptors}
Let $R(A,P_0)$ be of Type~$2$ and platonic
such that $l_{i1} \ge \ldots \ge l_{in_i}$ 
holds for all $i$ and $l_{i1} = 1$ for $i \ge 3$.
Moreover, assume 
$\gcd(\mathfrak{l}_1, \mathfrak{l}_2 ) = \mathfrak{l}$.
Then, with $K_0 = \ZZ^{n+m}/\im(P_0^*)$,
the kernel of $ \ZZ^{n+m}\to K_0/K_0^{\mathrm{tors}}$ 
is generated by the rows of the matrix 
$$
P_1
\ := \
\left[
\begin{array}{ccccccccc}
\frac{-1}{\gcd(\mathfrak{l}_0,\mathfrak{l}_1)} l_0
& 
\frac{1}{\gcd(\mathfrak{l}_0,\mathfrak{l}_1)} l_1
& 0 &  \dots & &0 & 0 & \dots & 0
\\[5pt]
\frac{-1}{\gcd(\mathfrak{l}_0,\mathfrak{l}_2)} l_0
& 0 
& 
\frac{1}{\gcd(\mathfrak{l}_0,\mathfrak{l}_2)} l_2 
& 0 & & 0& & &\\
-l_0 & 0 &  & \bm{1} & & 0 &\vdots &&\vdots
\\
\vdots &  &  & \vdots & \ddots & \vdots && &
\\
-l_0 & 0 & \dots & 0      &        & \bm{1} & 0& \dots & 0
\end{array}
\right],
$$
where, as before, the symbols $\bm{1}$ indicate vectors 
of length~$n_i$ with all entries equal to one.
\end{lemma}

\begin{proof}
Observe that the rows of $P_0$ generate a sublattice 
of finite index in the row lattice $P_1$.
Thus, we have a commutative diagram
$$ 
\xymatrix{
K_0 \ar[rr] 
\ar[rd] 
&
&
K_0/K_0^{\mathrm{tors}}
\\
&
\ZZ^{n+m}/\im(P_1^*) \ar[ur]
&
}
$$
It suffices to show, that $\ZZ^{n+m}/\im(P_1^*)$ is torsion free. 
Applying suitable elementary column operations to $P_1$,
reduces the problem to showing that for the 
$2 \times 3$ matrix
$$
\left[
\begin{array}{ccc}
\frac{\mathfrak{l}_0}{\gcd(\mathfrak{l}_0,\mathfrak{l}_1)} 
& 
\frac{\mathfrak{l}_1}{\gcd(\mathfrak{l}_0,\mathfrak{l}_1)}
& 
0 
\\[5pt]
\frac{\mathfrak{l}_0}{\gcd(\mathfrak{l}_0,\mathfrak{l}_2)} 
& 
0 
& 
\frac{\mathfrak{l}_2}{\gcd(\mathfrak{l}_0,\mathfrak{l}_1)} 
\end{array}
\right],
$$
all determinantal divisors equal one.
The entries of the above matrix are coprime 
and its $2\times 2$ minors are 
$$
\frac{\mathfrak{l}_0\mathfrak{l}_2}{\gcd(\mathfrak{l}_0,\mathfrak{l}_1)\gcd(\mathfrak{l}_0,\mathfrak{l}_2)}, 
\qquad
\frac{\mathfrak{l}_1\mathfrak{l}_2}{\gcd(\mathfrak{l}_0,\mathfrak{l}_1)\gcd(\mathfrak{l}_0,\mathfrak{l}_2)}, 
\qquad 
\frac{\mathfrak{l}_0\mathfrak{l}_1}{\gcd(\mathfrak{l}_0,\mathfrak{l}_1)\gcd(\mathfrak{l}_0,\mathfrak{l}_2)}.
$$
up to sign.
By assumption, we have 
$\gcd(\mathfrak{l}_1, \mathfrak{l}_2)
=
\mathfrak{l}$.
Consequently, we obtain
$$
\gcd(\mathfrak{l}_0\mathfrak{l}_2,\mathfrak{l}_0\mathfrak{l}_1,\mathfrak{l}_1\mathfrak{l}_2) 
\ = \ 
\gcd(\mathfrak{l}_0\mathfrak{l}, \mathfrak{l}_1\mathfrak{l}_2)
\ = \ 
\gcd(\mathfrak{l}_0,\mathfrak{l}_1)\gcd(\mathfrak{l}_0,\mathfrak{l}_2)
$$
and therefore the second determinantal divisor equals one. 
As remarked, the first one equals one as well
and the assertion follows.
\end{proof}

\begin{lemma}
\label{lem:numbercomp}
Let $R(A,P_0)$ be of Type~2
and $\b{X} = \Spec \, R(A,P_0)$.
Then, for any generator 
$T_{01}$ of $R(A,P_0)$,
we have 
$$ 
V(\b{X},T_{01})
\ \cong \
V(T_{01}) \cap V(T_1^{l_1}-T_{i}^{l_{i}}; \; i = 2, \ldots, r)
\ \subseteq \
\CC^{n+m}.
$$
In particular, the number of irreducible components of 
$V(\b{X},T_{01})$ equals the product of the invariant 
factors of the matrix
$$ 
\left[
\begin{array}{cccc}
-\mathfrak{l}_1 & \mathfrak{l}_2 &        & 0
\\
\vdots          &                & \ddots & 
\\
-\mathfrak{l}_1 &       0         &        &  \mathfrak{l}_{r}
\end{array}
\right].
$$
\end{lemma}

\begin{proof}
First observe that the ideal
$\bangle{T_{01},g_0,\ldots,g_{r-2}} \subseteq \CC[T_{ij},S_k]$
is generated by binomials which can be brought into 
the above form by scaling the variables appropriately.
Now consider the homomorphism of tori
$$ 
\pi \colon
\TT^{n_1+\ldots+n_r} \ \to \ \TT^{r-1},
\qquad
(t_1, \ldots, t_r) 
\ \mapsto \ 
\left( 
\frac{t_2^{l_2}}{t_1^{l_1}}, \ldots, \frac{t_r^{l_{r}}}{t_1^{l_1}}
\right).
$$
Then the number of connected components of $\ker(\pi)$ 
equals the product of the invariant factors of the above 
matrix.
Moreover, $\TT^{n_0-1} \times \ker(\pi) \times \TT^m $ 
is isomorphic to 
$V(\b{X},T_{01}) \cap \TT^{n+m}$.
Finally, one directly checks that $V(\b{X},T_{01})$ has 
no further irreducible components outside $\TT^{n+m}$.
\end{proof}

\begin{lemma}
\label{lem:numbercomp2}
Let $R(A,P_0)$ be of Type~2 and platonic.
Assume that $P_0$ is ordered. 
Then the number $c(i)$ 
of irreducible components of 
$V(\b{X},T_{ij})$ is given as 

\begin{center}
\begin{tabular}{c||c|c|c|c}
$i$ & $0$ & $1$ & $2$ & $\ge 3$
\\
\hline
$c(i)$
& 
$\gcd(\mathfrak{l}_1,\mathfrak{l}_2)$
& 
$\gcd(\mathfrak{l}_0,\mathfrak{l}_2)$
& 
$\gcd(\mathfrak{l}_0,\mathfrak{l}_1)$
&
$\mathfrak{l}^{2}\mathfrak{l}_{01}\mathfrak{l}_{02}\mathfrak{l}_{12}$
\end{tabular}
\end{center}
\end{lemma}

\begin{proof}
Suitable admissible operations turn
$T_{ij}$ to $T_{01}$.
Then the number of components is computed 
via Lemma~\ref{lem:numbercomp}.
\end{proof}

\begin{proposition}
\label{prop::isotropy}
Let $R(A,P_0)$ be of Type~2, platonic
and non-factorial.
Assume that $P_0$ is ordered and let 
$P_1$ be as in Lemma~\ref{lemma::Ptors}.
Set 
$$ 
n_{i,1}, \ldots, n_{i,c(i)} \ := \ n_i,
\qquad
l_{ij,1}, \ldots, l_{ij,c(i)}
\ := \ 
\gcd ( (P_{1})_{1,ij},\ldots,(P_{1})_{r,ij} ).
$$
The
$l_{i,\alpha} := (l_{i1,\alpha}, \ldots, l_{in_i,\alpha}) \in \ZZ^{n_{i,\alpha}}$
build up an $r' \times (n'+ m)$ matrix $P_0'$,
where $n' := c(0)n_0 + \ldots + c(r)n_r$.
With a suitable matrix $A'$, the following holds.
\begin{enumerate}
\item
The affine variety $\Spec \, R(A',P'_0)$ is the 
total coordinate space of the affine variety 
$\Spec \, R(A,P_0)$,
\item
The leading platonic triple (l.p.t.) of $R(A',P'_0)$ 
can be expressed in terms of that of $R(A,P_0)$ 
as 

\begin{center}
\begin{tabular}{c|c}
l.p.t. of $R(A,P_0)$ & l.p.t. of $R(A',P'_0)$
\\
\hline
$(4,3,2)$ 
& 
$(3,3,2)$
\\
$(3,3,2)$ 
& 
$(2,2,2)$
\\
$(y,2,2)$ 
& 
$(z,z,1)$ or $\left(\frac{y}{2},2,2\right)$
\\
$(x,y,1)$ 
& 
$\left(
\frac{x}{\gcd(\mathfrak{l}_0,\mathfrak{l}_1)}, 
\frac{y}{\gcd(\mathfrak{l}_0,\mathfrak{l}_1)},
1
\right)$
\\
\end{tabular}
\end{center}
\end{enumerate}
\end{proposition}

\begin{proof}
We compute the Cox ring of $\b{X} = \Spec \, R(A,P_0)$ 
according to~\cite[Thm.~4.4.1.6]{ArDeHaLa};
use Corollary 1.9 \cite{HaWr} to obtain
the statement given there also in the affine case.
That means that we have to figure out which
invariant divisors are identified under 
the rational map onto the curve $Y$ with 
function field $\CC(\b{X})^{H_0^0}$ 
and we have to determine the orders of 
isotropy groups of invariant divisors.

Let $P_1$ be as in Lemma~\ref{lemma::Ptors}.
Then the torus $H_0^0$ acts diagonally 
on $\CC^{n+m}$ 
with weights provided by the projection 
$Q_1 \colon \ZZ^{n+m} \to K_0^0$, where 
$K_0^0 = \ZZ^{n+m} / \im(P_1^*)$ 
equals the character group of~$H_0^0$.
Consider the commutative diagram
$$ 
\xymatrix{
{\b{X}_0}
\ar[d]
\ar@{}[r]|\subseteq
&
{\CC^{n+m}_0}
\ar[d]
\\
{\b{X}_0} / H_0^0
\ar[d]
\ar@{}[r]|\subseteq
&
{\CC^{n+m}_0} / H_0^0
\ar[d]
\\
Y
\ar@{}[r]|\subseteq
&
{\PP}
}
$$
where $\b{X}_0 \subseteq \b{X}$
and $\CC^{n+m}_0 \subseteq \CC^{n+m}$
denote the open $H_0^0$-invariant subsets 
obtained by removing all coordinate 
hyperplanes $V(S_k)$ 
and all intersections $V(T_{i_1j_1},T_{i_2j_2})$
with $(i_1,j_1)  \ne (i_2,j_2)$ from 
$\CC^{n+m}$.
Moreover, the geometric quotient spaces 
in the middle row are possibly non-separated
and the maps to the lower row are separation
morphisms.

We determine the orders of isotropy groups.
Every point in $\TT^{n+m}$ has trivial
$H_0^0$-isotropy.
Thus, we only have to look what happens on 
the sets $V(T_{ij}) \cap \CC^{n+m}_0$.
According to~\cite[Prop.~2.1.4.2]{ArDeHaLa}, 
the order of isotropy group of $H_0^0$ at any 
point $x \in V(T_{ij}) \cap \CC^{n+m}_0$ equals 
the greatest common divisor of the entries 
of the $ij$-th column of $P_1$:
$$ 
\vert H_{0,x}^0 \vert 
\ = \ 
l_{ij}'
\ := \ 
\gcd((P_{1})_{1,ij},\ldots,(P_{1})_{r,ij})
\quad 
\text{ for all }
x \in V(T_{ij}) \cap \CC^{n+m}_0.
$$

Now we figure out which $H_0^0$-invariant 
divisors of $\b{X}_0$ are identified under 
the map $\b{X}_0 \to Y$.
Lemma~\ref{lem:numbercomp2} provides us 
explicit numbers $c(0), \ldots, c(r)$ 
such that for fixed 
$i$ and $j = 1, \ldots, n_i$, we have
the decomposition into prime divisors
$$
V(\b{X},T_{ij})
\ = \ 
D_{ij,1} \cup \ldots \cup D_{ij, c(i)},
$$
in particular, the number $c(i)$ does not depend 
on the choice of $j$.
The components $D_{ij,1}, \ldots, D_{ij, c(i)}$
lie in the common affine chart 
$W_0 \subseteq \b{X}_0$ obtained by localizing 
at all $T_{i'j'}$ different from $T_{ij}$.
Their images thus lie in the affine 
chart $W_0/H_0^0 \subseteq \b{X}_0/H_0^0$.
Consequently, the $D_{ij,1}, \ldots, D_{ij, c(i)}$
have pairwise disjoint images under the composition
$\b{X}_0 \to \b{X}_0/H_0^0 \to Y$.

On the other hand, $V(\b{X},T_{ij})$ and  
$V(\b{X},T_{ij'})$ are identified isomorphically 
under the separation map $\b{X}_0/H_0^0 \to Y$
Thus, suitably numbering, we obtain for every 
$i$, and $\alpha = 1, \ldots, c(i)$ a chain 
$$ 
D_{i1,\alpha}, \ldots, D_{in_i, \alpha},
$$
of divisors identified under the morphism
$\b{X}_0/H_0^0 \to Y$.
The order of isotropy for any 
$x \in D_{ij,\alpha}$ equals $l_{ij}'$.
Now, using~\cite[Thm.~4.4.1.6]{ArDeHaLa}, 
we can compute the defining relations 
of the Cox ring of $\b{X}$,
which establishes the two assertions.
\end{proof}

\begin{remark}
\label{rem:ithelp}
Let $R(A,P_0)$ be a non factorial platonic 
ring with ordered $P_0$ and leading platonic 
triple $(l_{01},l_{11},l_{21})$. 
Denote by $R(A',P'_0)$ the Cox ring of $\Spec \, R(A,P_0)$. 
Then the exponents of the defining relations 
of $R(A',P'_0)$ are listed in the following table, 
where $\bm{1}_{n_1}$ denotes the vector of 
length $n_i$ with all entries equal to one.

\begin{center}

\setlength\extrarowheight{5pt}

\begin{tabular}{l|c}
leading plat. triple
& 
exponents in $R(A',P'_0)$
\\
\hline
$(4,3,2)$ 
& 
$l_1$, $l_1$, $l_0/2$, $\bm{1}_{n_2}$ 
and $2 \times \bm{1}_{n_i}$ for $i \geq 3$.
\\
\hline
$(3,3,2)$
& 
$l_2$, $l_2$, $l_2$, $\bm{1}_{n_0}$, $\bm{1}_{n_1}$ 
and $3 \times \bm{1}_{n_i}$ for $i \geq 3$.
\\
\hline
$(x,2,2)$ and $\mathfrak{l}=2$ 
& 
$l_0/2$, $l_0/2$, $2 \times \bm{1}_{n_1}$,
$2 \times \bm{1}_{n_2}$, and $4 \times \bm{1}_{n_i}$ for $i \geq 3$. 
\\
\hline
$(x,2,2)$ and $2 \nmid \mathfrak{l}_0$ 
& 
$l_0$, $l_0$, $\bm{1}_{n_1}$, $\bm{1}_{n_2}$
and $2 \times \bm{1}_{n_i}$ for $i \geq 3$.
\\
\hline
$(x,2,2)$ and $\mathfrak{l}_{2}=1$ 
& 
$l_0/2$, $l_2$, $l_2$, $\bm{1}_{n_1}$ and 
$2 \times \bm{1}_{n_i}$ for $i \geq 3$.
\\
\hline
$(x,y,1)$  
& 
$\frac{l_0}{\gcd(\mathfrak{l}_0,\mathfrak{l}_1)}$, 
$\frac{l_1}{\gcd(\mathfrak{l}_0,\mathfrak{l}_1)}$, 
$\gcd(\mathfrak{l}_0, \mathfrak{l}_1) \times \bm{1}_{n_i}$ 
for $i \geq 2$.
\end{tabular}

\end{center}

\end{remark}

\begin{proof}[Proof of Theorem~\ref{theorem::CoxIteration}]
We start with a rational, normal, affine, log 
terminal $X_1$ of complexity one.
According to Theorem~\ref{thm:affltcharintro}, 
the Cox ring $R_2$ of $X_1$ is a platonic ring. 
If the greatest common divisors of pairs 
$\mathfrak{l}_i,\mathfrak{l}_j$ of $R_2$ 
all equal one, then $R_2$ is factorial by~\cite[Thm.~1.1]{HaHe}
and we are done.
If not, then we pass to the Cox ring $R_3$ of $X_2:= \Spec \, R_2$ 
and so on.
Proposition~\ref{prop::isotropy} ensures that 
this procedure terminates with a factorial 
platonic ring $R_p$.
\end{proof}

\begin{proof}[Proof of Theorem~\ref{theorem::quotient}]
Let $X_1$ be any rational, normal, affine variety 
with a torus action of complexity one of Type~2
and at most log terminal singularities.
Then Theorem~\ref{theorem::CoxIteration} provides 
us with a chain of quotients
$$
\xymatrix{ 
X_p 
\ar[r]^{ \quot H_{p-1} \ } 
&
X_{p-1} 
\ar[r]^{\quot H_{p-2} \ } 
&
\quad \ldots \quad 
\ar[r]^{ \quot H_3 \ } 
&
X_3 
\ar[r]^{\quot H_{2} \ }
&
X_2 
\ar[r]^{\quot H_1 \ }
&
X_1 },
$$ 
such that $X_i = \Spec(R_i)$ holds with a platonic 
ring $R_i$ when $i \ge 2$, 
the ring $R_p$ is factorial and 
each $X_{i+1} \to X_{i}$ is the total 
coordinate space.
The idea is to construct stepwise 
solvable linear algebraic groups 
$G_i \subseteq \Aut(X_{i+1})$ acting 
algebraically on~$X_{i+1}$ such 
that the unit component $G_i^0 \subseteq G_i$ 
is a torus, $G_i$ contains $H_{i}$ as a 
normal subgroup, 
$G_{i-1} = G_i/H_{i}$ holds 
and we have $G_1 = H_1$.

Start with $G_1 := H_1$, acting on $X_2$.
According to~\cite[Thm.~2.4.3.2]{ArDeHaLa},
there exists an (effective) action of a 
torus $\mathcal{G}_1$ on $X_3$ lifting the action of
$G_1^0$ on $X_2$ and commuting with the action 
of $H_2$ on $X_3$.
Moreover,~\cite[Thm.~5.1]{ArGa} provides us with 
an exact sequence of groups 
$$
\xymatrix{ 
1
\ar[r] 
&
H_2
\ar[r]
&
\mathrm{Aut}(X_3,H_2)
\ar[r]^{\pi}
&
\mathrm{Aut}(X_2)
\ar[r] 
& 
1
},
$$
where $\mathrm{Aut}(X_3,H_2)$ denotes the 
group of automorphisms of $X_3$ normalizing 
the quasitorus $H_2$. 
Set $G_2 := \pi^{-1}(G_1)$.
Then $H_2^0\mathcal{G}_1$, 
as a factor group of the torus 
$H_2^0 \times \mathcal{G}_1$ by a closed subgroup, 
is an algebraic torus and it is of finite index 
in $G_2$.
Thus, $G_2$ is an affine algebraic group with  
$G_2^0 = H_2^0\mathcal{G}_1$ being a torus.
By construction, $H_2 \subseteq G_2$ is 
the kernel of $\alpha_1 := \pi \vert_{G_2}$ and hence
a normal subgroup.
Moreover, $G_2$ is solvable and acts 
algebraically on $X_3$.
Iterating this procedure gives a sequence
$$ 
\xymatrix{ 
G_{p -1}
\ar[r]^{\alpha_{p-2} \ } 
&
G_{p - 2}
\ar[r]^{\alpha_{p-3} \ } 
&
\quad \ldots \quad 
\ar[r]^{\alpha_{2} \ } 
&
G_{2}
\ar[r]^{\alpha_{1} \ } 
&
G_{1}
\ar[r]^{\alpha_{0} \ } 
&
1
}
$$
of group epimorphisms, where, as wanted, $G_i$ 
is a solvable reductive group 
acting algebraically on $X_{i+1}$ such that 
$H_i = \ker(\alpha_{i-1})$ is the characteristic 
quasitorus of $X_i$.
In particular, the group $G := G_{p-1} \subseteq \Aut(X_p)$
satisfies the first assertion of the theorem.

We turn to the second assertion.
From~\cite[Prop.~1.6.1.6]{ArDeHaLa}, we 
infer that $G_1 = H_1$ acts freely on the 
preimage $U_2 \subseteq X_2$ of the 
set of smooth points $U_1 \subseteq X_1$
and moreover, the complement $X_2 \setminus U_2$ 
is of codimension at least two in $X_2$.
Let $U_3 \subseteq X_3$ be the preimage of 
$U_2 \subseteq X_2$.
Again, the complement of $U_3$ is of 
codimension at least two in $X_3$ and,
as $U_2$ consists of smooth points of 
$X_2$, the quasitorus $H_2$ acts freely 
on $U_3$. 
Because of $G_2/H_2 = G_1$, we conclude that 
$U_3$ is $G_2$-invariant and $G_2$ acts
freely on $U_2$.
Repeating this procedure, we end up with 
an open set $U_p \subseteq X_p$ having complement 
of codimension at least two such that
$G$ acts freely on $U_p$.
Thus, $G$ acts strongly stably on $X_p$.
Now consider
$$ 
G = \mathcal{D}_0
\ \supseteq \ 
\mathcal{D}_1
\ \supseteq \ 
\ldots 
\ \supseteq \ 
\mathcal{D}_{p-2} 
\ \supseteq \ 
\mathcal{D}_{p-1}
= 1,
\qquad
\mathcal{D}_i := \ker(\alpha_i \circ \ldots \circ \alpha_{p-2}).
$$
Then we have $X_i = X_p \quot \mathcal{D}_{i-1}$
and $H_i = \mathcal{D}_{i-1} / \mathcal{D}_{i}$.
Moreover for each $\mathcal{D}_i$, its action on
$X_p$ is strongly stable, as remarked before,
and  $X_{p}$ is $G$-factorial because it
is factorial.
Using~\cite[Prop.~3.5]{ArGa}, we obtain 
a commutative diagram
$$ 
\xymatrix{
X_p \quot [\mathcal{D}_i,\mathcal{D}_i]
\ar[dr]_{\quot \mathcal{D}_i / [\mathcal{D}_i,\mathcal{D}_i] \qquad}
\ar[rr]^{\beta}
&
&
X_p \quot \mathcal{D}_{i+1}
\ar[dl]^{\quot \mathcal{D}_i / \mathcal{D}_{i+1}}
\\
&
X_p \quot \mathcal{D}_i
&
}
$$
where the left downward map is a total coordinate space.
As $\mathcal{D}_i / \mathcal{D}_{i+1} = H_{i+1}$ 
is abelian, $[\mathcal{D}_i,\mathcal{D}_i]$
is contained in $\mathcal{D}_{i+1}$ and
we have the horizontal morphism $\beta$.
Since the right hand side is a total coordinate space 
as well, we infer from~\cite[Sec.~1.6.4]{ArDeHaLa} that 
$\beta$ is an isomorphism.
This implies $\mathcal{D}_{i+1} = [\mathcal{D}_i,\mathcal{D}_i]$,
proving the second assertion.
\end{proof}

\section{Compound du Val singularities}
\label{sec:cDV}

Between the Gorenstein terminal and canonical threefold
singularities lie the compound du Val singularities, 
introduced by Miles Reid in~\cite{Re}, see 
also~\cite{Re1, Ma, KoMo}.
We discuss compound du Val singularities in the context 
of $T$-varieties of complexity one and provide first 
constraints on the defining data for affine threefolds,
preparing the proof of our classification results.

\begin{definition}
\label{def:cDV}
\cite[Def.~2.1]{Re},~\cite[Thm.~5.34, Cor.~2.3.2]{KoMo}.
A normal, canonical,
 Gorenstein threefold singularity $x \in X$ 
is called \emph{compound du Val}, if one of the 
following equivalent criteria is satisfied:
\begin{enumerate}
\item 
For a general hypersurface $Y  \subseteq X$ with $x \in Y$,
the point $x$ is a du Val surface singularity of $Y$.
\item 
Near $x$, the threefold $X$ is analytically isomorphic 
to a hypersurface of the following shape
$$
V(f(T_1,T_2,T_3) +  g(T_1,T_2,T_3,T_4) \, T_4)
\ \subseteq \ 
\CC^4,
$$
where $f$ is a defining polynomial for a du Val 
surface singularity in $\CC^3$
and~$g$ is any polynomial in $T_1,T_2,T_3,T_4$.
\item 
For every
resolution $\varphi \colon X' \to X$ 
of singularities and every irreducible 
exceptional divisor $E \subseteq \varphi^{-1}(x)$, 
the discrepancy of $E$ is greater than zero.
\item 
There is a
resolution $\varphi \colon X' \to X$ 
of singularities such that every irreducible 
exceptional divisor $E \subseteq \varphi^{-1}(x)$
is of discrepancy greater than zero.
\end{enumerate}
\end{definition}

For an affine toric threefold $X$, Condition~\ref{def:cDV}~(iv)
means the following: $X$ is defined by the cone over 
$\triangle \times \{1\}$ with a hollow lattice polytope 
$\triangle \subseteq \QQ^2$, where hollow means that $\triangle$
has no lattice points in its interior.
Based on this characterization, one obtains the list of 
toric compound du Val singularities provided in~\cite{Da}:

\begin{proposition}
\label{prop:toriccDV}
Let $X$ be an affine toric variety with a compound du Val 
singularity.
Then $X \cong X(\sigma)$ holds with a cone 
$\sigma \subseteq \QQ^3$ generated 
by the columns of one of the following matrices
$$
(1)
\
\begin{bmatrix}
 0  & 0 & k \\
 0  & 1 & 0 \\
 1 & 1 & 1
\end{bmatrix},
\ k \in \ZZ_{\ge 2},
\quad
(2)
\
\begin{bmatrix}
 0  & 0 & k_1 & k_2 \\
 0  & 1 & 0   & 1 \\
 1 & 1 & 1 & 1
 \end{bmatrix},
\ k_1,k_2 \in \ZZ_{\ge 1},
\quad
(3)
\
\begin{bmatrix}
 0  & 0 & 2 \\
 0  & 2 & 0   \\
 1 & 1 & 1
 \end{bmatrix}.
$$
\end{proposition} 

\begin{proof}
After removing the third row from the matrices, we find 
in their columns the vertices of the hollow polytopes 
$\triangle \subseteq \QQ^2$; see~\cite{RaSt}. 
\end{proof}

We turn to affine $T$-varieties $X$ of complexity one.
As the toric case is settled, we can concentrate on
the varieties $X = X(A,P)$ of Type~2.
The basic tool is the anticanonical complex $A_X^c$,
described in Proposition~\ref{prop:acancompstruct}.
The following statement specifies a bit more.

\begin{proposition}
\label{prop:roofacan}
Let $X = X(A,P)$ be an affine Gorenstein,
log-terminal threefold of Type~2 such that
$P$ is in the form of Proposition~\ref{prop:zeta}.
Consider the intersections
$$
\partial A_X^c(\lambda)
\, := \, 
\partial A_X^c \cap \lambda,
\quad
\partial A_X^c(\lambda_i)
\, := \, 
\partial A_X^c \cap \lambda_i,
\quad 
\partial A_X^c (\lambda_i, \tau) 
\, := \,
\partial A_X^c(\lambda_i) \cap \tau,
$$
where $\partial A_X^c$ is the relative 
boundary of the anticanonical complex,
$\lambda \subseteq \trop(X)$ 
the lineality part,
$\lambda_0, \ldots, \lambda_r \subseteq \trop(X)$
are the leaves
and $\tau$ is any $P$-elementary cone.
\begin{enumerate}
\item
Let $x_1, \ldots, x_{r+2}$ be the standard 
coordinates on the column space $\QQ^{r+2}$
of~$P$ and set $x_0 := - x_1 -  \ldots - x_r$. 
Then $x_i,x_{r+1},x_{r+2}$ are linear coordinates 
on the three-dimensional vector space 
$\Lin_\QQ(\lambda_i)$ and we have 
$$
\partial A_X^c(\lambda_i)
\ = \
A_X^c
\cap
\lambda_i 
\cap 
\mathcal{H}_i
\ \subseteq \
\Lin_\QQ(\lambda_i)
$$
with the plane 
$
\mathcal{H}_i 
:= 
V(\zeta_X x_{r+2} + \mu_i x_{i} - \imath_X)
\subseteq 
\Lin_\QQ(\lambda_i)
$,
where $\mu_i$ is the integer defined
in Remark~\ref{rem:zetaPprops}.
In particular, for fixed $i$, 
the columns $v_{ij}$ of $P$ lie 
on the half plane $\lambda_i \cap \mathcal{H}_i$.
\item
The set $A_X^c \cap \tau$ is a two-dimensional 
and $\partial A_X^c \cap \tau$ a 
one-dimensional polyhedral complex. 
Furthermore, 
$\partial A_X^c (\lambda_i, \tau)$
is a line segment.
\end{enumerate}
\end{proposition}

\begin{proof}
We show~(i). 
Let $\sigma \subseteq \QQ^{r+2}$ be the cone
over the columns of $P$.
Then the set $\partial A_X^c(\lambda_i)$ 
equals $\partial A_X^c \cap \sigma \cap \lambda_i$.
By the assumption on $P$, the equation from 
Proposition~\ref{prop:acancompstruct}~(ii) 
gives the assertion.

For~(iii), write 
$\tau=\cone(w_0,\ldots,w_r)$ with $w_i \in \lambda_i$. 
Observe that $A_X^c \cap \tau  \cap \lambda_i$ 
has the vertices $0$, $w_i$, $v(\tau)'$ and is 
thus two-dimensional. 
Only  $w_i$ and $v(\tau)'$ satisfy the equation
$\zeta_X x_{r+2} + \mu_i x_{i} = \imath_X$. 
Thus $A_X^c \cap \tau$ is two-dimensional and
$\partial A_X^c \cap \tau$ as well 
as $\partial A_X^c (\lambda_i, \tau)$ 
are one-dimensional.
\end{proof}

The following figures visualize the situation of 
Proposition~\ref{prop:roofacan} for the case $r=2$.
The first one shows the leaves $\lambda_i$,
the second one the half planes $\lambda_i \cap \mathcal{H}_i$,
the third one all $A_X^c(\lambda_i)$
and the last one all $A_X^c (\lambda_i, \tau)$
for a given $P$-elementary cone $\tau$.

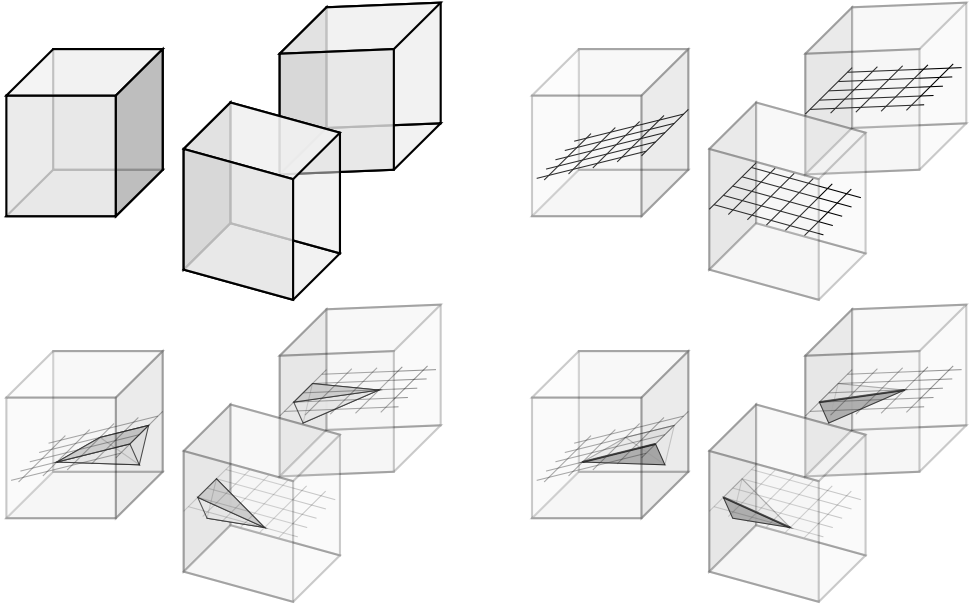
\begin{center}
\begin{tikzpicture}[scale=0.8]


\draw[thick, draw=black, fill=gray!20!, fill opacity=0.5] (0,-1,-1)--(0,-1,1)--(-1.8,-1,1)--(-1.8,-1,-1)--cycle;
\draw[thick, draw=black, fill=gray!20!, fill opacity=0.5] (0,-1,-1)--(0,1,-1)--(-1.8,1,-1)--(-1.8,-1,-1)--cycle;

\draw[thick, draw=black, fill=gray!20!, fill opacity=0.80] (0,-1,1)--(0,1,1)--(-1.8,1,1)--(-1.8,-1,1)--cycle;
\draw[thick, draw=black, fill=gray!10!, fill opacity=0.80] (0,1,1)--(0,1,-1)--(-1.8,1,-1)--(-1.8,1,1)--cycle;

\draw[thick, draw=black, fill=gray!60!, fill opacity=0.80] (0,-1,-1)--(0,-1,1)--(0,1,1)--(0,1,-1)--cycle;

\draw[thick, draw=black, fill=gray!90!, fill opacity=0.80]
(2,-1,-2.8)--(2,-1,-0.8)--(2,1,-0.8)--(2,1,-2.8)--cycle;

\draw[thick, draw=black, fill=gray!20!, fill opacity=0.5] (2,-1,-2.8)--(2,-1,-0.8)--(3.8,-1,-1)--(3.8,-1,-3)--cycle;
\draw[thick, draw=black, fill=gray!20!, fill opacity=0.5] (2,-1,-2.8)--(2,1,-2.8)--(3.8,1,-3)--(3.8,-1,-3)--cycle;

\draw[thick, draw=black, fill=gray!10!, fill opacity=0.80] (2,1,-0.8)--(2,1,-2.8)--(3.8,1,-3)--(3.8,1,-1)--cycle;
\draw[thick, draw=black, fill=gray!20!, fill opacity=0.80] (2,1,-0.8)--(2,-1,-0.8)--(3.8,-1,-1)--(3.8,1,-1)--cycle;

\draw[thick, draw=black, fill=gray!90!, fill opacity=0.80] (2,-1,1.3)--(2,-1,3.3)--(2,1,3.3)--(2,1,1.3)--cycle;

\draw[thick, draw=black, fill=gray!20!, fill opacity=0.5] (2,-1,1.3)--(2,-1,3.3)--(4.3,-1,4.6)--(4.3,-1,2.6)--cycle;
\draw[thick, draw=black, fill=gray!20!, fill opacity=0.5] (2,-1,1.3)--(2,1,1.3)--(4.3,1,2.6)--(4.3,-1,2.6)--cycle;

\draw[thick, draw=black, fill=gray!10!, fill opacity=0.80] (2,1,3.3)--(2,1,1.3)--(4.3,1,2.6)--(4.3,1,4.6)--cycle;
\draw[thick, draw=black, fill=gray!20!, fill opacity=0.80] (2,1,3.3)--(2,-1,3.3)--(4.3,-1,4.6)--(4.3,1,4.6)--cycle;

\end{tikzpicture}
\hfill
\begin{tikzpicture}[scale=0.8]


\draw[thick, draw=black, fill=gray!20!, opacity=0.2] (0,-1,-1)--(0,-1,1)--(-1.8,-1,1)--(-1.8,-1,-1)--cycle;
\draw[thick, draw=black, fill=gray!20!, opacity=0.2] (0,-1,-1)--(0,1,-1)--(-1.8,1,-1)--(-1.8,-1,-1)--cycle;

\foreach \i in {0,-0.4,...,-1.6}
	{
\draw[draw=black] (\i,\i/4,1)--(\i,\i/4,-1);	
	}
	
\foreach \i in {0.8,0.4,...,-1.2}
	{
\draw[draw=black] (0,0,\i)--(-1.8,-0.45,\i);	
	}

\draw[thick, draw=black, fill=gray!20!, opacity=0.2] (0,-1,1)--(0,1,1)--(-1.8,1,1)--(-1.8,-1,1)--cycle;
\draw[thick, draw=black, fill=gray!10!, opacity=0.2] (0,1,1)--(0,1,-1)--(-1.8,1,-1)--(-1.8,1,1)--cycle;

\draw[thick, draw=black, fill=gray!60!, opacity=0.2] (0,-1,-1)--(0,-1,1)--(0,1,1)--(0,1,-1)--cycle;

\draw[thick, draw=black, fill=gray!90!, opacity=0.2]
(2,-1,-2.8)--(2,-1,-0.8)--(2,1,-0.8)--(2,1,-2.8)--cycle;

\draw[thick, draw=black, fill=gray!20!, opacity=0.2] (2,-1,-2.8)--(2,-1,-0.8)--(3.8,-1,-1)--(3.8,-1,-3)--cycle;
\draw[thick, draw=black, fill=gray!20!, opacity=0.2] (2,-1,-2.8)--(2,1,-2.8)--(3.8,1,-3)--(3.8,-1,-3)--cycle;

\foreach \i in {0,0.4,...,1.6}
	{
\draw[draw=black] (2+\i,0,-0.8-\i*0.1)--(2+\i,0,-2.8-\i*0.1);	
	}
	
\foreach \i in {-1,-1.4,...,-2.6}
	{
\draw[draw=black] (2,0,\i)--(3.8,0,-0.2+\i);	
	}

\draw[thick, draw=black, fill=gray!10!, opacity=0.2] (2,1,-0.8)--(2,1,-2.8)--(3.8,1,-3)--(3.8,1,-1)--cycle;
\draw[thick, draw=black, fill=gray!20!, opacity=0.2] (2,1,-0.8)--(2,-1,-0.8)--(3.8,-1,-1)--(3.8,1,-1)--cycle;

\draw[thick, draw=black, fill=gray!90!, opacity=0.2] (2,-1,1.3)--(2,-1,3.3)--(2,1,3.3)--(2,1,1.3)--cycle;

\draw[thick, draw=black, fill=gray!20!, opacity=0.2] (2,-1,1.3)--(2,-1,3.3)--(4.3,-1,4.6)--(4.3,-1,2.6)--cycle;
\draw[thick, draw=black, fill=gray!20!, opacity=0.2] (2,-1,1.3)--(2,1,1.3)--(4.3,1,2.6)--(4.3,-1,2.6)--cycle;

\foreach \i in {0,0.4,...,2}
	{
\draw[draw=black] (2+\i,0,1.3+\i*0.5652174)--(2+\i,0,3.3+\i*0.5652174);	
	}
	
\foreach \i in {1.5,1.9,...,3.1}
	{
\draw[draw=black] (2,0,\i)--(4.3,0,1.3+\i);	
	}

\draw[thick, draw=black, fill=gray!10!, opacity=0.2] (2,1,3.3)--(2,1,1.3)--(4.3,1,2.6)--(4.3,1,4.6)--cycle;
\draw[thick, draw=black, fill=gray!20!, opacity=0.2] (2,1,3.3)--(2,-1,3.3)--(4.3,-1,4.6)--(4.3,1,4.6)--cycle;

\end{tikzpicture}
\hfill
\begin{tikzpicture}[scale=0.8]


\draw[thick, draw=black, fill=gray!20!, opacity=0.2] (0,-1,-1)--(0,-1,1)--(-1.8,-1,1)--(-1.8,-1,-1)--cycle;
\draw[thick, draw=black, fill=gray!20!, opacity=0.2] (0,-1,-1)--(0,1,-1)--(-1.8,1,-1)--(-1.8,-1,-1)--cycle;

\draw[draw=black, fill=gray!20!, opacity=0.8] (0,0,-0.4)--(-0.8,-0.2,-0.4)--(0,-0.5,0)--cycle;
\draw[draw=black, fill=gray!20!, opacity=0.8] (-0.8,-0.2,-0.4)--(-1.2,-0.3,0.4)--(0,-0.5,0)--cycle;
\draw[draw=black, fill=gray!20!, opacity=0.8] (0,0,0.4)--(-1.2,-0.3,0.4)--(0,-0.5,0)--cycle;
\draw[draw=black, fill=gray!50!, opacity=0.8] (0,0,-0.4)--(-0.8,-0.2,-0.4)--(-1.2,-0.3,0.4)--(0,0,0.4)--cycle;

\foreach \i in {0,-0.4,...,-1.6}
	{
\draw[draw=black, opacity=0.4] (\i,\i/4,1)--(\i,\i/4,-1);	
	}
	
\foreach \i in {0.8,0.4,...,-1.2}
	{
\draw[draw=black, opacity=0.4] (0,0,\i)--(-1.8,-0.45,\i);	
	}

\draw[thick, draw=black, fill=gray!20!, opacity=0.2] (0,-1,1)--(0,1,1)--(-1.8,1,1)--(-1.8,-1,1)--cycle;
\draw[thick, draw=black, fill=gray!10!, opacity=0.2] (0,1,1)--(0,1,-1)--(-1.8,1,-1)--(-1.8,1,1)--cycle;

\draw[thick, draw=black, fill=gray!60!, opacity=0.2] (0,-1,-1)--(0,-1,1)--(0,1,1)--(0,1,-1)--cycle;

\draw[thick, draw=black, fill=gray!90!, opacity=0.2]
(2,-1,-2.8)--(2,-1,-0.8)--(2,1,-0.8)--(2,1,-2.8)--cycle;

\draw[thick, draw=black, fill=gray!20!, opacity=0.2] (2,-1,-2.8)--(2,-1,-0.8)--(3.8,-1,-1)--(3.8,-1,-3)--cycle;
\draw[thick, draw=black, fill=gray!20!, opacity=0.2] (2,-1,-2.8)--(2,1,-2.8)--(3.8,1,-3)--(3.8,-1,-3)--cycle;

\draw[draw=black, fill=gray!20!, opacity=0.8] (2,0,-2.2)--(3.2,0,-1.92)--(2,-0.5,-1.8)--cycle;
\draw[draw=black, fill=gray!20!, opacity=0.8] (2,0,-1.4)--(3.2,0,-1.92)--(2,-0.5,-1.8)--cycle;
\draw[draw=black, fill=gray!50!, opacity=0.8] (2,0,-2.2)--(3.2,0,-1.92)--(2,0,-1.4)--cycle;

\foreach \i in {0,0.4,...,1.6}
	{
\draw[draw=black, opacity=0.4] (2+\i,0,-0.8-\i*0.1)--(2+\i,0,-2.8-\i*0.1);	
	}
	
\foreach \i in {-1,-1.4,...,-2.6}
	{
\draw[draw=black, opacity=0.4] (2,0,\i)--(3.8,0,-0.2+\i);	
	}

\draw[thick, draw=black, fill=gray!10!, opacity=0.2] (2,1,-0.8)--(2,1,-2.8)--(3.8,1,-3)--(3.8,1,-1)--cycle;
\draw[thick, draw=black, fill=gray!20!, opacity=0.2] (2,1,-0.8)--(2,-1,-0.8)--(3.8,-1,-1)--(3.8,1,-1)--cycle;

\draw[thick, draw=black, fill=gray!90!, opacity=0.2] (2,-1,1.3)--(2,-1,3.3)--(2,1,3.3)--(2,1,1.3)--cycle;

\draw[thick, draw=black, fill=gray!20!, opacity=0.2] (2,-1,1.3)--(2,-1,3.3)--(4.3,-1,4.6)--(4.3,-1,2.6)--cycle;
\draw[thick, draw=black, fill=gray!20!, opacity=0.2] (2,-1,1.3)--(2,1,1.3)--(4.3,1,2.6)--(4.3,-1,2.6)--cycle;

\draw[draw=black, fill=gray!20!, opacity=0.8] (2,0,1.9)--(3.6,0,4)--(2,-0.5,2.3)--cycle;
\draw[draw=black, fill=gray!20!, opacity=0.8] (2,0,2.7)--(3.6,0,4)--(2,-0.5,2.3)--cycle;
\draw[draw=black, fill=gray!50!, opacity=0.8] (2,0,2.7)--(3.6,0,4)--(2,0,1.9)--cycle;

\foreach \i in {0,0.4,...,2}
	{
\draw[draw=black, opacity=0.2] (2+\i,0,1.3+\i*0.5652174)--(2+\i,0,3.3+\i*0.5652174);	
	}
	
\foreach \i in {1.5,1.9,...,3.1}
	{
\draw[draw=black, opacity=0.2] (2,0,\i)--(4.3,0,1.3+\i);	
	}

\draw[thick, draw=black, fill=gray!10!, opacity=0.2] (2,1,3.3)--(2,1,1.3)--(4.3,1,2.6)--(4.3,1,4.6)--cycle;
\draw[thick, draw=black, fill=gray!20!, opacity=0.2] (2,1,3.3)--(2,-1,3.3)--(4.3,-1,4.6)--(4.3,1,4.6)--cycle;

\end{tikzpicture}
\hfill
\begin{tikzpicture}[scale=0.8]


\draw[thick, draw=black, fill=gray!20!, opacity=0.2] (0,-1,-1)--(0,-1,1)--(-1.8,-1,1)--(-1.8,-1,-1)--cycle;
\draw[thick, draw=black, fill=gray!20!, opacity=0.2] (0,-1,-1)--(0,1,-1)--(-1.8,1,-1)--(-1.8,-1,-1)--cycle;

\draw[draw=black, fill=gray!20!, opacity=0.2] (0,0,-0.4)--(-0.8,-0.2,-0.4)--(0,-0.5,0)--cycle;
\draw[draw=black, fill=gray!20!, opacity=0.2] (-0.8,-0.2,-0.4)--(-1.2,-0.3,0.4)--(0,-0.5,0)--cycle;
\draw[draw=black, fill=gray!80!, opacity=0.9] (0,0,0.4)--(-1.2,-0.3,0.4)--(0,-0.5,0)--cycle;
\draw[thick, draw=black, opacity=0.9] (0,0,0.4)--(-1.2,-0.3,0.4);
\draw[draw=black, fill=gray!50!, opacity=0.2] (0,0,-0.4)--(-0.8,-0.2,-0.4)--(-1.2,-0.3,0.4)--(0,0,0.4)--cycle;

\foreach \i in {0,-0.4,...,-1.6}
	{
\draw[draw=black, opacity=0.4] (\i,\i/4,1)--(\i,\i/4,-1);	
	}
	
\foreach \i in {0.8,0.4,...,-1.2}
	{
\draw[draw=black, opacity=0.4] (0,0,\i)--(-1.8,-0.45,\i);	
	}

\draw[thick, draw=black, fill=gray!20!, opacity=0.2] (0,-1,1)--(0,1,1)--(-1.8,1,1)--(-1.8,-1,1)--cycle;
\draw[thick, draw=black, fill=gray!10!, opacity=0.2] (0,1,1)--(0,1,-1)--(-1.8,1,-1)--(-1.8,1,1)--cycle;

\draw[thick, draw=black, fill=gray!60!, opacity=0.2] (0,-1,-1)--(0,-1,1)--(0,1,1)--(0,1,-1)--cycle;

\draw[thick, draw=black, fill=gray!90!, opacity=0.2]
(2,-1,-2.8)--(2,-1,-0.8)--(2,1,-0.8)--(2,1,-2.8)--cycle;

\draw[thick, draw=black, fill=gray!20!, opacity=0.2] (2,-1,-2.8)--(2,-1,-0.8)--(3.8,-1,-1)--(3.8,-1,-3)--cycle;
\draw[thick, draw=black, fill=gray!20!, opacity=0.2] (2,-1,-2.8)--(2,1,-2.8)--(3.8,1,-3)--(3.8,-1,-3)--cycle;

\draw[draw=black, fill=gray!20!, opacity=0.2] (2,0,-2.2)--(3.2,0,-1.92)--(2,-0.5,-1.8)--cycle;
\draw[draw=black, fill=gray!80!, opacity=0.9] (2,0,-1.4)--(3.2,0,-1.92)--(2,-0.5,-1.8)--cycle;
\draw[thick, draw=black, opacity=0.9] (2,0,-1.4)--(3.2,0,-1.92);
\draw[draw=black, fill=gray!50!, opacity=0.2] (2,0,-2.2)--(3.2,0,-1.92)--(2,0,-1.4)--cycle;

\foreach \i in {0,0.4,...,1.6}
	{
\draw[draw=black, opacity=0.4] (2+\i,0,-0.8-\i*0.1)--(2+\i,0,-2.8-\i*0.1);	
	}
	
\foreach \i in {-1,-1.4,...,-2.6}
	{
\draw[draw=black, opacity=0.4] (2,0,\i)--(3.8,0,-0.2+\i);	
	}

\draw[thick, draw=black, fill=gray!10!, opacity=0.2] (2,1,-0.8)--(2,1,-2.8)--(3.8,1,-3)--(3.8,1,-1)--cycle;
\draw[thick, draw=black, fill=gray!20!, opacity=0.2] (2,1,-0.8)--(2,-1,-0.8)--(3.8,-1,-1)--(3.8,1,-1)--cycle;

\draw[thick, draw=black, fill=gray!90!, opacity=0.2] (2,-1,1.3)--(2,-1,3.3)--(2,1,3.3)--(2,1,1.3)--cycle;

\draw[thick, draw=black, fill=gray!20!, opacity=0.2] (2,-1,1.3)--(2,-1,3.3)--(4.3,-1,4.6)--(4.3,-1,2.6)--cycle;
\draw[thick, draw=black, fill=gray!20!, opacity=0.2] (2,-1,1.3)--(2,1,1.3)--(4.3,1,2.6)--(4.3,-1,2.6)--cycle;

\draw[draw=black, fill=gray!20!, opacity=0.2] (2,0,1.9)--(3.6,0,4)--(2,-0.5,2.3)--cycle;
\draw[draw=black, fill=gray!80!, opacity=0.9] (2,0,2.7)--(3.6,0,4)--(2,-0.5,2.3)--cycle;
\draw[thick, draw=black, opacity=0.9] (2,0,2.7)--(3.6,0,4);
\draw[draw=black, fill=gray!50!, opacity=0.2] (2,0,2.7)--(3.6,0,4)--(2,0,1.9)--cycle;

\foreach \i in {0,0.4,...,2}
	{
\draw[draw=black, opacity=0.2] (2+\i,0,1.3+\i*0.5652174)--(2+\i,0,3.3+\i*0.5652174);	
	}
	
\foreach \i in {1.5,1.9,...,3.1}
	{
\draw[draw=black, opacity=0.2] (2,0,\i)--(4.3,0,1.3+\i);	
	}

\draw[thick, draw=black, fill=gray!10!, opacity=0.2] (2,1,3.3)--(2,1,1.3)--(4.3,1,2.6)--(4.3,1,4.6)--cycle;
\draw[thick, draw=black, fill=gray!20!, opacity=0.2] (2,1,3.3)--(2,-1,3.3)--(4.3,-1,4.6)--(4.3,1,4.6)--cycle;

\end{tikzpicture}
\end{center}

The following statement shows that the relative boundary 
$\partial A_X^c$ of the anticanonical complex replaces 
the lattice polytope $\triangle$ from the toric setting 
discussed before.

\begin{proposition}
\label{prop:cpl1hollow}
Let $X = X(A,P)$ be an affine Gorenstein,
log-terminal threefold of Type~2.
Then $X$ has at most compound du Val 
singularities if and only if 
there are no integral points in the relative 
interior of $\partial A_X^c$.
\end{proposition}

\begin{lemma} 
\label{lem:bigsingcpl1}
Let $X = X(A,P,\Phi)$, denote by $\Sigma$ the fan of
the minimal toric ambient variety $Z$ of $X$
and let $\sigma \in \Sigma$ be a big cone.
\begin{enumerate}
\item
The toric orbit $T_Z \cdot z_\sigma \subseteq Z$ 
corresponding to the cone $\sigma \in \Sigma$ is 
contained in $X \subseteq Z$. 
\item
If $T_Z \cdot z_\sigma \subseteq X$ contains a 
singular point of $X$, then every point of 
$T_Z \cdot z_\sigma$ is singular in $X$.
\end{enumerate}
\end{lemma}

\begin{proof} 
We show~(i).
By the structure of the defining relations $g_i$, 
the corresponding statement holds for 
$\b{X} \subseteq \b{Z} = \CC^{n+m}$. 
Passing to the quotient by the characteristic 
quasitorus $H$ gives the assertion.

We turn to~(ii).
Let $z \in \rq{X}$ be a point mapping 
to $T_Z \cdot z_\sigma$.
Using once more the specific shape of the defining 
relations $g_i$, we see that if the point 
$z \in \rq{X}$ is singular in $\b{X}$, 
then every point of $\TT^{n+m} \mal z$ is singular in 
$\b{X}$. 
Thus, the assertion follows 
from~\cite[Cor.~3.3.1.12]{ArDeHaLa}.
\end{proof}

\begin{proof}[Proof of Proposition~\ref{prop:cpl1hollow}]
Since $X$ is Gorenstein and log terminal, 
it is canonical.
Let $Z$ be the minimal toric ambient variety 
of $X$.
Recall that $Z$ is the affine toric variety 
defined by the cone $\sigma$ over the columns 
of $P$ and that the toric fixed point $x \in Z$ 
belongs to $X$.
For any point $x' \in X$ different from $x$, 
we infer from Lemma~\ref{lem:bigsingcpl1} 
and~\cite[3.4.4.6]{ArDeHaLa} that, if $x'$ is singular 
in $X$, then it belongs to a curve consisting 
of singular points of $X$.
According to~\cite[Cor.~5.4]{KoMo}, the point $x'$ 
is at most a compund du Val singularity.
Thus, $X$ has at most compound du Val 
singularities if and only if 
every prime divisor $E \subseteq \varphi^{-1}(x)$
has positive discrepancy;
use Condition~\ref{def:cDV}~(iv).
By Proposition~\ref{lemm::discrepancy}, 
the latter holds if and only if there are 
no integral points in 
$\partial A_X^c \cap \sigma^\circ$,
which in turn is the relative interior 
of $\partial A_X^c$.
\end{proof}

\begin{definition}
Let the matrix $P$ be of Type~2 and ordered 
in the sense of Remark~\ref{rem:leadPlat}. 
By the \emph{leading block} of $P$, we mean the 
matrix $[v_{01},\ldots,v_{r1}]$.
\end{definition}

\begin{lemma}
\label{le:lBdata}
Let $X=X(A,P)$ be an affine, Gorenstein, log-terminal
threefold of canonical multiplicity one of Type~2.
\begin{enumerate}
\item
By admissible operations one achieves
that $P$ is ordered in the sense of Remark~\ref{rem:leadPlat},
in the form of Corollary~\ref{cor::Gore}
and the entry $\mathfrak{d}_i$
sitting in column~$v_{i1}$ and row 
number~$r+1$ of~$P$
satisfies $\mathfrak{d}_{i}=0$ whenever $i \ge 3$.
\item
In the situation of~(i), the leading block of 
the matrix $P$ is fully determined by the data
$(l_{01},l_{11}, l_{21}; \mathfrak{d}_0, \mathfrak{d}_1, \mathfrak{d}_2)$.
\end{enumerate}
\end{lemma}

\begin{proof}
The leading block contains the leading platonic triple
$(l_{01},l_{11},l_{21})$. 
All other~$l_{i1}$ must be equal to one.
Due to Corollary~\ref{cor::Gore}, 
the last row of $P$ is determined by these data. 
Subtracting the $\mathfrak{d}_i$-fold of the $i$-th 
from the $r+1$-th row, we obtain 
$\mathfrak{d}_i=0$ for $i \geq 3$. 
Thus apart from $l_{01},l_{11},l_{21}$, the only free 
parameters in the leading block 
are $\mathfrak{d}_0,\mathfrak{d}_1, \mathfrak{d}_2$.
\end{proof}

\begin{definition}
In the situation of Lemma~\ref{le:lBdata}~(i), 
we call 
$(l_{01},l_{11}, l_{21}; \mathfrak{d}_0, \mathfrak{d}_1, \mathfrak{d}_2)$
the \emph{leading block data} of $P$.
\end{definition}

\begin{proposition}
\label{prop:leadingblockdata}
Let $X=X(A,P)$ be an affine Gorenstein log-terminal threefold
of Type 2 and canonical multiplicity one in the form of 
Lemma~\ref{le:lBdata}.
By admissible operations, keeping the form of 
Lemma~\ref{le:lBdata}, we achieve that
the  leading block has one of the following data:
\begin{align*}
(i)
&
\left(5,3,2;0,0,0\right)
&
(ii)
&
\left(4,3,2;0,0,0\right)
&
(iii)
&
\left(4,3,2;1,0,0\right)
\\
(iv)
&
\left(3,3,2;0,0,0\right)
&
(v)
&
\left(3,3,2;1,0,0\right)
&
(vi)
&
\left(l_{01},2,2;0,0,0\right)
\\
(vii)
&
\left(l_{01},2,2;1,0,0\right)
&
(viii)
&
\left(l_{01},2,2;0,1,0\right)
&
(ix)
&
\left(l_{01},l_{11},1;\mathfrak{d}_0,0,0\right)
\end{align*}
\end{proposition}

\begin{proof}
We go through all possible leading platonic triples 
and explicitly list the admissible operations on $P$ 
that produce the desired leading block data.
First, we modify $P$ by subtracting the $i$-th 
row from the last for $i\geq 3$.
Then we have 
$$
\nu_{01} \ = \ 1-l_{01},
\qquad
\nu_{11} \ = \ \nu_{21} \ = \ 1,
\qquad
\nu_{i1} \ = \ \mathfrak{d}_i \ = \ 0,
\ i =3, \ldots, r.
$$
In the sequel, by 
``applying $a = (a_1,a_2,a_3)$'' 
we mean performing the following sequence 
of admissible operations on $P$: 
add the $a_1$-fold of the first, the $a_2$-fold 
of the second and the $a_3$-fold of the last 
to the penultimate row of $P$.

\medskip

\noindent
\emph{Case 1}: 
The leading platonic triple is $(5,3,2)$. 
We arrive at Case~(i) by applying
$$
a 
\ = \ 
\left(
2\mathfrak{d}_0+3\mathfrak{d}_1+5\mathfrak{d}_2
, \
3\mathfrak{d}_0+5\mathfrak{d}_1+7\mathfrak{d}_2
, \
-6\mathfrak{d}_0-10\mathfrak{d}_1-15\mathfrak{d}_2
\right).
$$

\medskip

\noindent
\emph{Case 2}:  
The leading platonic triple is $(4,3,2)$. 
If $\mathfrak{d}_0 \equiv \mathfrak{d}_2 \mod 2$ holds, 
then we arrive at Case~(ii) by applying
$$
a 
\ = \
\left(
\mathfrak{d}_0+\mathfrak{d}_1+2\mathfrak{d}_2
, \
\frac{3}{2}\mathfrak{d}_0+2\mathfrak{d}_1+ \frac{5}{2} \mathfrak{d}_2 
, \
-3\mathfrak{d}_0-4\mathfrak{d}_1-6\mathfrak{d}_2
\right).
$$
If $\mathfrak{d}_0 \equiv \mathfrak{d}_2+1 \mod 2$ holds,
then we arrive at Case~(iii) by applying
$$
a 
\ = \ 
\left(
\mathfrak{d}_0+\mathfrak{d}_1+2\mathfrak{d}_2-1
, \
\frac{3}{2}\mathfrak{d}_0+2\mathfrak{d}_1+ \frac{5}{2} \mathfrak{d}_2 -\frac{3}{2}
, \
-3\mathfrak{d}_0-4\mathfrak{d}_1-6\mathfrak{d}_2+3
\right).
$$

\medskip

\noindent
\emph{Case 3}: 
The leading platonic triple is $(3,3,2)$.
We distinguish the cases 
$\mathfrak{d}_0 \equiv \mathfrak{d}_1 \mod 3$ 
and  
$\mathfrak{d}_0 \equiv \mathfrak{d}_1 +1 \mod 3$ 
(if $\mathfrak{d}_0 \equiv \mathfrak{d}_1 -1 \mod 3$, 
then exchange the data of the blocks $0$ and $1$ 
of $P$). 
We arrive at Cases~(iv) and~(v) by applying respectively
$$
a 
\ = \ 
\left(
\frac{2}{3} \mathfrak{d}_0+\frac{1}{3}\mathfrak{d}_1+\mathfrak{d}_2
, \
\mathfrak{d}_0+\mathfrak{d}_1+\mathfrak{d}_2 
, \
-2\mathfrak{d}_0-2\mathfrak{d}_1-3\mathfrak{d}_2
\right),
$$
$$
a 
\ = \ 
\left(
\frac{2}{3} \mathfrak{d}_0+\frac{1}{3}\mathfrak{d}_1+\mathfrak{d}_2-\frac{2}{3}
, \
\mathfrak{d}_0+\mathfrak{d}_1+\mathfrak{d}_2-1
, \
-2\mathfrak{d}_0-2\mathfrak{d}_1-3\mathfrak{d}_2+2
\right).
$$

\medskip

\noindent
\emph{Case 4:} 
The leading platonic triple is $(l_{01},2,2)$. 
We distinguish several subcases and will work with 
$$
a
\ = \ 
\left(
\frac{1}{2}\mathfrak{d}_0+
\frac{l_{01}-2}{4}
\mathfrak{d}_1+\frac{l_{0j_0}}{4}\mathfrak{d}_2
, \, 
\frac{1}{2}\mathfrak{d}_0+\frac{l_{01}}{4}\mathfrak{d}_1+
\frac{l_{01}-2}{4}\mathfrak{d}_2
, \,
-\mathfrak{d}_0-\frac{l_{01}}{2}(\mathfrak{d}_1+\mathfrak{d}_2)
\right).
$$

\medskip

\noindent
\emph{4.1:} 
We have $l_{01} \equiv 1 \mod 4 $. 

\smallskip

\noindent
\emph{4.1.1:} 
$\mathfrak{d}_1 \equiv \mathfrak{d}_2 \mod 4$.
If $\mathfrak{d}_0$ is even, then applying $a$, we arrive at Case (vi).  
If $\mathfrak{d}_0$ is odd, then applying 
$a  + (-1/2, \ -1/2, \ 1 )$,
we arrive at Case (vii).

\smallskip

\noindent
\emph{4.1.2:} 
$\mathfrak{d}_1 \equiv \mathfrak{d}_2 + 1 \mod 4 $.
If $\mathfrak{d}_0$ is even, then applying 
$a  + (1/4 , \ -1/4 , \ 1/2)$ leads to Case~(viii). 
If $\mathfrak{d}_0$ is odd, then applying
$a + (-1/4 , \ 1/4 , \ 1/2)$ and exchanging 
the data of column blocks 1 and 2 leads to
Case~(viii).

\smallskip

\noindent
\emph{4.1.3:}
$\mathfrak{d}_1 \equiv \mathfrak{d}_2 - 1 \mod 4$.
Exchanging the data of column blocks 1 and 2, we are in 4.1.2 and thus
arrive at Case~(viii).

\smallskip

\noindent
\emph{4.1.4:} 
$\mathfrak{d}_1 \equiv \mathfrak{d}_2 + 2 \mod 4 $. 
If $\mathfrak{d}_0$ is odd, then applying $a$, we arrive at Case~(vi). 
If $\mathfrak{d}_0$ is even, then applying $a  + (-1/2, \, -1/2, \, 1 )$
leads to Case~(vii).

\medskip

\noindent
\emph{4.2:} 
We have $l_{01} \equiv 2 \mod 4 $. 

\smallskip

\noindent
\emph{4.2.1:} 
$\mathfrak{d}_0 \equiv \mathfrak{d}_1 \equiv \mathfrak{d}_2 \mod 2 $. 
Applying $a$, we arrive at Case~(vi).

\smallskip

\noindent
\emph{4.2.2:} 
$\mathfrak{d}_0 \equiv \mathfrak{d}_1 \not\equiv \mathfrak{d}_2 \mod 2 $. 
Applying $a +  (0, \,  -1/2, \,  1)$, we arrive at Case (viii).

\smallskip

\noindent
\emph{4.2.3:} 
$\mathfrak{d}_0 \equiv \mathfrak{d}_2 \not\equiv \mathfrak{d}_1 \mod 2 $.
Exchanging the data of column blocks 1 and 2, we are in 4.2.2
and thus arrive at Case (viii).

\smallskip

\noindent
\emph{4.2.4:}
$\mathfrak{d}_0 \not\equiv \mathfrak{d}_1 \equiv \mathfrak{d}_2 \mod 2 $. 
Applying $a  + (-1/2, \, -1/2, \, 1 )$, we arrive at Case~(vii).

\medskip

\noindent
\emph{4.3 and 4.4:} 
$l_{01} \equiv 3 \mod 4 $ or $l_{01} \equiv 3 \mod 4 $,
respectively. 
These cases are settled by similar arguments as~4.1 
and~4.2. 
That means that the same admissible operations are 
applied after, if necessary exchanging the data of 
column blocks~1 and~2.

\medskip

\noindent
\emph{Case 5:} 
The leading platonic triple is $(l_{01},l_{11},1)$. 
Applying 
$(0 , \, \mathfrak{d}_1-\mathfrak{d}_2 , \, -\mathfrak{d}_1)$, 
we arrive at Case (ix).

\medskip

Finally, in each of the cases~(i) to~(ix), 
we modify the matrix $P$ obtained so far by adding 
the $i$-th row to the last one for $i = 3, \ldots, r$.
This brings ~$P$ again into the form of 
Lemma~\ref{le:lBdata}~(i).
\end{proof}

\section{Proof of Theorems~\ref{thm:cdv-class-intro} 
and~\ref{thm:cdv-graph-intro}}

In Propositions~\ref{prop:Qfac}, 
\ref{prop:nonQfac} and~\ref{prop:zeta>1}, 
we classify the compound du Val singularities 
admitting a torus action of complexity one
and list their defining matrices~$P$,
numerated according to their appearance in 
Theorem~\ref{thm:cdv-class-intro}. 
We begin with the case of 
$\QQ$-factorial non-toric 
threefolds of canonical multiplicity one.

\begin{proposition}
\label{prop:Qfac}
Let $X$ be a non-toric affine threefold of Type~2.
Assume that~$X$ is $\QQ$-factorial, of canonical multiplicity 
one and has at most compound du Val singularities. 
Then $X$, for suitable $A$, is isomorphic to $X(A,P)$,
where $P$ is one of the following matrices:

\setlength{\arraycolsep}{.1cm}
$$
{
\begin{array}{rlrlrl}
\text{(8)}
&
{\tiny
\begin{bmatrix}
  -5 & 3 & 0 & 0 \\
	-5 & 0 & 2 & 0 \\
	0 & 0 & 0 & 1 \\
	-4 & 1 & 1 & 1
 \end{bmatrix}
}
&
\text{(7) }
&
{\tiny
 \begin{bmatrix}
  -4 & 3 & 0 & 0 \\
	-4 & 0 & 2 & 0 \\
	0 & 0 & 0 & 1 \\
	-3 & 1 & 1 & 1
 \end{bmatrix}
}
&
\text{(18)}
&
 {\tiny
\begin{bmatrix}
  -4 & -1 & 3 & 0  \\
	-4 & -1 & 0 & 2  \\
	1 & 3 & 0 & 0  \\
	-3 & 0 & 1 & 1 
 \end{bmatrix}
}
 \\[.5cm]
\text{(6)}
 &
{\tiny
 \begin{bmatrix}
  -3 & 3 & 0 & 0 \\
	-3 & 0 & 2 & 0 \\
	0 & 0 & 0 & 1 \\
	-2 & 1 & 1 & 1
 \end{bmatrix}
}
 &
\text{(15)}
 & 
 {\tiny
\begin{bmatrix}
  -3 & -1 & 3 & 0  \\
	-3 & -1 & 0 & 2  \\
	1 & 2 & 0 & 0  \\
	-2 & 0 & 1 & 1 
 \end{bmatrix}
}
&
\text{(17)}
 &
 {\tiny
\begin{bmatrix}
  -3 & -2 & 3 & 0  \\
	-3 & -2 & 0 & 2  \\
	1 & 1 & 0 & 0  \\
	-2 & -1 & 1 & 1 
 \end{bmatrix}
}
 \\[.5cm]
\text{(4)}
 &
{\tiny
 \begin{bmatrix}
  -k & 2 & 0 & 0 \\
	-k & 0 & 2 & 0 \\
	0 & 0 & 0 & 1 \\
	1-k & 1 & 1 & 1
 \end{bmatrix}
}
&
\text{(12-e-e)}
&
{\tiny
\begin{bmatrix}
  -k_1 & -k_2 & 2 & 0  \\
	-k_1 & -k_2 & 0 & 2  \\
	0 & 1 &  0 & 0  \\
	1-k_1 & 1-k_2 & 1 & 1 
 \end{bmatrix}
}
 &
\text{(5-o)}
 &
  {\tiny
\begin{bmatrix}
  -k & 2 & 0 & 0  \\
	-k & 0 & 2 & 0  \\
	1 & 0 & 0 & 0  \\
	1-k & 1 & 1 & 1 
 \end{bmatrix}
}
\\[.5cm]
\text{(11)}
&
 {\tiny
\begin{bmatrix}
  -k & 2 & 1 & 0  \\
	-k & 0 & 0 & 2  \\
	1 & 0 & 0 & 0  \\
	1-k & 1 & 1 & 1 
 \end{bmatrix}
}
&
\text{(12-o-e/o)}
&
{\tiny
\begin{bmatrix}
  -2k_1 & - 2k_2-1 & 2  & 0  \\
	-2k_1 & - 2k_2-1  & 0  & 2  \\
	0 &  k_1-k_2  & 1  & 0  \\
	1-2k_1 & -2k_2 & 1 & 1 
 \end{bmatrix}
}
&
\text{(16)}
&
{\tiny
 \begin{bmatrix}
  -4 & 2 & 1 & 0  \\
	-4 & 0 & 0 & 2  \\
	0 & 1 & 2 & 0  \\
	-3 & 1 & 1 & 1 
 \end{bmatrix}
}
\\[.5cm]
\text{(5-e)}
&
{\tiny
 \begin{bmatrix}
  -2k-1 & 2  & 0 &0 \\
	-2k-1 & 0  & 2 & 0  \\
	0 & 1 & 0 & k+1 \\
	-2k & 1  & 1 & 1
 \end{bmatrix}
}
\quad
&
\text{(10-o)}
&
{\tiny
\begin{bmatrix}
  -2k-1   & 2  & 1 & 0  \\
	-2k-1 & 0  & 0 & 2   \\
	0     & 1  & \left\lceil \frac{2k+1}{4} \right\rceil & 0  \\
	-2k   & 1  & 1 & 1
 \end{bmatrix}
}
\qquad
&
\end{array}
}
$$
where the parameters $k,k_1,k_2$ are positive integers
and in (4), (5-o) and (11), we have $k\geq2$.
Moreover, (12-e-e) indicates that the two exponents 
in the defining equation of 
Theorem~\ref{thm:cdv-class-intro}~(12)
are even, in (5-o) the exponent is odd etc..
\end{proposition}

\begin{proof}
We may assume that $P$ is irredundant and in 
the form of Proposition~\ref{prop:leadingblockdata}.
As~$X$ is $\QQ$-factorial, the matrix $P$ has 
precisely $r+2$ columns, i.e., is a square matrix, 
see Corollary~\ref{affQfact}. 
Since we assume~$P$ to be irredundant and $l_{ij} =1$ holds
for $i \ge 3$, we must have $n_i \ge 2$ for $i \ge 3$. 
This forces $r \le 3$. 
The strategy is now to compute suitable parts of
$\partial A_X^c$ explicitly
according to Proposition~\ref{prop:acancompstruct}
and to use the fact that they don't contain interior 
lattice points, as guaranteed by 
Proposition~\ref{prop:cpl1hollow}.

Consider the case $r = 3$. 
Here, we have $n_0=n_1=n_2=1$ and $n_3=2$. 
Moreover, $(l_{01},l_{11},l_{21})$ is a 
platonic triple with $l_{21} > 1$ 
and $l_3=(1,1)$ holds.
The column apart from the leading block of $P$ 
is $v_{32} = (0,0,1,t,0)$, where we may 
assume that~$t$ is a positive integer.
The vertices of $\partial A_X^c(\lambda)$ 
thus are 
$$ 
\left(0,0,0,\frac{\alpha}{\beta},1\right),
\qquad\qquad
\left(0,0,0,\frac{\alpha + tl_{01}l_{11}l_{21}}{\beta},1\right),
$$
where
$$
\alpha 
:= 
\mathfrak{d}_0l_{11}l_{21}
+\mathfrak{d}_1l_{01}l_{21}
+\mathfrak{d}_2l_{01}l_{11},
\quad
\beta := l_{11}l_{21}+l_{01}l_{21}+l_{01}l_{11}-l_{01}l_{11}l_{21}
$$
Since $l_{01},l_{11},l_{21}$ all differ from one, 
$tl_{01}l_{11}l_{21}/\beta \geq 2$ holds and  
thus $\partial A_X^c(\lambda)$ contains an integral 
point in its relative interior. 
Consequently $r=3$ is impossible.

We are left with the case $r=2$. 
Here, $P$ is a $4 \times 4$ matrix,
the leading block columns are 
$v_{01},v_{11},v_{21}$ and the 
column $v$ of $P$ apart from these three
is one of 
$$
v_{02}=(-k, -k,t,1-k),
\quad
v_{12}=(k,0,0,t,1),
\quad
v_{22}=(0,k,t,1),
\quad
v_1=(0,0,t,1).
$$
We now go through the list of all possible leading block data
provided by Proposition~\ref{prop:leadingblockdata}.
We will often compute the line segment 
$\partial A_X^c(\lambda) \subseteq \QQ^4$
from Proposition~\ref{prop:roofacan} explicitly.
According to Proposition~\ref{prop:acancompstruct},
the $P$-elementary cone spanned by the columns of the 
leading block produces the first vertex $w_1$ 
of $\partial A_X^c(\lambda)$
and the second vertex $w_2$ either arises from a
(unique) second $P$-elementary cone or one has 
$w_2 = v = v_1$.

Let $P$ have the leading block data $(5,3,2;0,0,0)$. 
Then the first vertex of $\partial A_X^c(\lambda)$
is $w_1 = (0,0,0,1)$.
Consider the case that the additional column 
$v$ lies in the 
relative interior $\lambda_0^{\circ} \subseteq \lambda_0$.
Then $v = v_{02}=(-k,-k,t,1-k)$ 
with $1 \le k \le 5$, where we may assume $t>0$. 
We compute $w_2 = (0,0,6t/(6-k),1)$.
The following figures show  
$\partial A_X^c(\lambda_0) \subseteq \mathcal{H}_0$ 
with the lower edge being
$\partial A_X^c(\lambda)$,
where the plane $\mathcal{H}_0$ is defined 
as in Proposition~\ref{prop:roofacan}:

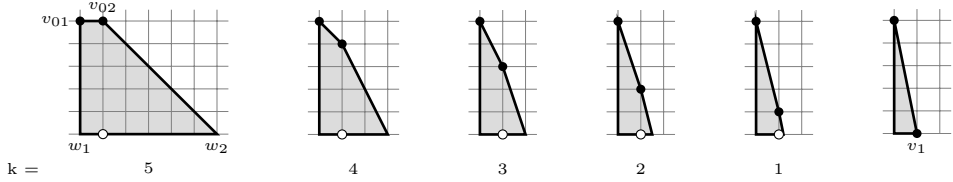
\begin{center}
\begin{tikzpicture}[scale=0.3]
\fill[fill=gray!30!, fill opacity=0.90] (0,0)--(0,5)--(1,5)--(6,0)--cycle;
\draw[gray, very thin] (-0.5,0) grid (6.5,5.5);
\draw[line width=1pt] (0,0)--(0,5)--(1,5)--(6,0)--cycle;
\coordinate[label=left:\tiny $v_{01}$] (A) at (0,5);
\coordinate[label=above:\tiny $v_{02}$] (B) at (1,5);
\coordinate[label=below:\tiny $w_1$] (C) at (0,0);
\coordinate[label=below:\tiny $w_2$] (D) at (6,0);
\fill (0,5) circle (0.2);
\fill (1,5) circle (0.2);
\fill[fill=white] (1,0) circle (0.2);
\draw (1,0) circle (0.2);
\node at (-2.5,-1.5) {\tiny k =};
\node at (3,-1.5) {\tiny 5};
\end{tikzpicture}
\hfill
\begin{tikzpicture}[scale=0.3]
\fill[fill=gray!30!, fill opacity=0.90] (0,0)--(0,5)--(1,4)--(3,0)--cycle;
\draw[gray, very thin] (-0.5,0) grid (3.5,5.5);
\draw[line width=1pt] (0,0)--(0,5)--(1,4)--(3,0)--cycle;
\fill (0,5) circle (0.2);
\fill (1,4) circle (0.2);
\fill[fill=white] (1,0) circle (0.2);
\draw (1,0) circle (0.2);
\node at (1.5,-1.5) {\tiny 4};
\end{tikzpicture}
\hfill 
\begin{tikzpicture}[scale=0.3]
\fill[fill=gray!30!, fill opacity=0.90](0,0)--(0,5)--(1,3)--(2,0)--cycle;
\draw[gray, very thin] (-0.5,0) grid (2.5,5.5);
\draw[line width=1pt] (0,0)--(0,5)--(1,3)--(2,0)--cycle;
\fill (0,5) circle (0.2);
\fill (1,3) circle (0.2);
\fill[fill=white] (1,0) circle (0.2);
\draw (1,0) circle (0.2);
\node at (1,-1.5) {\tiny 3};
\end{tikzpicture}
\hfill 
\begin{tikzpicture}[scale=0.3]
\fill[fill=gray!30!, fill opacity=0.90](0,0)--(0,5)--(1,2)--(1.5,0)--cycle;
\draw[gray, very thin] (-0.5,0) grid (2.5,5.5);
\draw[line width=1pt] (0,0)--(0,5)--(1,2)--(1.5,0)--cycle;
\fill (0,5) circle (0.2);
\fill (1,2) circle (0.2);
\fill[fill=white] (1,0) circle (0.2);
\draw (1,0) circle (0.2);
\node at (1,-1.5) {\tiny 2};
\end{tikzpicture}
\hfill 
\begin{tikzpicture}[scale=0.3]
\fill[fill=gray!30!, fill opacity=0.90](0,0)--(0,5)--(1,1)--(1.2,0)--cycle;
\draw[gray, very thin] (-0.5,0) grid (2.5,5.5);
\draw[line width=1pt] (0,0)--(0,5)--(1,1)--(1.2,0)--cycle;
\fill (0,5) circle (0.2);
\fill (1,1) circle (0.2);
\fill[fill=white] (1,0) circle (0.2);
\draw (1,0) circle (0.2);
\node at (1,-1.5) {\tiny 1};
\end{tikzpicture}
\hfill 
\begin{tikzpicture}[scale=0.3]
\node at (1,-1.4) {${\color{white}{0}}$};
\fill[fill=gray!30!, fill opacity=0.90] (0,0)--(0,5)--(1,0)--cycle;
\draw[gray, very thin] (-0.5,0) grid (2.5,5.5);
\draw[line width=1pt] (0,0)--(0,5)--(1,0)--cycle;
\coordinate[label=below:\tiny $v_{1}$] (B) at (1,0);
\fill (0,5) circle (0.2);
\fill (1,0) circle (0.2);
\end{tikzpicture}
\hfill \
\end{center}
\noindent
where the last figure indicates the case of the
additional column lying in $\lambda$, treated below.
Now, because of $6t/(6-k)  \geq 6/5$, we find the 
point $(0,0,1,1)$ in the relative interior
of $\partial A_X^c(\lambda)$ and hence
in the relative interior of $\partial A_X^c$.
According to Proposition~\ref{prop:cpl1hollow},
we leave the compound du Val case here.

We proceed in a more condensed way.
Assume $v \in \lambda_1^{\circ}$.
Then $v = v_{12}=(k,0,t,1)$ with 
$1 \le k \le 3$ and we can assume $t>1$.
We obtain $w_2 = (0,0,10t/(10-3k),1)$.
We find again $(0,0,1,1)$ in 
$\partial A_X^c(\lambda)^{\circ}$
and thus leave the compound du Val case. 
Assume $v \in \lambda_2^{\circ}$.
Then $v = v_{22}=(0,k,t,1)$  with 
$1 \le k \le 2$ and we can assume $t>1$.
We obtain $w_2  = (0,0,15t/(15-7k,1)$.
Once more, $(0,0,1,1)$ shows up in 
$\partial A_X^c(\lambda)^{\circ}$ and we leave 
the compound du Val case. 
Finally, assume $v = v_1=(0,0,t,1)$.
We may assume $t > 0$.
Only for $t=1$ there are no lattice points
in $\partial A_X^c(\lambda)^{\circ}$. 
Moreover, if $t=1$, then all $\partial A_X^c(\lambda_i)$ 
are hollow polytopes of the first type 
of Proposition~\ref{prop:toriccDV}
and we arrive at matrix~$(8)$ from the assertion
defining the compund du Val singularity
$E_8 \times \CC$. 

Let $P$ have the leading block data $(4,3,2;0,0,0)$. 
Also here, the first vertex of $\partial A_X^c(\lambda)$ 
is $w_1 = (0,0,0,1)$.
Assume $v \in \lambda_0^{\circ}$.
Then $v = v_{02}=(-k,-k,t,1-k)$ 
with $1 \le k \le 4$, where we may assume $t>0$.
We obtain $w_2 = (0,0,6t/(6-k),1)$.
Thus, $(0,0,1,1)$ lies in  
$\partial A_X^c(\lambda)^{\circ}$ and we leave 
the compound du Val case. 
Assume $v \in \lambda_1^{\circ}$. 
Then $v = v_{12}=(k,0,t,1)$, where 
$1 \le k \le 3$ and we can assume $t > 0$. 
We obtain $w_2 = (0,0,4t/(4-k),1)$
and find $(0,0,1,1)$ in the relative interior 
of $\partial A_X^c(\lambda)$ and thus 
leave the compound du Val case.
Assume $v \in \lambda_2^{\circ}$.
Then $v = v_{22}=(0,k,t,1)$  with 
$k = 1, 2$ and we can assume $t > 0$.
We obtain $w_2 = (0,0,12t/(12-5k),1)$
and see that $(0,0,1,1)$ lies in 
$\partial A_X^c(\lambda)^{\circ}$.
Thus, we leave the compound du Val case.
Finally, assume $v = v_1=(0,0,t,1)$. 
For $t > 1$, we find $(0,0,1,1)$
in $\partial A_X^c(\lambda)^{\circ}$. 
The case $t=1$ gives matrix~$(7)$, 
defining the compound du Val singularity
$E_7 \times \CC$.%

Let $P$ have the leading block data
$\left(4,3,2;1,0,0\right)$. 
Here, the first vertex 
of $\partial A_X^c(\lambda)$ is 
$w_1 = (0,0,3,1)$.
To visualize the setting,
consider the $P$-elementary cone 
$\tau \subseteq \QQ^4$ generated by the 
columns $v_{01}, v_{11}, v_{21}$ of 
the leading block and the line segments
$\partial A_X^c(\lambda_i, \tau) \subseteq \mathcal{H}_i$,
where $i = 0,1,2$, 
from Proposition~\ref{prop:roofacan}:

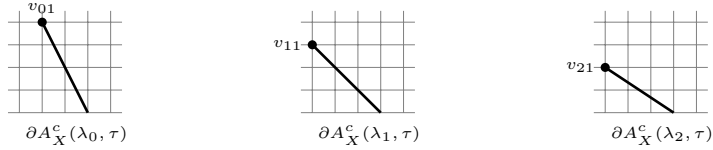
\begin{center}

\hfill
\begin{tikzpicture}[scale=0.3]
\draw[gray, very thin] (-0.5,0) grid (4.5,4.5);
\draw[line width=1pt] (3,0)--(1,4);
\coordinate[label=above:\tiny $v_{01}$] (A) at (1,4);
\fill (1,4) circle (0.2);
\node at (2.5,-1) {\tiny $\partial A_X^c(\lambda_0, \tau)$};
\end{tikzpicture}
\hfill
\begin{tikzpicture}[scale=0.3]
\draw[gray, very thin] (-0.5,0) grid (4.5,4.5);
\draw[line width=1pt] (3,0)--(0,3);
\coordinate[label=left:\tiny $v_{11}$] (A) at (0,3);
\fill (0,3) circle (0.2);
\node at (2.5,-1) {\tiny $\partial A_X^c(\lambda_1, \tau)$};
\end{tikzpicture}
\hfill 
\begin{tikzpicture}[scale=0.3]
\draw[gray, very thin] (-0.5,0) grid (4.5,4.5);
\draw[line width=1pt] (3,0)--(0,2);
\coordinate[label=left:\tiny $v_{21}$] (A) at (0,2);
\fill (0,2) circle (0.2);
\node at (2.5,-1) {\tiny $\partial A_X^c(\lambda_2, \tau)$};
\end{tikzpicture}
\hfill \

\end{center}
Note that the additional column $v$
is represented in the above figures by a lattice point not 
contained in $\partial A_X^c(\lambda_i, \tau)$, indicated by 
the black line.
Going through the cases, we will also have to
look at the polytopes 
$\partial A_X^c(\lambda_i)$
and will encounter the following situations:%

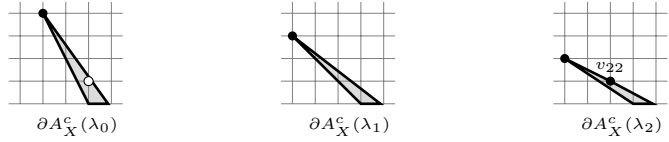
\begin{center}

\hfill
\begin{tikzpicture}[scale=0.3]
\fill[fill=gray!30!, fill opacity=0.90](3,0)--(1,4)--(27/7,0);
\draw[gray, very thin] (-0.5,0) grid (4.5,4.5);
\draw[line width=1pt] (3,0)--(1,4)--(27/7,0)--cycle;
\fill (1,4) circle (0.2);
\fill[fill=white] (3,1) circle (0.2);
\draw (3,1) circle (0.2);
\node at (2.5,-1) {\tiny $\partial A_X^c(\lambda_0)$};
\end{tikzpicture}
\hfill
\begin{tikzpicture}[scale=0.3]
\fill[fill=gray!30!, fill opacity=0.90](3,0)--(0,3)--(27/7,0);
\draw[gray, very thin] (-0.5,0) grid (4.5,4.5);
\draw[line width=1pt] (3,0)--(0,3)--(27/7,0)--cycle;
\fill (0,3) circle (0.2);
\node at (2.5,-1) {\tiny $\partial A_X^c(\lambda_1)$};
\end{tikzpicture}
\hfill 
\begin{tikzpicture}[scale=0.3]
\fill[fill=gray!30!, fill opacity=0.90](3,0)--(0,2)--(2,1)--(27/7,0);
\draw[gray, very thin] (-0.5,0) grid (4.5,4.5);
\draw[line width=1pt] (3,0)--(0,2)--(2,1)--(27/7,0)--cycle;
\coordinate[label=above:\tiny $v_{22}$] (B) at (2,1);
\fill (0,2) circle (0.2);
\fill (2,1) circle (0.2);
\node at (2.5,-1) {\tiny $\partial A_X^c(\lambda_2)$};
\end{tikzpicture}
\hfill \

\end{center}

Assume $v \in \lambda_0^{\circ}$.
Then $v = v_{02} = (-k,-k,t,1-k)$ with 
$1 \le k \le 4$.  The second vertex 
of $\partial A_X^c(\lambda)$ is $w_2 = (0,0,6t/(6-k),1)$.
We find one of the points $(0,0,4,1)$ or $(0,0,2,1)$ in 
$\partial A_X^c(\lambda)^\circ$
for $k=2,4$.
Moreover, for $k=3$, we find $(-1,-1,3,0)$
in $\partial A_X^c(\lambda_0)^\circ$.
Thus, we end up with non compound 
du Val singularities for $k=2,3,4$.
In the case $k=1$, we may assume $t>2$.
Only for $t=3$,
no lattice points are inside $\partial {A_X^c}^\circ$.
For $t>3$, the point $(-1,-1,3,0)$ 
lies in  $\partial A_X^c(\lambda_0)^\circ$.
So with $t=3$, we obtain matrix~$(18)$, defining 
a compound du Val singularity.

We show that the remaining possible locations of 
$v$ all lead to non compound du Val singularities.
Assume $v \in \lambda_1^{\circ}$.
Then $v = v_{12} = (k,0,t,1) \in \lambda_1^\circ$ 
with $1 \le k \le 3$. 
The second vertex of $\partial A_X^c(\lambda)$ 
is $w_2 = (0,0,(k+4t)/(4-k),1)$.
Thus, either $(0,0,2,1)$ or $(0,0,4,1)$ lies 
in $\partial A_X^c(\lambda)^\circ$.
Assume $v \in \lambda_2^{\circ}$. 
Then $v = v_{22} = (0,k,t,1)$ 
with $k = 1,2$, where we can assume 
$t > 1$ or $t>0$ accordingly.
The second vertex of 
$\partial A_X^c(\lambda)$ 
is $w_2 = (0,0,(3k+12t)/(12-5k),1)$. 
For $k=2$, we find $(0,0,4,1)$
in $\partial A_X^c(\lambda)^\circ$.
For $k=1$, 
the segment $\partial A_X^c(\lambda)$
is of length $(12t-18)/7$. 
Thus, for $t \ge 3$, we find a lattice
point in $\partial A_X^c(\lambda)^\circ$.
For $t= 2$, we look at $\partial A_X^c(\lambda_0)^\circ$ 
and see that it contains $(-1,-1,3,0)$; 
see the figure above. 
Finally, if $v = v_1 \in \lambda$,
one finds $(-1,-1,3,0)$
in $\partial A_X^c(\lambda_0)^\circ$.

Let $P$ have the leading block data $(3,3,2;0,0,0)$. 
Then the first vertex of $\partial A_X^c(\lambda)$ is 
$w_1 = (0,0,0,1)$. 
Assume $v=(-k,-k,t,1-k) \in \lambda_0^\circ$ 
or $v=(k,0,t,1)  \in \lambda_1^\circ$
with $k=1,2,3$.
Then we can assume $t>0$.
We obtain  $w_2 = (0,0,6t/(6-k),1)$,
find $(0,0,1,1)$ in
$\partial A_X^c(\lambda)^\circ$
and thus leave the compound du Val case. 
If $v=(0,k,t,1)  \in  \lambda_2^\circ$,
with
$k=1,2$, we can assume $t > 0$.
We obtain $w_2 = (0,0,3t/(3-k),1)$
and find
$(0,0,1,1)$ in
$\partial A_X^c(\lambda)^\circ$.
Thus also here, we leave the compound du Val 
case.
Finally, if
$v=(0,0,t,1) \in \lambda$,
then we end up with $t=1$ and the 
matrix~$(6)$, defining the compound 
du Val singularity $E_6 \times \CC$. 

Let $P$ have the leading block  data $(3,3,2;1,0,0)$. 
Then the first vertex of $\partial A_X^c(\lambda)$ is 
$w_1 = (0,0,2,1)$. 
We will take a look at the leaves:

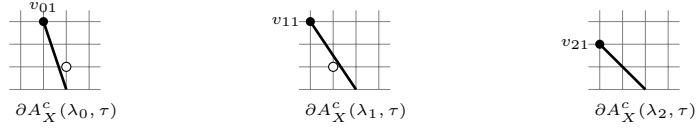
\begin{center}

\hfill
\begin{tikzpicture}[scale=0.3]
\draw[gray, very thin] (-0.5,0) grid (3.5,3.5);
\draw[line width=1pt] (2,0)--(1,3);
\coordinate[label=above:\tiny $v_{01}$] (A) at (1,3);
\fill (1,3) circle (0.2);
\fill[fill=white] (2,1) circle (0.2);
\draw (2,1) circle (0.2);
\node at (2,-1) {\tiny $\partial A_X^c(\lambda_0, \tau)$};
\end{tikzpicture}
\hfill
\begin{tikzpicture}[scale=0.3]
\draw[gray, very thin] (-0.5,0) grid (3.5,3.5);
\draw[line width=1pt] (2,0)--(0,3);
\coordinate[label=left:\tiny $v_{11}$] (A) at (0,3);
\fill (0,3) circle (0.2);
\fill[fill=white] (1,1) circle (0.2);
\draw (1,1) circle (0.2);
\node at (2,-1) {\tiny $\partial A_X^c(\lambda_1, \tau)$};
\end{tikzpicture}
\hfill 
\begin{tikzpicture}[scale=0.3]
\draw[gray, very thin] (-0.5,0) grid (3.5,3.5);
\draw[line width=1pt] (2,0)--(0,2);
\coordinate[label=left:\tiny $v_{21}$] (A) at (0,2);
\fill (0,2) circle (0.2);
\node at (2,-1) {\tiny $\partial A_X^c(\lambda_2, \tau)$};
\end{tikzpicture}
\hfill \

\end{center}

Assume $v \in \lambda_0^\circ$.
Then $v = v_{01} = (-k,-k,t,1-k) $ with $k = 1,2,3$.
We obtain $w_2 = (0,0,6t/(6-k),1)$.
In the case $k=3$ as well as in the case
$k=2$ with $t \neq 1$,
we find one of $(0,0,1,1)$ and $(0,0,3,1)$ 
in $\partial A_X^c (\lambda)^\circ$
and leave the compound du Val case.
For $k=2$ and $t=1$, there are no lattice 
points in $\partial A_X^c$ and the resulting 
matrix is~(15), defining a compound du Val 
singularity. 
If $k=1$ and $t  \ne 2$, 
we  find $(0,0,1,1)$ or $(0,0,3,1)$ in
$\partial A_X^c (\lambda)^\circ$.
The case $t=2$
leads to the matrix~(7), defining 
a compound du Val singularity.
The case of $v \in \lambda_1^\circ$ 
can be reduced by means of admissible operations 
to the previous case.
We show that for the remaining possible
locations of $v$, we leave the compound 
du Val case.
If $v=(0,k,t,1) \in \lambda_2^\circ$, 
then $w_2 = (0,0,(3t+k)/(3-k),1)$
and we find $(0,0,1,1)$ or $(0,0,3,1)$ in
$\partial A_X^c (\lambda)^\circ$.
If $v=(0,0,t,1) \in \lambda$, 
then $(-1,-1,2,0)$ or $(1,0,1,1)$
lies in $\partial A_X^c (\lambda)^\circ$.

Let $P$ have the leading block data $(l_{01},2,2;0,0,0)$. 
Then the first vertex of $\partial A_X^c(\lambda)$ is 
$(0,0,0,1)$. 
Assume $v \in \lambda_0^\circ$.
Then $v = v_{02} = (-k,-k,t,1-k)$ 
with $1 \le k \le l_{01}$,
where we can assume $t > 0$.
We have $w_2=(0 ,0 ,t ,1)$. 
For $t>1$, we obtain  $(0,0,1,1) \in \partial A_X^c(\lambda)^\circ$
and thus leave the compound du Val case.
For $t=1$, the resulting singularity is compound du Val
for every $k$ and has defining matrix~(12-e-e) with $k_1 \ge k_2$.
Assume $v \in \lambda_1^\circ$.
Then $v=v_{12}=(k,0,t,1)$ with $k=1,2$. 
We can assume $l_{01}>2$ and $t > 0$.
For $k=1$ we have $w_2=(0,0,2tl_{01}/(2+l_{01}),1)$
and for $k=2$, we have $w_2=(0,0, tl_{01}/2, 1)$. 
In both cases, $\partial A_X^c(\lambda)^\circ$ contains 
$(0,0,1,1)$ and we obtain a 
non compound du Val singularity.
The case of $v \in \lambda_2^\circ$ can be transformed 
via exchanging the data of 
blocks~1 and~2 into the previous one.
Finally, if $v=(0,0,t,1) \in \lambda$,
then we must have $t=1$ and this 
gives the compound du Val singularity $D_{l_{01}+2} \times \CC$,
defined by the matrix~(4).

Let $P$ have the leading block data $(l_{01},2,2;1,0,0)$. 
Then the first vertex of $\partial A_X^c(\lambda)$ is 
$(0,0,1,1)$. 
Assume $v  \in \lambda_0^\circ$. 
Then $v = v_{02} = (-k,-k,t,1-k)$ with $1 \le k \le l_{01}$.
We can assume $t<1$.
For $t<0$, we have $(0,0,0,1) \in \partial A_X^c(\lambda)^\circ$.
For $t=0$, we obtain a matrix (12-e-e) as in the case of leading 
block data $(l_{01},2,2;0,0,0)$, now with $k_1 \leq k_2$.
Assume $v \in \lambda_1^\circ$. 
The case $l_{01}=2$ can be transformed
via admissible operations 
into the case of leading block data
$(l_{01},2,2;0,1,0)$ and an additional column 
in  $\lambda_0^\circ$, 
which is discussed below.
So, let $l_{01}>2$. 
Then $v = (k,0,t,1)$, where $k = 1,2$.
For $k=2$, we can assume $t>0$.
We obtain $w_2=(0,0,1+ tl_{01}/2, 1)$
and $(0,0,2,1) \in \partial A_X^c (\lambda)^\circ$
and thus leave the compound du Val case.
Now let $k=1$. 
Here, $t$ may be any integer and we obtain 
$w_2 = (0,0,2(1+l_{01}t)/(2+l_{01}),1)$. 
Only for $t=0,1$ there are no lattice points 
in $\partial A_X^c (\lambda)^\circ$.
Both cases lead by admissible operations 
to the compound du Val singularity 
with defining matrix~(11).
The case of $v \in \lambda_2^\circ$ 
can be transformed to the previous one by exchanging the 
data of column blocks~1 and~2.
Finally, if $v \in \lambda$, then it equals either
$(0,0,0,1)$ or  $(0,0,2,1)$.
Both cases lead to the compound du Val singularity with 
defining matrix~(5o).

Let $P$ have leading block data 
$(l_{01},2,2;0,1,0)$. 
Then the first vertex of $\partial A_X^c (\lambda)$
is $w_1 = (0,0, l_{01}/2 ,1)$. 

\medskip

\noindent
\emph{Case 1:} The exponent $l_{01}$ is even. 
Assume $v \in \lambda$.
Then $v = v_1 = w_2 = (0,0,t,1)$.
Exchanging the data of blocks~0 and~1
transforms the case $l_{01}=2$ into the 
corresponding case with leading block data 
$(l_{01},2,2;1,0,0)$ treated before.
So, let $l_{01}>2$.
Having no lattice points in
$\partial {A}_{X}^c (\lambda)^\circ$
implies $t = l_{01}/2 \pm 1$.
But then, there are integer points
in $\partial {A}_{X}^c (\lambda_0)^\circ$:
for $t = l_{01}/2 + 1$ we find
$$ 
\left(-1,-1,\frac{l_{01}}{2},0\right)
\ = \ 
\frac{1}{l_{01}} v_{01} + \frac{1}{2} w_1 
+ \left(\frac{1}{2}-\frac{1}{l_{01}}\right) w_2
$$
and for $t = l_{01}/2 - 1$ we find
$$ 
\left(-1,-1,\frac{l_{01}}{2}-1,0\right)
\ = \ 
\frac{1}{l_{01}} v_{01} + \left(\frac{1}{2}-\frac{2}{l_{01}}\right) w_1 
+ \left( \frac{1}{l_{01}}+\frac{1}{2}\right) w_2.
$$

Assume $v \in \lambda_0^\circ$.
Then $v = v_{02} = (-k,-k,t,1-k)$ with 
$1 \le k \le l_{01}$.
The second vertex of $\partial A_X^c (\lambda)$
is $w_2 = (0,0,t+ k/2,1)$.
If $k$ is even, then having no lattice points in
$\partial {A}_{X}^c (\lambda)^\circ$
implies $t = (l_{01}-k)/2 \pm 1$.
Again there are integer points
in $\partial {A}_{X}^c (\lambda_0)^\circ$:
for $t = (l_{01}-k)/2 + 1$ we find
$$ 
\left(-1,-1,\frac{l_{01}}{2},0\right)
\ = \ 
\frac{1}{k} v_{02} + \frac{1}{2} w_1 
+ \left(\frac{1}{2}-\frac{1}{k}\right) w_2
$$
and for $t = (l_{01}-k)/2 - 1$ we find
$$ 
\left(-1,-1,\frac{l_{01}}{2}-1,0\right)
\ = \ 
\frac{1}{k} v_{02} + \frac{1}{2} w_1 
+ \left(\frac{1}{2}-\frac{1}{k}\right) w_2.
$$
If $k$ is odd, then having no lattice points in
$\partial {A}_{X}^c (\lambda)^\circ$
implies $t = (l_{01}-k \pm 1)/2$.
For both choices of $t$, this setting produces a 
compound du Val singularity with matrix (12-o-e/o) 
and parameters $k_1\geq k_2$.

Before entering the discussion of the cases 
$v \in  \lambda_i^\circ$ with $i=1,2$,
the parameter $k$ occurring in $v$ 
might be $k=1,2$ and the vertex $w_2$
is given by 
$$
w_2
\ = \ 
\begin{cases}
\left(0,0,\frac{2tl_{01}}{2l_{01}+2k-kl_{01}},1\right),
&
v=(k,0,t,1) \in \lambda_1^\circ,
\\[2ex]
\left(0,0,\frac{2tl_{01}+kl_{01}}{2l_{01}+2k-kl_{01}},1\right),
&
v=(0,k,t,1) \in \lambda_2^\circ.
\end{cases}
$$

\medskip

\noindent
\emph{Case 1.1:} We have $l_{01} \equiv 0 \mod 4$.
If $v \in \lambda_2^\circ$ or $v=(2,0,t,1) \in \lambda_1^\circ$,
then we find one of $(0,0,l_{01}/2 \pm 1,1)$ 
in $\partial A_X^c (\lambda)^\circ$. 
Thus, we are left with $v \in \lambda_1^\circ$ and
$k=1$.
For any $t \ne l_{01}/4 + 1$, we find the 
lattice point
$(1,0, l_{01}/4 +1,1)$ in  $\partial {A_X^c}(\lambda_1)^\circ$.
Thus, we end up with 
$$ 
v 
\ = \
(k,0,t,1)
\ = \ 
(1,0, l_{01}/4 +1,1),
\qquad
w_2
\ = \ 
(0,0, l_{01}(l_{01}+4)/(2l_{01}+4),1).
$$
Note that the segment $\partial A_X^c (\lambda)$ contains 
no lattice points, because its length equals $l_{01}/(l_{01}+2)<1$. 
Taking a look at $\lambda_0$, we observe
$$
(-1,-1, l_{01}/2,0) 
\in 
\partial A_X^c(\lambda_0)^\circ
\iff
\frac{l_{01}}{l_{01}+2} 
>
\frac{l_{01}}{2(l_{01}-1)}
\iff
l_{01} > 4.
$$
Thus, to obtain compound du Val singularities, 
we must have $l_{01} \le 4 $.
As $l_{01} \equiv 0 \mod 4$ holds, only $l_{01} = 4$ is left
and, indeed, this leads to the compound du Val 
singularity with defining matrix~(16).

\medskip

\noindent
\emph{Case 1.2:} We have $l_{01} \equiv 2 \mod 4$.
If $v \in \lambda_1^\circ$ or $v=(0,2,t,1) \in \lambda_2^\circ$,
then we find one of $(0,0,l_{01}/2 \pm 1,1)$ 
in $\partial A_X^c (\lambda)^\circ$. 
Thus, we are left with $v \in \lambda_2^\circ$ and
$k=1$.
For any $t \ne l_{01}/4 + 1/2$, we find the 
lattice point
$(0,1, l_{01}/4 + 1/2,1)$ in $\partial A_X^c (\lambda_2)^\circ$. 
We end up with 
$$ 
v 
\ = \ 
(0,k,t,1)
\ = \ 
(0,1, l_{01}/4 + 1/2,1) 
\ \in \ 
\lambda_2^\circ.
$$ 
Similar to Case~1.1, we obtain that 
$(-1,-1, l_{01}/2,0) \in \partial A_X^c(\lambda_0)^\circ$  
as soon as $l_{01} > 4$.
Thus, only $l_{01} = 2$ might lead to a 
compound du Val singularity.
In this case, we exchange the data of blocks~0 and~2
and land in case of leading block data 
$(l_{01},2,2;0,1,0)$ 
and an additional column in $\lambda_0^\circ$.

\medskip

\noindent
\emph{Case 2:} The exponent $l_{01}$ is odd. 
If $v \in \lambda$, 
then $v = v_1 = w_2 = (0,0, (l_{01}+1)/2,1)$ holds
and we arrive at the compound du Val singularity 
with defining matrix~(5e). 
If $v \in \lambda_0^\circ$ holds, then the 
arguing runs similar as in Case~1.
Only for $k$ odd and 
$v = v_{02} = (-k,-k,(l_{01}-k+1)/2,1-k)$,
there are no lattice points  
in the relative interior of 
$\partial A_X^c (\lambda)^\circ$
and we end up with the
matrix~(12-o-e/o) as in Case 1, 
but now with parameters $k_1 \leq k_2$.

Assume $v \in \lambda_1^\circ$.
Then $v = v_{12} = (k,0,t,1)$ with $k=1,2$.
The case $k=2$ gives $w_2=(0,0,tl_{01}/2,1)$,
the point $(0,0,(l_{01} \pm 1)/2,1)$ lies 
$\partial A_X^c (\lambda)^\circ$ and
thus we leave the compound du Val case.
So, let $k=1$.
Then we have $v = (1,0,t,1) \in \lambda_1^\circ$.
Moreover,
$w_1=(0,0, l_{01} / 2 ,1)$ and 
$w_2=(0,0,2tl_{01}/(2+l_{01}),1)$.
Now, as $l_{01}$ is odd, we see
that $\partial A_X^c (\lambda)$ 
to have no lattice points in the
relative interior means 
$$
\frac{1}{2} 
\ \geq \ 
\left| \frac{2tl_{01}}{2+l_{01}} -\frac{l_{01}}{2} \right|.
$$
If $l_{01} \equiv 1 \mod 4$, this is only fulfilled 
for $t=(l_{01}+3)/4$. 
If $l_{01} \equiv 3 \mod 4$, it is only fulfilled for
$t=(l_{01}+1)/4$. 
Altogether, it is fulfilled for
$t= \left\lceil  l_{01} / 4 \right\rceil$. 
This leads to the compound du Val singularity 
with defining matrix (10o).

The case $v \in \lambda_2^\circ$ can be transformed 
by suitable admissible operations to the 
case $v \in \lambda_1^\circ$ just discussed.

Let $P$ have leading block data 
$(l_{01},l_{11},1;\mathfrak{d}_0,0,0)$. 
As $P$ is irredundant, the additional column 
is forced to be  $(0,1,t,1) \in \lambda_2^\circ$
and we have $l_{01},l_{11} \geq 2$. 
The vertices of $\partial A_X^c (\lambda)$ 
turn out to be  
$$
w_1 
\ = \  
\left(0,0, \frac{\mathfrak{d}_0l_{11}}{l_{01}+l_{11}},1\right),
\qquad
w_2 
\ = \ 
\left(0,0, \frac{\mathfrak{d}_0l_{11} + t l_{01} l_{11}}{l_{01}+l_{11}},1\right).
$$
We have $0 < tl_{01} l_{11}/(l_{01}+l_{11}) \le 1$ 
only for $t=1$ and $l_{01}=l_{11}=2$.
In this case, the second inequality becomes an equality
and thus $w_1$ is integral $w_1$ which implies 
$\mathfrak{d}_0=0$.
We arrive at the compound du Val singularity with matrix~(12-e-e)
and parameters $k_1=k_2=1$.
\end{proof}

We turn to the non-toric non-$\QQ$-factorial threefolds, 
still of canonical multiplicity one.
The following observation provides the link 
to the $\QQ$-factorial case.
Given defining data $A,P$ for a ring 
$R(A,P)$ of Type~2, 
we will have to deal with quadratic
submatrices $P'$ of $P$, obtained 
by erasing columns and rows from $P$.
The corresponding submatrix $A'$ of $A$
gathers all columns $a_i$ of $A$ 
such that at least one column $v_{ij}$ 
is not erased from $P$ when passing to $P'$.

\begin{lemma}
\label{le:nonQfact}
Let $X=X(A,P)$ be a compound du Val threefold of Type~2 
and canonical multiplicity $\zeta_X$ 
with $P$ irredundant in the form of Proposition~\ref{prop:zeta} 
and ordered in the sense of 
Remark~\ref{rem:leadPlat}. 
\begin{enumerate}
\item 
Let $P'$ be an $(r+2) \times (r+2)$ submatrix of $P$
such that for any $i = 0, \ldots, r$ at least 
one $v_{ij}$ is not erased from $P$.
\begin{enumerate}
\item
$A'=A$ and $P'$ are defining data of Type~2
in the sense of Construction~\ref{constr:RAPdown};
moreover, $P'$ is in the form of 
Proposition~\ref{prop:zeta}.
\item
$X' = X(A',P')$ is a $\QQ$-factorial threefold with at most 
compound du Val singularities 
of canonical multiplicity $\zeta_{X'}=\zeta_X$.
\end{enumerate}
Moreover, one always finds a submatrix $P'$ as above 
being ordered and having the same leading block as $P$.
\item
Every $P'$ as in~$(i)$ admits a $4 \times 4$ submatrix $P''$ 
with the same leading block as $P'$ such that 
\begin{enumerate}
\item
$A''$ and $P''$ are defining data of Type~2
in the sense of Construction~\ref{constr:RAPdown},
the matrix $P''$ is ordered and the form of 
Proposition~\ref{prop:zeta}.
\item
The varieties $X' = X(A',P')$ and $X'' = X(A'',P'')$
are equivariantly isomorphic to each other.
\end{enumerate}
\item 
If the leading platonic triple of $P$ is 
different from $(x,y,1)$, then $r=2$ holds.
\item 
One always finds $P'$ and $P''$ as in~$(ii)$ with 
the same leading block as $P$ such that 
\begin{enumerate}
\item 
in case of the leading platonic triple of $P$ 
differing from $(x,y,1)$,
up to admissible operations, $P''$ is one 
of the matrices from Proposition~\ref{prop:Qfac}. 
\item 
in case of the leading platonic triple of $P$  
equal to $(x,y,1)$, we have $n_2'' = 2$ for $P''$. 
\end{enumerate}
\end{enumerate}
\end{lemma}

\begin{proof}
We verify~(i). Note that each column of $P'$ is 
as well a column of $P$.
By Proposition~\ref{prop:affchar}, the columns 
of $P$ generate the extremal rays of a full 
dimensional cone $\sigma \subseteq \QQ^{r+2}$. 
Thus, also the columns of $P'$ generate
the extremal rays of a cone $\sigma' \subseteq \QQ^{r+2}$.
We show that $\sigma'$ is full dimensional.
If~$P'$ has a column $v_1 \in \lambda$,
then, using Proposition~\ref{prop:acancompstruct}~(iii)
we see that  the remaining $r+1$ columns of $P'$ 
are linearly independent and $v_1$ does not lie in 
their linear span.
If $P'$ has no column inside $\lambda$, then 
we can form two different $P$-elementary cones 
$\tau_1$ and $\tau_2$ out of columns of~$P'$.
The corresponding $v_{\tau_i} \in \tau_i^\circ$
generate the pointed two-dimensional cone 
$\sigma' \cap \lambda$ and we see that
the columns of $P$ generate $\QQ^{r+2}$.
Thus, we can conclude that $P'$ satisfies the conditions of
Type~2 of Construction~\ref{constr:RAPdown}
and, together with $A' = A$ gives defining data.
Observe that $X' = X(A',P')$ 
is $\QQ$-factorial by construction.
Using Remark~\ref{rem:computezeta}, we 
obtain $\zeta_{X'}=\zeta_X$
and see that $P'$ still is 
in the form of Proposition~\ref{prop:zeta}.
Using Remark~\ref{rem:zetaPprops}, 
conclude $\imath_{X'}=\imath_X=1$.
Moreover, according to 
Proposition~\ref{prop:acancompstruct},
the anticanonical complex $A_{X'}^c$ is a 
subcomplex of $A_X^c$ and the same holds 
for $\partial A_{X'}^c$ and $\partial A_{X}^c$.
Thus, Proposition~\ref{prop:cpl1hollow} shows 
that~$X'$ inherits from~$X$ the 
property of having at most compound du Val 
singularities.
The supplement is obvious.

We prove~(ii). 
For $r=2$, there is nothing to show. 
So, assume $r \geq 3$. 
If $P'$ has a column $v_k \in \lambda$,
then we have $n_i=l_{i1}=1$ for $i \geq 3$ and
Remark~\ref{rem:redundant}, applied $r-2$ times,
yields the desired $4 \times 4$ matrix $P''$.
We turn to the case that $P'$ has no column 
in $\lambda$.
Then $n_k = 2$ for some $0 \le k \le r$ and 
all other $n_i$ equal one.
If $k \le 2$ holds, then we have $n_i=l_{i1}=1$ 
for $i \geq 3$ and proceed as before to obtain~$P''$.
We discuss $k=3$.
First assume that the leading platonic triple 
of $P'$ equals $(x,y,1)$. 
Then, exchanging the data of column blocks $3$ 
and $2$ of $P'$, we are in the case $k \le 2$
just treated.
If the leading platonic triple of $P'$ differs
from $(x,y,1)$ 
then, applying $r-3$ times Remark~\ref{rem:redundant},
we arrive at an irredundant $5 \times 5$ matrix $P''$
defining a variety $X'' = X(A'',P'')$ isomorphic to 
$X' = X(A',P')$;
a contradiction to Proposition~\ref{prop:Qfac}.
Finally, if $k \ge 4$, then we exchange the data 
of column blocks~$k$ and $3$ of $P'$ and are 
in the case $k=3$.
This proves~(ii).

We turn to~(iii).  Assume $r \geq 3$. 
Since $P$ is irredundant and ordered in 
the sense of Remark~\ref{rem:leadPlat}, 
we have $n_i \ge 2$ and $l_{ij} = 1$ 
for $i \ge 3$.  
Consider the submatrices
$$
P' 
\ := \ 
[v_{01}, v_{11}, v_{21}, v_{31}, v_{32}, v_{41}, \ldots, v_{r1}],
\qquad 
P^\sim
\ := \ 
[v_{01}, v_{11}, v_{21}, v_{31}, v_{32}].
$$
Let $P''$ be the matrix obtained by erasing from $P^\sim$ 
erasing all but the first three and the last two rows.
Then $P''$ is an irredundant $5 \times 5$
matrix and $X'' = X(A'',P'')$ is isomorphic 
to $X' = X(A',P')$; a contradiction to Proposition~\ref{prop:Qfac}.

Finally, we show~(iv). 
For~(a), observe that because of 
$\imath_{X''}=\imath_X=1$, 
Proposition~\ref{prop:zetaexcep} 
gives $\zeta_{X''}=\zeta_X=1$. 
Thus $X''$ is $\QQ$-factorial compound du Val and $P''$ must, 
up to admissible operations, be one of the matrices 
from Proposition~\ref{prop:Qfac}.
We turn to~(b).
For any $i \ge 2$, we have $n_i \ge 2$,
because $P$ is irredundant.
Consider the submatrices
$$
P'
\ := \ 
[v_{01}, v_{11}, v_{21}, v_{22}, v_{31}, \ldots, v_{r1}],
\qquad 
P^\sim
\ := \ 
[v_{01}, v_{11}, v_{21}, v_{22}].
$$
Then we obtain the desired $P''$ from $P^\sim$ by 
erasing all but the first two and the last two rows.
\end{proof}

\goodbreak

\begin{proposition} 
\label{prop:nonQfac}
Let $X = X(A,P)$ be a non-toric affine threefold of Type~2.
Assume that $X$ is not $\QQ$-factorial, 
of canonical multiplicity one 
and has at most compound du Val singularities. 
Then $P$ can be assumed to be the matrix
$$
\text{(10-e)}
\
\begin{bmatrix}
  -k & 2 & 1 & 0 & 0  \\
	-k & 0 & 0 & 2 & 1 \\
	1 & 0 & 0 & 0  & 0\\
	1-k & 1 & 1 & 1 & 1
 \end{bmatrix},
\qquad 
k \in \ \ZZ_{\ge 2}.
$$
\end{proposition}

\begin{proof}
The strategy is to look first for 
not necessarily irredundant
matrices $P''$ with $r''=2$ 
defining a $\QQ$-factorial 
$X'' = X(A'',P'')$ 
of canonical multiplicity one 
with at most compound du Val 
singularities.
Then we obtain,  up to 
admissible operations, all matrices 
$P$ with~$X(A,P)$ satisfying 
the assumptions of the proposition
by enlarging the $P''$ in the sense of 
Lemma~\ref{le:nonQfact}.
We organize the subsequent discussion 
according to the possible leading block 
data, as listed in 
Proposition~\ref{prop:leadingblockdata},
and treat pairs $P'', P$ sharing 
the same leading block data.
Note that we have $r=2$ for $P$ 
whenever the leading platonic triple
differs from $(x,y,1)$.

Consider the leading block data 
$(5,3,2;0,0,0)$.
Proposition~\ref{prop:Qfac} tells us 
that after suitable admissible operations,
we have
$$
P''
 \ = \ 
\begin{bmatrix}
  -5 & 3 & 0 & 0 \\
	-5 & 0 & 2 & 0 \\
	0 & 0 & 0 & 1 \\
	-4 & 1 & 1 & 1
 \end{bmatrix}.
$$
After performing the corresponding admissible operations 
on $P$, we find $P''$ as a submatrix of $P$. 
Moreover, $P$ has at least one further column
and thus a submatrix
$$
P'''
\ = \ 
\begin{bmatrix}
  -5 & 3 & 0 & * \\
	-5 & 0 & 2 & * \\
	0 & 0 & 0 & * \\
	-4 & 1 & 1 & *
 \end{bmatrix}.
$$
Lemma~\ref{le:nonQfact}~(i) says that $X''' = X(A''',P''')$
is $\QQ$-factorial, of canonical multiplicity one 
and with at most compound du Val singularities.
Thus, up to admissible operations, $P'''$
occurs in the list of Proposition~\ref{prop:Qfac}.
So, the last column must be one of
$$
(0,0,1,1), \qquad (0,0,-1,1)
$$
The first case is impossible, because the columns of the 
defining matrix $P$ are pairwise different.
For $(0,0,-1,1)$ as last column, the point $(0,0,0,1)$ 
lies in $\partial A_{X}^c(\lambda)^\circ$; 
a contradiction to Proposition~\ref{prop:cpl1hollow}.

The case of leading block data $(4,3,2;0,0,0)$
is treated by exactly the same arguments as the preceding 
case.

Consider the leading block data $(4,3,2;1,0,0)$. 
Again, Proposition~\ref{prop:Qfac} tells us that, 
up to admissible operations, we have 
$$
P''
\ = \
\begin{bmatrix}
  -4 & -1 & 3 & 0  \\
	-4 & -1 & 0 & 2  \\
	1 & 3 & 0 & 0  \\
	-3 & 0 & 1 & 1 
 \end{bmatrix}.
$$
Adapting $P$ by admissible operations, it comprises $P''$
as a submatrix. 
As before, we obtain a matrix $P'''$ by enhancing 
the leading block with a further column of~$P$,
which this time must be one of
$$ 
(-1,-1,3,0), \qquad (-1,-1,2,0).
$$ 
The first leads to two identical columns of $P$ 
and this is excluded.
For the second we find $(0,0,3,1)$ inside 
$\partial A_{X}^c(\lambda)^\circ$ 
and leave the compound du Val case.

The case of leading block data $(3,3,2;0,0,0)$
runs exactly as the case of $(5,3,2;0,0,0)$.

Consider the leading block data $(3,3,2;1,0,0)$. 
Here Proposition~\ref{prop:Qfac} leaves us with two 
possibilities for the submatrix $P''$ of the accordingly 
adapted $P$. The first possibility is
\begin{equation}
\label{eq:P332}
P''
\ = \
 \begin{bmatrix}
  -3 & -2 & 3 & 0  \\
	-3 & -2 & 0 & 2  \\
	1 & 1 & 0 & 0  \\
	-2 & -1 & 1 & 1 
 \end{bmatrix}
\end{equation}
with columns $v_{01},v_{02},v_{11},v_{21}$.
Using as above Proposition~\ref{prop:Qfac},
we arrive at three possibilities for submatrices 
$P''' = [v_{01},v_{11},v_{21},*]$;
with $\sigma = \cone(v_{01},v_{02},v_{11},v_{21})$,
we find  the following situation in the
$\partial A_{X}^c (\lambda_i) \cap  \sigma$: 
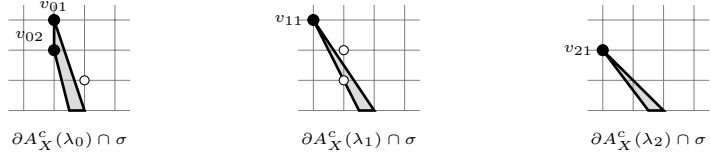
\begin{center}

\hfill
\begin{tikzpicture}[scale=0.4]
\fill[fill=gray!30!, fill opacity=0.90](2,0)--(1,3)--(1,2)--(3/2,0);
\draw[gray, very thin] (-0.5,0) grid (3.5,3.5);
\draw[line width=1pt] (2,0)--(1,3)--(1,2)--(3/2,0)--cycle;
\coordinate[label=above:\tiny $v_{01}$] (A) at (1,3);
\fill (1,3) circle (0.2);

\coordinate[label=above left:\tiny $v_{02}$] (A) at (1,2);
\fill (1,2) circle (0.2);

\fill[fill=white] (2,1) circle (0.15);
\draw (2,1) circle (0.15);
\node at (1.5,-1) {\tiny $\partial A_X^c(\lambda_0) \cap \sigma$};
\end{tikzpicture}
\hfill
\begin{tikzpicture}[scale=0.4]
\fill[fill=gray!30!, fill opacity=0.90](2,0)--(0,3)--(3/2,0);
\draw[gray, very thin] (-0.5,0) grid (3.5,3.5);
\draw[line width=1pt] (2,0)--(0,3)--(3/2,0)--cycle;
\coordinate[label=left:\tiny $v_{11}$] (A) at (0,3);
\fill (0,3) circle (0.2);
\fill[fill=white] (1,1) circle (0.15);
\draw (1,1) circle (0.15);

\fill[fill=white] (1,2) circle (0.15);
\draw (1,2) circle (0.15);

\node at (1.5,-1) {\tiny $\partial A_X^c(\lambda_1) \cap \sigma$};
\end{tikzpicture}
\hfill 
\begin{tikzpicture}[scale=0.4]
\fill[fill=gray!30!, fill opacity=0.90](2,0)--(0,2)--(3/2,0);
\draw[gray, very thin] (-0.5,0) grid (3.5,3.5);
\draw[line width=1pt] (2,0)--(0,2)--(3/2,0)--cycle;
\coordinate[label=left:\tiny $v_{21}$] (A) at (0,2);
\fill (0,2) circle (0.2);
\node at (1.5,-1) {\tiny $\partial A_X^c(\lambda_2) \cap \sigma$};
\end{tikzpicture}
\hfill \

\end{center}
where the circles indicate the prospective columns $*$ 
of $P'''$ leading to compound du Val singularities $X(A''',P''')$ 
of canonical multiplicity one. They are
$$
(-1,-1,2,0) \in \lambda_0, 
\quad
(1,0,1,1),
\
(2,0,1,1) \in \lambda_1.
$$
The lower one in the middle picture 
is contained in $\sigma$ which is not possible.
The other two force $(0,0,2,1)$ 
to lie in $\partial A_X^c(\lambda)^\circ$
which is as well impossible. 
So, (\ref{eq:P332}) does not occur as a submatrix
of $P$.
The second possibility is 
$$
P''
\ = \ 
 \begin{bmatrix}
  -3 & -1 & 3 & 0  \\
	-3 & -1 & 0 & 2  \\
	1 & 2 & 0 & 0  \\
	-2 & 0 & 1 & 1 
 \end{bmatrix}.
$$
Here we proceed analogously as with~(\ref{eq:P332}) and see
the only possible additional column in $P$ is $(1,0,1,1)$. 
In this case again $(0,0,2,1)$ lies in 
$\partial A_X^c(\lambda)^\circ$ 
and we leave the compound du Val case.

Consider the leading block data $(l_{01},2,2;0,0,0)$.
Here Proposition~\ref{prop:Qfac} tells us that 
the submatrix $P''$ of the accordingly adapted $P$ is
$$
P''
\ = \
\begin{bmatrix}
  -k_1 & -k_2 & 2 & 0  \\
	-k_1 & -k_2 & 0 & 2  \\
	0 & 1 &  0 & 0  \\
	1-k_1 & 1-k_2 & 1 & 1 
 \end{bmatrix},
$$
where we allow $k_2=0$ here and in this case change the 
second and fourth column to have a proper defining matrix. 
A possible further column for $P'''$ must have the form 
$(-k_3,-k_3,t,1-k_3)$ with $t=\pm 1$. For $t=1$, 
one of $(-k_2,-k_2,1,1-k_2)$ or $(-k_3,-k_3,1,1-k_3)$ 
does not give an extremal ray of the cone spanned 
by the columns of $P$. 
For $t=-1$, the point $(0,0,0,1)$ lies in  
$\partial A_X^c(\lambda)^\circ$ and we leave the 
compound du Val case.

Consider the leading block data $(l_{01},2,2;1,0,0)$.
Proposition~\ref{prop:Qfac} allows two choices for 
the submatrix $P''$ of the accordingly adapted $P$.
The first one is
$$
P''
\ = \
\begin{bmatrix}
  -k & 2 & 0 & 0  \\
	-k & 0 & 2 & 0  \\
	1 & 0 & 0 & 0  \\
	1-k & 1 & 1 & 1 
 \end{bmatrix}.
$$
We check the possible further columns of $P$.
A column in $\lambda$ would lead to 
$(0,0,1,1) \in \partial A_X^c(\lambda)^\circ$
and this is impossible. 
For any  $P'''$ sharing the first three columns with $P''$, 
the additional column, due to Proposition~\ref{prop:Qfac},
must be $(1,0,t,1)$ or $(0,1,t,1)$, where $t=0,1$. 
For $t=0$, such column would not generate an extremal 
ray of the cone spanned by the columns of $P$.
For $t=1$, we obtain $(0,0,1,1) \in \partial A_X^c(\lambda)^\circ$ 
and we leave the compound du Val case. 
The second choice is 
$$
P''
\ = \
\begin{bmatrix}
  -k & 2 & 1 & 0  \\
	-k & 0 & 0 & 2  \\
	1 & 0 & 0 & 0  \\
	1-k & 1 & 1 & 1 
 \end{bmatrix}.
$$
Proposition~\ref{prop:Qfac} tells us that 
$(1,0,t,1)$ or $(0,1,t,1)$ with $t=0,1$ 
are the only possible further columns of $P$.
But $(1,0,0,1)$ is impossible, since this 
column already exists in $P$ and for $(1,0,1,1)$, 
we obtain $(0,0,1,1) \in \partial A_X^c(\lambda)^\circ$. 
The same holds for $(0,1,1,1)$. 
For $(0,1,0,1)$, the line segment $\partial A_X^c(\lambda)$ has, 
in addition to $w_1=(0,0,1,1)$, the vertex
$$
w_2 \ = \ \left( 0,0,\frac{1}{1+k} ,1\right).
$$
If we have a look at the leaves, we see that we get a 
compound du Val singularity with defining matrix~(10-e):
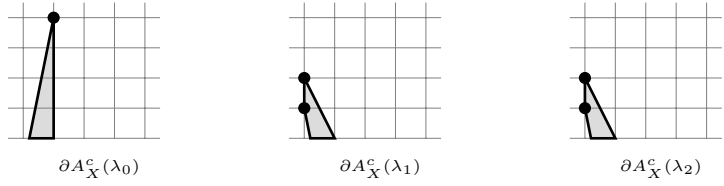
\begin{center}

\hfill
\begin{tikzpicture}[scale=0.4]
\fill[fill=gray!30!, fill opacity=0.90](1,0)--(1,4)--(1/5,0);
\draw[gray, very thin] (-0.5,0) grid (4.5,4.5);
\draw[line width=1pt] (1,0)--(1,4)--(1/5,0)--cycle;
\fill (1,4) circle (0.2);

\node at (2.5,-1) {\tiny $\partial A_X^c(\lambda_0)$};
\end{tikzpicture}
\hfill
\begin{tikzpicture}[scale=0.4]
\fill[fill=gray!30!, fill opacity=0.90](1,0)--(0,2)--(0,1)--(1/5,0);
\draw[gray, very thin] (-0.5,0) grid (4.5,4.5);
\draw[line width=1pt] (1,0)--(0,2)--(0,1)--(1/5,0)--cycle;
\fill (0,2) circle (0.2);
\fill (0,1) circle (0.2);
\node at (2.5,-1) {\tiny $\partial A_X^c(\lambda_1)$};
\end{tikzpicture}
\hfill 
\begin{tikzpicture}[scale=0.4]
\fill[fill=gray!30!, fill opacity=0.90](1,0)--(0,2)--(0,1)--(1/5,0);
\draw[gray, very thin] (-0.5,0) grid (4.5,4.5);
\draw[line width=1pt] (1,0)--(0,2)--(0,1)--(1/5,0)--cycle;
\fill (0,2) circle (0.2);
\fill (0,1) circle (0.2);
\node at (2.5,-1) {\tiny $\partial A_X^c(\lambda_2)$};
\end{tikzpicture}
\hfill \

\end{center}

Consider the leading block  data $(l_{01},2,2;0,1,0)$. 
Proposition~\ref{prop:Qfac} allows four possible 
submatrices $P''$ of the suitably adapted $P$.
We distinguish the following cases.

\medskip

\noindent
\emph{Case 1:}
The exponent $l_{01}$ is odd. 
First assume $P$ has after suitable admissible operations a submatrix
\begin{equation*}
P''
\ = \ 
 \begin{bmatrix}
  -2k_1-1 & - 2k_2 & 2  & 0  \\
	-2k_1-1 & - 2k_2  & 0  & 2  \\
	0 &  \frac{k_1-k_2 +1}{2}  & 1  & 0  \\
	-2k_1 & 1-2k_2 & 1 & 1 
 \end{bmatrix}.
\end{equation*}
Assume the matrix $P$ has a further column 
$(-k,-k,t,1-k)$ in $\lambda_0$. 
We regard the submatrix containing this further column
as well as the last two columns of $P''$ 
and either the first (if $k$ odd) 
or the second (if $k$ even) of $P''$. 
This matrix does not show up in Proposition~\ref{prop:Qfac} 
and we leave the compound du Val case.  
So $P$ can have no further column $(-k,-k,t,1-k)$.

Also an additional column $(0,0,t,1)$ 
in the lineality part is impossible,
 because due to Proposition~\ref{prop:Qfac}, 
 the only possibilities are $t=k_1$ and $t=k_1+1$.
But these would either not give an extremal ray
of the cone spanned by the columns of $P$ (for  $t=k_1+1$) 
or $(-1,-1,k_1,0)$ would show up in $\partial A_X^c(\lambda_0)^\circ$. 
Now the last possibility is an additional column 
$(1,0,t,1)$ in $\lambda_1$ or $(0,1,t,1)$ in $\lambda_2$. 
But the possible values of $t$, 
i.e. those giving a compound du Val submatrix 
of type (10-o) from Proposition~\ref{prop:Qfac}, 
either generate no extremal ray of the cone 
spanned by the columns of $P$ 
or $(-1,-1,k_1,0)$ is an interior point 
of $\partial A_X^c(\lambda_0)$.
Thus assume $P$ has, 
after suitable admissible operations, 
no submatrix of the above form and one
$$
P''=
\begin{bmatrix}
  -2k-1   & 2  & 1 & 0  \\
	-2k-1 & 0  & 0 & 2   \\
	0     & 1  & \left\lceil \frac{2k+1}{4} \right\rceil & 0  \\
	-2k   & 1  & 1 & 1
 \end{bmatrix}.
$$
Now, the submatrix of $P$ given by 
the first, second and third column 
of this submatrix and one further column 
must as well be of this form 
after suitable admissible operations. 
So the only possible additional column 
is $(0,1,\left\lceil (2k+1)/4 \right\rceil -1 ,1)$ in $\lambda_2$, 
but then $(-1,-1,k_1,0)$ is 
an inner point of $\partial A_X^c(\lambda_0)$ 
and we leave the compound du Val case.

\medskip

\noindent
\emph{Case  2:} The exponent $l_{01}$ equals $4$. 
After suitable admissible operations, 
the matrix $P$ has a submatrix
$$
P''=
\begin{bmatrix}
  -4 & 2 & 1 & 0  \\
	-4 & 0 & 0 & 2  \\
	0 & 1 & 2 & 0  \\
	-3 & 1 & 1 & 1 
 \end{bmatrix}.
$$
A further column must, together with the first two 
and the last row of $P''$,
give a compound du Val submatrix $P'''$ of $P$ as well.
So due to Proposition~\ref{prop:Qfac}, 
the only possible further column is $(1,0,1,1)$. 
But with this, the point $(0,0,2,1)$ 
is an inner point of $\partial A_X^c(\lambda)$ 
and we leave the compound du Val case.

Consider the leading block data 
$\left(l_{01},l_{11},1;\mathfrak{d}_0,0,0\right)$. 
Note that here, we also have to take care about 
redundant matrices $P''$.
Proposition~\ref{prop:Qfac} provides us with one 
irredundant matrix
$$
P''
\ = \
\begin{bmatrix}
  -2 & 2 & 0 & 0 \\
  -2 & 0 & 1 & 1 \\
  0 & 0 & 0 & 1 \\
  -1 & 1 & 0 & 0 
 \end{bmatrix}.
$$
The only possible further columns of $P$ are of 
the form 
$(-2,-2,t_0,-1)$, $(2,0,t_1,1)$ or $(0,1,t_2,0)$.
Each of them would stretch the segment 
$\partial A_X^c(\lambda)$ which already has 
the vertices $(0,0,0,1)$ and $(0,0,1,1)$.

Now we treat the redundant $P''$, which means 
to deal with $l_{11}=1$.
Due to Lemma~\ref{le:nonQfact}~(iv)~(b), 
after suitable admissible operations, 
the matrix $P$ has a submatrix 
$$
P''=
\begin{bmatrix}
 -l_{01} & 1  & 0 & 0 \\
 -l_{01} & 0 &  1 & 1 \\
 \mathfrak{d}_0  & 0 & 0 & t_2 \\
 1-l_{01} & 1  & 1 & 1 
 \end{bmatrix}.
$$
But since $P$ is irredundant, 
it must have a further submatrix
$$
P'''=
\begin{bmatrix}
 -l_{01} & 1 & 1 & 0 & 0 \\
 -l_{01} & 0 & 0 & 1 & 1 \\
 \mathfrak{d}_0  & 0 & t_1& 0 & t_2 \\
 1-l_{01} & 1 & 1 & 1 & 1 
 \end{bmatrix}
$$
comprising $P''$ and one further column in $\lambda_1$.
For this matrix and the vertices 
of the respective $\partial A_{X'''}^c(\lambda)$, we have
$$
w_1
\ = \ 
\left(0,0, \frac{\mathfrak{d}_0}{l_{01}+1},1\right),
\qquad
w_2
\ = \ 
\left(0,0, \frac{\mathfrak{d}_0+(t_1+t_2)l_{01}}{l_{01}+1},1\right),
$$
But $(t_1+t_2)l_{01}/(l_{01}+1)\leq 1$ 
only for $t_1=t_2=l_{01}=1$. 
But as $P$ is irredunbdant, it must have
a sixth column 
$(-1,-1,\mathfrak{d}_0+t_0,0)$ in $P$.
The distance between then
the vertices of $\partial A_X^c(\lambda)$ becomes
$$
\frac{t_0+t_1+t_2}{2} \ \geq \ \frac{3}{2}.
$$
Thus, $\partial A_X^c (\lambda)^\circ$ 
contains an integral point. 
So we obtain no compound du Val singularity in this case. 
\end{proof}

Finally, we have to deal with the non-toric 
threefolds of canonical multiplicity 
greater than one.

\setcounter{MaxMatrixCols}{20}
\begin{proposition} \label{prop:zeta>1}
Let $X  = X(A,P)$ be a non-toric affine threefold 
of Type~2.
Assume that~$X$ is of canonical multiplicity 
greater than one and has at most compound du Val 
singularities. 
Then one may assume $P$ to be one of the following 
matrices: 
$$
\text{(9)}
\quad
{\tiny\begin{bmatrix}
 -k     & -k      & \zeta_X-k & \zeta_X-k & 0 & 0 & \cdots &  0 & 0\\
 -k     & -k      & 0         & 0         & 1 & 1  & & 0 & 0\\
 \vdots & \vdots  & \vdots    & \vdots    &   &   & \ddots \\
 -k     & -k 	  & 0 		  & 0         &0  &0   &        & 1 & 1 \\
 0      & \mathfrak{d}_0     & 0         & \mathfrak{d}_1       & 0 & \mathfrak{d}_2& \cdots & 0 & \mathfrak{d}_r \\
 \frac{1-\mu k}{\zeta_X} & \frac{1-\mu k}{\zeta_X} & \frac{1-\mu k}{\zeta_X} + \mu & \frac{1-\mu k}{\zeta_X} + \mu & 0 & 0   &\cdots & 0 & 0
\end{bmatrix}}
\\
$$
$$
\text{(13-e)}
\
{\tiny\begin{bmatrix}
   -2\zeta_X +1 & 1 & 1  &0&0\\
-2\zeta_X +1 &	0 & 0   & 1 & 1 \\
	0 & 0 & 1  & 0  & 1 \\
	2 & 0 & 0  & 0 & 0 
 \end{bmatrix}}
 \quad
 \text{(13-o)}
 \
{ \tiny{
\begin{bmatrix}
 -2\zeta_X +2 & 2 &0&0\\
 -2\zeta_X +2 & 0  & 1 & 1 \\
 0 & 0 & 0  & 1 \\
 \zeta_X &  -1 & 0 & 0 
 \end{bmatrix}
}}
 \quad
 \text{(14)}
 \
 \tiny{\begin{bmatrix}
  -3  & 3 &0&0\\
	-3 & 0  & 1 & 1 \\
	0 & 0 & 0  & 1 \\
	-1 & 2 & 0 & 0 
 \end{bmatrix}}.
$$
In (9), $r\geq 2$ holds, the integers $\zeta_X \ge 2$ 
and $k \geq 1$ are coprime and $\mu$ 
is the unique integer $1 \leq \mu < \zeta_X$ 
with $\zeta_X \mid ( 1-\mu k)$.
Moreover $\mathfrak{d}_i \in \ZZ_{\ge 1}$ holds for $i \ge 0$
and if $k \geq 2$ ($\zeta_X-k \geq 2$), then one may erase the 
second (fourth) column of the matrix.
In (13-e), we have $\zeta_X \geq 2$. 
In (13-o), we have $\zeta_X \geq 3$ odd. 
In (14), we have $\zeta_X=2$.
\end{proposition}

\begin{proof}
The strategy is similar to that of the proof
of Proposition~\ref{prop:nonQfac}.
We look first for 
not necessarily irredundant
matrices $P''$ with $r=2$ and $n_2''= 2$
defining a $\QQ$-factorial 
$X'' = X(A'',P'')$ 
with at most compound du Val 
singularities and of canonical 
multiplicity bigger than one.
Lemma~\ref{le:nonQfact} then 
ensures that for~$X=X(A,P)$ 
satisfying the assumptions of the proposition,
the matrix $P$ contains, after suitable 
admissible operations, 
one of our $P''$ as a submatrix with the 
same leading platonic triple as $P$.
In other words, we can construct the possible 
$P$ by suitably enlarging~$P''$.

The matrix $P''$ we are looking for 
is $4 \times 4$.
Since $\zeta_{X''}>1$ holds,
we are in the setting 
of Proposition~\ref{prop:zetaexcep} 
and because of $\imath_{X''}=1$, we end up in 
Case~\ref{prop:zetaexcep}~(vi).
In addition to the leading block, we have 
the extra column $v_{22}$ in $P''$.
Moreover, the integer $\mu:=(1-\nu_{01}\zeta_{X''})/l_{01}$ 
as well as $l_{01}$ and $l_{11}$ must all be coprime to $\zeta_{X''}$, 
since we have the integer entries 
$\nu_{01} = (1-\mu l_{01})/\zeta_{X''}$ and 
$\nu_{11} = (1+\mu l_{11})/\zeta_{X''}$.
We also see that $\zeta_{X''}$ divides 
$l_{01} + l_{11}$ by subtracting $\nu_{01}$ and $\nu_{11}$ from 
each other. Now let 
$$
k_0 := \left\lfloor l_{01} / \zeta_{X''} \right\rfloor,
\qquad
k_1 := \left\lceil l_{11} / \zeta_{X''} \right\rceil,
\qquad
\delta := l_{01}-k_0\zeta_{X''}.
$$
Furthermore, let in this proof $\mathfrak{d}_{ij}$ be the 
third entry of the column $v_{ij}$ of $P''$.
With these definitions, our  matrix has the following shape
\begin{equation}
P''
\ = \ 
\label{eq:p-shape}
\begin{bmatrix}
  -( k_0 \zeta_{X''} + \delta)  & k_1\zeta_{X''} -\delta &0&0\\
	-( k_0 \zeta_{X''} + \delta) & 0  & 1 & 1 \\
	\mathfrak{d}_{01} & \mathfrak{d}_{11} & 0  & \mathfrak{d}_{22} \\
	\frac{1-\mu\delta}{\zeta_{X''}}-\mu k_0 & \frac{1-\mu\delta}{\zeta_{X''}}+\mu k_1 & 0 & 0 
 \end{bmatrix},
\end{equation}
where we  achieve $1 \leq \mu < \zeta_{X''}$ by subtracting 
the $\lfloor \mu / \zeta_{X''} \rfloor$-fold of the first 
from the last row, simultaneously. 
Moreover, we achieve $\mathfrak{d}_{01}=0$ by subtracting 
the $\mathfrak{d}_{01}\zeta_{X''}$-fold of the last and 
the $\mathfrak{d}_{01}\mu$-fold of the first from the 
penultimate row.
Exchanging, if necessary, the data of column blocks~0 and 1, 
we achieve $k_1 > k_0 \geq 0$. 
We now figure out those $P''$ defining a compound du Val 
singularity.
For this, we consider several constellations of $k_0$ 
and $k_1$.

\medskip

\noindent
\emph{Case 1:} We have $k_0=0$ and $k_1=1$. 
Here we can also achieve $\mathfrak{d}_{11}=0$ 
by subtracting the $\mathfrak{d}_{11}(1-\mu\delta)/\zeta_{X''}$-fold 
of the first and the $\mathfrak{d}_{11}\delta$-fold 
of the last from the penultimate row. 
The vertices of $\partial {A}_{X''}^c (\lambda)$ are
$$
w_1 
\ = \ 
\left(0,0,0,\frac{1}{\zeta_{X''}}\right),
\qquad\qquad
w_2 
\ = \ 
\left(0,0,\frac{\mathfrak{d}_{22}\delta(\zeta_{X''}-\delta)}{\zeta_{X''}}, \frac{1}{\zeta_{X''}}\right).
$$

We illustrate the situation for the case  
$\delta=2$, $\zeta_{X''}=5$, $\mathfrak{d}_{22}=2$ 
below; 
observe that the lineality part $\lambda$ contains 
no integer points and the union of the 
$\lambda_i \cap \mathcal{H}_i \cap \ZZ^{4}$ for $i=0,1$ 
is a sublattice

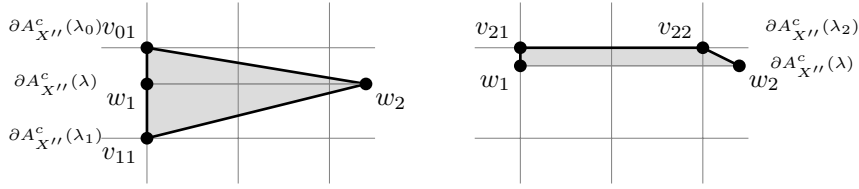
\begin{center}

\hfill
\begin{tikzpicture}[scale=1.2]
\fill[fill=gray!30!, fill opacity=0.90](0,1)--(0,0)--(12/5,3/5);
\draw[gray, very thin] (-0.5,-0.5) grid (2.5,1.5);
\draw[line width=1pt] (0,1)--(0,0)--(12/5,3/5)--(0,1);
\draw[line width=0.5pt, gray] (0,3/5)--(12/5,3/5);

\coordinate[label=above left: $v_{01}$] (A) at (0,1);
\fill (0,1) circle (0.07);

\coordinate[label=below left: $v_{11}$] (A) at (0,0);
\fill (0,0) circle (0.07);

\coordinate[label=below left: $w_{1}$] (A) at (0,3/5);
\fill (0,3/5) circle (0.07);

\coordinate[label=below right: $w_{2}$] (A) at (12/5,3/5);
\fill (12/5,3/5) circle (0.07);

\node at (-1,3/5) {\tiny $\partial A_{X''}^c(\lambda)$};
\node at (-1,6/5) {\tiny $\partial A_{X''}^c(\lambda_0)$};
\node at (-1,0/5) {\tiny $\partial A_{X''}^c(\lambda_1)$};
\end{tikzpicture}
\hfill 
\begin{tikzpicture}[scale=1.2]
\fill[fill=gray!30!, fill opacity=0.90](0,4/5)--(0,1)--(2,1)--(12/5,4/5);
\draw[gray, very thin] (-0.5,-0.5) grid (2.5,1.5);
\draw[line width=1pt] (0,4/5)--(0,1)--(2,1)--(12/5,4/5);
\draw[line width=0.5pt, gray] (0,4/5)--(12/5,4/5);

\coordinate[label=above left: $v_{21}$] (A) at (0,1);
\fill (0,1) circle (0.07);

\coordinate[label=above left: $v_{22}$] (A) at (2,1);
\fill (2,1) circle (0.07);

\coordinate[label=below left: $w_{1}$] (A) at (0,4/5);
\fill (0,4/5) circle (0.07);

\coordinate[label=below right: $w_{2}$] (A) at (12/5,4/5);
\fill (12/5,4/5) circle (0.07);

\node at (3.2,4/5) {\tiny $\partial A_{X''}^c(\lambda)$};
\node at (3.2,6/5) {\tiny $\partial A_{X''}^c(\lambda_2)$};
\end{tikzpicture}
\hfill \

\end{center}

The polytope $\partial A_{X''}^c(\lambda_0)$ does not contain integer points $(-k,-k,t,(1-\mu k)/\zeta_{X''})$ in its relative interior as for such integer points $k<\delta$ and $(1-\mu k)/\zeta_{X''}$ integral must hold, but $\delta$ is minimal with the second property. 
The same holds for $\partial A_{X''}^c(\lambda_1)$ and $\partial A_{X''}^c(\lambda_2)$ respectively. 
All points in $\partial A_{X''}^c (\lambda)$ have $1 / \zeta_{X''}$ as last coordinate, thus are not integral.
So, there is no integral point in the relative interior of $\partial A_{X''}^c$. 
Thus $P''$ defines a $\QQ$-factorial compound du Val singularity and meanwhile looks
as follows: 
\begin{equation}
\label{eq:p2}
\begin{bmatrix}
  -\delta  & \zeta_{X''} -\delta &0&0\\
	-\delta & 0  & 1 & 1 \\
	0 & 0 & 0  & \mathfrak{d}_{22} \\
	\frac{1-\mu\delta}{\zeta_{X''}} & \frac{1-\mu\delta}{\zeta_{X''}}+\mu & 0 & 0 
 \end{bmatrix},
\qquad
\gcd(\delta,\zeta_{X''})=1,
\quad
\mathfrak{d}_{22} \in \ZZ_{>0}.
\end{equation}

Now we check the possibilities of enlarging $P''$ 
in the sense of Lemma~\ref{le:nonQfact}
to a matrix $P$ defining a non-$\QQ$-factorial $X(A,P)$ 
as in the proposition.
As further columns we can insert one or both of
$$
v_{02}=\left(-\delta,-\delta,\mathfrak{d}_{02},\frac{1-\mu\delta}{\zeta_{X''}}\right),
\qquad
v_{12}=\left(\zeta_{X''}-\delta,0,\mathfrak{d}_{12},\frac{1-\mu\delta}{\zeta_{X''}} +\mu \right),
$$
with $\mathfrak{d}_{i2} \in \ZZ_{>0}$ arbitrary. 
We can not add other columns 
$(-k,-k,0,(1-\mu k)/\zeta_{X''})$ in $\lambda_0$. 
This is because first, $k \leq \delta$ must hold 
since $(\delta,\zeta_{X''}-\delta,1)$ 
is the leading platonic triple.
Second, $k = k' \zeta_{X''} +  \delta$ 
with $k' \geq 0 $ must hold. 
So we get $k=\delta$. 
But then one of the columns
$$
\left(-\delta,-\delta,\mathfrak{d}_{01},\frac{1-\mu\delta}{\zeta_{X''}}\right),
\qquad
\left(-\delta,-\delta,\mathfrak{d}_{02},\frac{1-\mu\delta}{\zeta_{X''}}\right),
\qquad
\left(-\delta,-\delta,\mathfrak{d}_{03},\frac{1-\mu\delta}{\zeta_{X''}}\right)
$$
lies in the cone spanned by the other two.
It can give no extremal ray of the cone 
spanned by the columns of $P$; a contradiction.
Exactly the same argument shows 
that no more columns can be added 
in $\lambda_1$ and $\lambda_2$.

Moreover, we can increase $r$ 
from two to arbitrary 
to get $P$ from $P''$. 
The leaves $\lambda_0,\ldots,\lambda_2$ 
stay untouched, we add new columns 
in leaves $\lambda_3,\ldots,\lambda_r$.
First we have $l_{ij}=1$, $n_i\geq 2$ 
for $i\geq 3$ due to log-terminality 
and irredundancy. 
Second, by the same argument as above 
for $\lambda_0,\ldots,\lambda_2$, 
we have $n_i\leq 2$. 
Thus $n_i=2$ holds for $i \geq 3$.
So $\lambda_i$ for $i\geq 3$ 
must have the same structure as $\lambda_2$ 
with two columns $e_i$ and $e_i+\mathfrak{d}_{i2} e_{r+1}$. 
Here $\mathfrak{d}_{i2} \in \ZZ_{>0}$ arbitrary 
and $e_j$ denotes the $j$-th basis vector. 
The distances $\mathfrak{d}_{i2}$ 
between $v_{i1}$ and $v_{i2}$ 
for $0 \leq i \leq r$ 
and in consequence between 
$w_1$ and $w_2$ may vary. 
Nevertheless, all polytopes 
$\partial A_X^c (\lambda_i)$ 
are subsets of polytopes 
of the second type 
of Proposition~\ref{prop:toriccDV} 
as also the following exemplary picture shows:
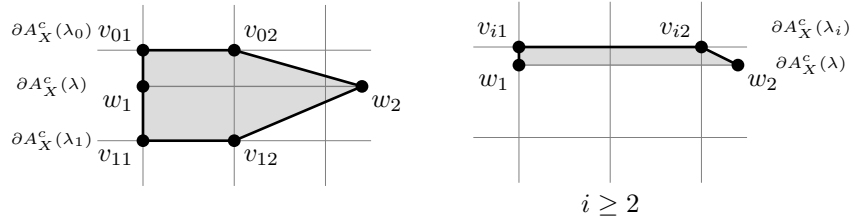
\begin{center}

\hfill
\begin{tikzpicture}[scale=1.2]
\fill[fill=gray!30!, fill opacity=0.90](0,1)--(0,0)--(1,0)--(12/5,3/5)--(1,1);
\draw[gray, very thin] (-0.5,-0.5) grid (2.5,1.5);
\draw[line width=1pt] (0,1)--(0,0)--(1,0)--(12/5,3/5)--(1,1)--(0,1);
\draw[line width=0.5pt, gray] (0,3/5)--(12/5,3/5);

\coordinate[label=above left: $v_{01}$] (A) at (0,1);
\fill (0,1) circle (0.07);

\coordinate[label=above right: $v_{02}$] (A) at (1,1);
\fill (1,1) circle (0.07);

\coordinate[label=below left: $v_{11}$] (A) at (0,0);
\fill (0,0) circle (0.07);

\coordinate[label=below right: $v_{12}$] (A) at (1,0);
\fill (1,0) circle (0.07);

\coordinate[label=below left: $w_{1}$] (A) at (0,3/5);
\fill (0,3/5) circle (0.07);

\coordinate[label=below right: $w_{2}$] (A) at (12/5,3/5);
\fill (12/5,3/5) circle (0.07);

\node at (-1,3/5) {\tiny $\partial A_X^c(\lambda)$};
\node at (-1,6/5) {\tiny $\partial A_X^c(\lambda_0)$};
\node at (-1,0/5) {\tiny $\partial A_X^c(\lambda_1)$};
\node at (1,-0.75) {\tiny $\color{white} i \geq 2 $};
\end{tikzpicture}
\hfill 
\begin{tikzpicture}[scale=1.2]
\fill[fill=gray!30!, fill opacity=0.90](0,4/5)--(0,1)--(2,1)--(12/5,4/5);
\draw[gray, very thin] (-0.5,-0.5) grid (2.5,1.5);
\draw[line width=1pt] (0,4/5)--(0,1)--(2,1)--(12/5,4/5);
\draw[line width=0.5pt, gray] (0,4/5)--(12/5,4/5);

\coordinate[label=above left: $v_{i1}$] (A) at (0,1);
\fill (0,1) circle (0.07);

\coordinate[label=above left: $v_{i2}$] (A) at (2,1);
\fill (2,1) circle (0.07);

\coordinate[label=below left: $w_{1}$] (A) at (0,4/5);
\fill (0,4/5) circle (0.07);

\coordinate[label=below right: $w_{2}$] (A) at (12/5,4/5);
\fill (12/5,4/5) circle (0.07);

\node at (3.2,4/5) {\tiny $\partial A_X^c(\lambda)$};
\node at (3.2,6/5) {\tiny $\partial A_X^c(\lambda_i)$};

\node at (1,-0.75) { $i \geq 2$};
\end{tikzpicture}
\hfill \

\end{center}
So for any $P$ of this form, 
there are no integral points 
in the relative interior of $\partial A_X^c$. 
Furthermore, as we have seen above, 
no more columns can be added in any leaf. 
In total, we get the series $(9)$ 
of defining matrices $P$
of compound du Val singularities.

\medskip

\noindent
\emph{Case 2:} We have $k_1 \geq 2$. 
Recall that we have $P''$ of shape~(\ref{eq:p-shape}) 
with $\mathfrak{d}_{01}=0$. 
Let $x_1, \ldots, x_4$ be the standard coordinates 
on the column space $\QQ^4$ of $P''$.  
Consider the line segments 
$\partial A_{X''}^c(\lambda)$ and
$$
L_{0,X''} 
 :=  
\partial A_{X''}^c(\lambda_0) \cap \{x_1=x_2=-\delta\},
\
L_{1,X''}
 :=  
\partial A_{X''}^c(\lambda_1) \cap \{x_1=\zeta_X-\delta, x_2=0 \},
$$
Let $w_1$, $w_2$ denote the vertices 
of $\partial A_{X''}^c(\lambda)$. 
Moreover, let $\omega_{01}$, $\omega_{02}$ 
be the vertices of $L_{0,X''}$
and $\omega_{11}$, $\omega_{12}$ 
the vertices of $L_{1,X''}$.
Then we have
\begin{align*}
w_1&=\left(0,0,\frac{\mathfrak{d}_{11}(k_0\zeta_{X''}+\delta)}{\zeta_{X''}(k_1+k_0)},\frac{1}{\zeta_{X''}}\right), \\
w_2&=w_1+ \mathfrak{d}_{22}\frac{(k_0\zeta_{X''}+\delta)(k_1\zeta_{X''}-\delta)}{\zeta_{X''}(k_1+k_0)} e_3, \\
\omega_{01}&=\left(-\delta,-\delta,\frac{\mathfrak{d}_{11}k_0\zeta_{X''}}{\zeta_{X''}(k_1+k_0)},\frac{1-\mu\delta}{\zeta_{X''}}\right), \\
\omega_{02}&=\omega_{01} +\mathfrak{d}_{22}\frac{(k_1\zeta_{X''}-\delta)k_0}{k_1+k_0} e_3 , \\
\omega_{11}&=\left(\zeta_{X''}-\delta,0, \mathfrak{d}_{11}\frac{\zeta_{X''} k_1-\delta k_0+\zeta_{X''} k_1 k_0-\delta}{(k_1\zeta_{X''}-\delta)(k_1+k_0)},\frac{1-\mu\delta}{\zeta_{X''}}+ \mu \right), \\
\omega_{12}&=\omega_{11} +  \mathfrak{d}_{22}\frac{(k_0\zeta_{X''}+\delta)(k_1-1)}{k_1+k_0} e_3 .
\end{align*}
Since there must be no integral point 
in the relative interior 
of the line segments $L_{0,X''}$ and $L_{1,X''}$, 
we at least require 
\begin{equation}
\label{eq:interval}
\mathfrak{d}_{22}\frac{(k_1\zeta_{X''}-\delta)k_0}{k_1+k_0} \leq 1,
\qquad 
\mathfrak{d}_{22}\frac{(k_0\zeta_{X''}+\delta)(k_1-1)}{k_1+k_0} \leq 1.
\end{equation}
These inequalities will be observed 
in the following different cases.

\smallskip

\noindent
\emph{Case 2.1:} We have $k_0=0$. 
Here, the inequalities~(\ref{eq:interval}) 
ease to $\mathfrak{d}_{22}\delta (k_1-1)/k_1 \leq 1$. 
We distinguish between $\delta=1$ and $\delta>1$.

\smallskip

\noindent
\emph{Case 2.1.1:} We have $\delta=1$. 
Here the matrix $P''$ is redundant. 
So any matrix $P$ with such submatrix 
must have an additional column in $\lambda_0$. 
We move on to a matrix $P$ 
also containing this additional column. 
Such matrix is of the form
$$
P=
\begin{bmatrix}
  -1  & -1 & k_1\zeta_X -1 &0&0\\
	-1 & -1 & 0  & 1 & 1 \\
	0 & \mathfrak{d}_{02} & \mathfrak{d}_{11} & 0  & \mathfrak{d}_{22} \\
	0 & 0 &  k_1 & 0 & 0 
 \end{bmatrix},
$$
where we can assume $\mathfrak{d}_{02}>0$. 
But here the length of the line segment $L_{1,X}$ is 
$$
(\mathfrak{d}_{2}+\mathfrak{d}_{02})(k_1-1)/k_1,
$$
which is less or equal to one
- which must hold if it 
does not contain an integral point - 
only for $\mathfrak{d}_{02}=\mathfrak{d}_{22}=1$ 
and $k_1=2$. 
Thus by adding multiples 
of the last to the penultimate row, 
we can assume that $\mathfrak{d}_{11}$ 
equals one or zero. If $\mathfrak{d}_{11}=1$, 
then the line segment $L_{1,X}$ has the vertices 
$$
\left(\zeta_X-1,0,\frac{2\zeta_X+1}{4\zeta_X-2},1\right),
\qquad
\left(\zeta_X-1,0,\frac{2\zeta_X+1}{4\zeta_X-2}+1,1\right).
$$
So it contains an integer point 
in its relative interior, 
since $(2\zeta_X+1)/(4\zeta_X-2)$ is not integral.
If $\mathfrak{d}_{11}=0$, 
then $L_{1,X}$ has the vertices 
$$
\left(\zeta_X-1,0,0,1\right),
\qquad
\left(\zeta_X-1,0,1,1\right)
$$
and thus contains no integer points. 
Since $L_{1,X}^\circ$ is the only subset 
of $\partial {A_X^c}^\circ$
that may contain integer points, 
we get the series of defining matrices 
$(13e)$ with arbitrary $\zeta_X$ from this. 

Such $P$ cannot again be the submatrix 
of a non-$\QQ$-factorial matrix 
with possibly larger $r$. 
This is because for any additional column 
in $\lambda_0,\ldots,\lambda_2$, 
the line segment $L_{1,X}$ 
would be stretched and then contain 
one its the former integral vertices 
$\left(\zeta_X-1,0,0,1\right)$ and 
$\left(\zeta_X-1,0,1,1\right)$. 
The same holds for additional leaves, 
which by irredundancy must contain 
at least two columns and also 
would lead to a stretching of $L_{1,X}$.

\smallskip

\noindent
\emph{Case 2.1.2:} We have $\delta>1$. 
Here  $\mathfrak{d}_{22}\delta (k_1-1)/k_1 \geq 2  (k_1-1)/k_1$ holds. Thus~(\ref{eq:interval}) is fulfilled 
only for $k_1=\delta=2$ and $\mathfrak{d}_{22}=1$. 
Moreover $\zeta_{X''}$ must be odd 
since $l_{01}=2$ is even. 
Also $\mu=(\zeta_{X''}+1)/2$ holds, 
i.e. we have the  matrix
$$
P''=
\begin{bmatrix}
 -2 & 2\zeta_{X''} -2 &0&0\\
 -2 & 0  & 1 & 1 \\
 0 & \mathfrak{d}_{11} & 0  & 1 \\
 -1 &  \zeta_{X''} & 0 & 0 
 \end{bmatrix}.
$$
By admissible operations, 
again $\mathfrak{d}_{11}$ 
can be assumed to be equal 
to zero or one. 
For $\mathfrak{d}_{11}=1$, 
the line segment $L_{1,X''}$ 
has the vertices
$$
\left(\zeta_{X''}-2,0,\frac{\zeta_{X''}-1}{\zeta_{X''}-2},\frac{\zeta_{X''}-1}{2}\right),
\qquad
\left(\zeta_{X''}-2,0,\frac{\zeta_{X''}-1}{\zeta_{X''}-2}+1,\frac{\zeta_{X''}-1}{2}\right),
$$
which have an integer point inbetween 
due to $(\zeta_{X''}-1)/(\zeta_{X''}-2)$ 
not being integral. 
In case $\mathfrak{d}_{11}$ equals zero, 
the segment $L_{1,X''}$ has the vertices 
$$
\left(\zeta_{X''}-2,0,0,\frac{\zeta_{X''}-1}{2}\right),
\qquad
\left(\zeta_{X''}-2,0,1,\frac{\zeta_{X''}-1}{2}\right).
$$
Since again $L_{1,X''}^\circ$ 
is the only subset of $\partial {A_{X''}^c}^\circ$ 
that may contain integer points, 
we get a compound du Val series 
with defining matrices (13-o) and odd $\zeta_{X''}$. 
With exactly the same argument as in Case 2.1.1, 
these matrices cannot serve as submatrices 
for other compound du Val defining matrices.

\smallskip

\noindent
\emph{Case 2.2:} We have $k_0 \geq 1$. 
Here, the first inequality 
of~(\ref{eq:interval}) leads to
\begin{equation}
\label{eq:ineq}
1 \leq k_0 \leq \frac{k_1}{k_1\zeta_{X''}-\delta - 1}
\ \Rightarrow \
0 \geq k_1(\zeta_{X''}-1)-\delta-1. 
\end{equation}
\smallskip

\noindent
\emph{Case 2.2.1:} We have $k_1 \geq 3$. 
Remembering $\delta<\zeta_{X''}$, 
we in total require 
$\delta < \zeta_{X''} \leq (\delta+4)/3$ 
from the above inequality~(\ref{eq:ineq}), 
leading to $1 = \delta < \zeta \leq 5/3$. 
This gives a contradiction, 
since  $\zeta_{X''}$ is integral.

\smallskip

\noindent
\emph{Case 2.2.2:} We have $k_1 =2$. 
The inequality~(\ref{eq:ineq}) gives 
$\delta < \zeta_{X''} \leq (\delta+3)/2$ here,
leading to $\delta<3$. 
While $\delta=2$ leads to $\zeta_{X''} \leq 5/2$, 
which contradicts $\delta <\zeta_{X''}$, 
the case $\delta=1$ allows $\zeta_{X''}=2$. 
The first inequality of~(\ref{eq:interval}) 
can only be fulfilled 
for $\mathfrak{d}_{2}=k_0=1$ here. 
Furthermore, $\mu=1$ must hold 
and inserting everything in~(\ref{eq:p-shape}), 
we get a defining matrix
$$
P''=
\begin{bmatrix}
  -3  & 3 &0&0\\
	-3 & 0  & 1 & 1 \\
	0 & \mathfrak{d}_{11} & 0  & 1 \\
	-1 & 2 & 0 & 0 
 \end{bmatrix}.
$$
Here in a first step, 
by admissible operations 
we can assume 
$\mathfrak{d}_{11} \in \{0,1,2\}$. 
In a second step, the vertices 
$$
\omega_{11}=\left(1,0,\mathfrak{d}_{11}\frac{2}{3},1 \right),
\qquad
\omega_{12}=\left(1,0,\mathfrak{d}_{11}\frac{2}{3}+1,1 \right)
$$
of the line segment $L_{1,X''}$ 
are integer only for $\mathfrak{d}_{11}=0$. 
Exactly the same holds for $L_{0,X''}$. 
So in this case, $P''$ itself gives 
the compound du Val defining matrix $(14)$. 
By the same arguments as in Case 2.1.1, 
these matrix cannot serve as submatrix 
for other compound du Val defining matrices.
\end{proof}

We now provide the necessary input for establishing
the defining equation in~$\CC^4$ of our 
compound du Val singularities.
Recall that the Cox ring $\mathcal{R}(X)$ of $X = X(A,P)$ 
is determined by the defining data,
where generators and relations are read off directly
and the degree matrix $Q$ of $\mathcal{R}(X)$, listing 
the generator degrees in $\Cl(X) = K$, needs to be computed.

\begin{proposition}
\label{prop:CLXQ9}
Consider $X = X(A,P)$ as in Case~(9) of 
Proposition~\ref{prop:zeta>1} with 
the defining matrix $P$ and the parameters 
therein.
As indicated there, we have four subcases:
$$ 
\setlength{\arraycolsep}{1.9pt}
\text{(9a)} \quad
P 
\ = \ 
{\tiny
\begin{bmatrix}
k & \zeta-k &0&0\\
k & 0  & 1 & 1 \\
0 & 0 & 0  & \mathfrak{d} \\
\frac{1-\mu k}{\zeta} &  \frac{1-\mu k}{\zeta} + \mu & 0 & 0 
\end{bmatrix}
},
\qquad
\text{(9b)} \quad
P 
\ = \ 
{\tiny
\begin{bmatrix}
 -k           & \zeta_X-k  & 0 & 0 & \cdots &  0 & 0\\
 -k           & 0                  & 1 & 1  & & 0 & 0\\
 \vdots   & \vdots        &   &   & \ddots \\
 -k     	  & 0 		           &0  &0   &        & 1 & 1 \\
 0           & 0                & 0 & \mathfrak{d}_2& \cdots & 0 & \mathfrak{d}_r \\
 \frac{1-\mu k}{\zeta_X}   & \frac{1-\mu k}{\zeta_X} + \mu & 0 & 0   &\cdots & 0 & 0
\end{bmatrix}
},
$$
$$
\setlength{\arraycolsep}{1.9pt}
\text{(9c)} \quad
P 
\ = \ 
{\tiny
\begin{bmatrix}
 -k           & \zeta_X-k & \zeta_X-k & 0 & 0 & \cdots &  0 & 0\\
 -k           & 0         & 0         & 1 & 1  & & 0 & 0\\
 \vdots   & \vdots    & \vdots    &   &   & \ddots \\
 -k     	  & 0 		  & 0         &0  &0   &        & 1 & 1 \\
 0           & 0         & \mathfrak{d}_1       & 0 & \mathfrak{d}_2& \cdots & 0 & \mathfrak{d}_r \\
 \frac{1-\mu k}{\zeta_X}  & \frac{1-\mu k}{\zeta_X} + \mu & \frac{1-\mu k}{\zeta_X} + \mu & 0 & 0   &\cdots & 0 & 0
\end{bmatrix}
},
$$
$$
\setlength{\arraycolsep}{1.9pt}
\text{(9d)} \quad
P 
\ = \ 
{\tiny
\begin{bmatrix}
 -k     & -k      & \zeta_X-k & \zeta_X-k & 0 & 0 & \cdots &  0 & 0\\
 -k     & -k      & 0         & 0         & 1 & 1  & & 0 & 0\\
 \vdots & \vdots  & \vdots    & \vdots    &   &   & \ddots \\
 -k     & -k 	  & 0 		  & 0         &0  &0   &        & 1 & 1 \\
 0      & \mathfrak{d}_0     & 0         & \mathfrak{d}_1       & 0 & \mathfrak{d}_2& \cdots & 0 & \mathfrak{d}_r \\
 \frac{1-\mu k}{\zeta_X} & \frac{1-\mu k}{\zeta_X} & \frac{1-\mu k}{\zeta_X} + \mu & \frac{1-\mu k}{\zeta_X} + \mu & 0 & 0   &\cdots & 0 & 0
\end{bmatrix}
}.
$$
According to these subcases, 
the divisor class group $\Cl(X)$ 
and the degree matrix $Q$ of the 
Cox ring $\mathcal{R}(X)$ are 
given as  follows:
\begin{itemize}
\item[\emph{(9a)}\enspace]
one has 
$\Cl(X) = \ZZ / \mathfrak{d}\ZZ$ 
and 
$
Q 
=
\begin{bmatrix}
\overline{0} & \overline{0} & \overline{1} & -\overline{1} 
\end{bmatrix}
$,
\item[\emph{(9b)}\enspace]
with $\mathfrak{d} := \gcd(\mathfrak{d}_2,\ldots,\mathfrak{d}_r)$
and integers $\alpha_i$ such that 
$\alpha_2\mathfrak{d}_2 + \ldots  + \alpha_r\mathfrak{d}_r = \mathfrak{d}$
holds, one has 
$\Cl(X) = \ZZ^{r-2} \times \ZZ / \mathfrak{d}\ZZ$ 
and
$$ 
\setlength{\arraycolsep}{1.9pt}
Q 
\ = \ 
{\tiny
\begin{bmatrix}
0  &0&   -\mathfrak{d}_3 & \mathfrak{d}_3 & \mathfrak{d}_2 & -\mathfrak{d}_2 &&0 & 0 \\
\vdots & \vdots & \vdots & \vdots & &&\ddots \\
 &  & -\mathfrak{d}_r & \mathfrak{d}_r & 0 & 0 & & \mathfrak{d}_2 & -\mathfrak{d}_2 \\
0  & 0 & -\overline{\alpha_2} & \overline{\alpha_2} & -\overline{\alpha_3} & \overline{\alpha_3}& \cdots & -\overline{\alpha_r} & \overline{\alpha_r}
 \end{bmatrix}
},
$$
\item[\emph{(9c)}\enspace]
with $\mathfrak{d} := \gcd(\mathfrak{d}_1,\ldots,\mathfrak{d}_r)$
and integers $\alpha_i$ such that 
$\alpha_1\mathfrak{d}_1 + \ldots  + \alpha_r\mathfrak{d}_r = \mathfrak{d}$
holds, one has 
$\Cl(X) = \ZZ^{r-1} \times \ZZ / \mathfrak{d}\ZZ$ 
and
$$ 
\setlength{\arraycolsep}{1.9pt}
Q 
\ = \ 
{\tiny
\begin{bmatrix}
0  & \mathfrak{d}_2 & -\mathfrak{d}_2 & -\mathfrak{d}_1 & \mathfrak{d}_1 \\
 
  &&  & -\mathfrak{d}_3 & \mathfrak{d}_3 & \mathfrak{d}_2 & -\mathfrak{d}_2 &&0 & 0 \\
\vdots &  && \vdots & \vdots & &&\ddots \\
 &&  & -\mathfrak{d}_r & \mathfrak{d}_r & 0 & 0 & & \mathfrak{d}_2 & -\mathfrak{d}_2 \\
0 & -\overline{\alpha_1} & \overline{\alpha_1} & -\overline{\alpha_2} & \overline{\alpha_2} & -\overline{\alpha_3} & \overline{\alpha_3}& \cdots & -\overline{\alpha_r} & \overline{\alpha_r}
 \end{bmatrix}
},
$$
\item[\emph{(9d)}\enspace]
with $\mathfrak{d} := \gcd(\mathfrak{d}_0,\ldots,\mathfrak{d}_r)$
and integers $\alpha_i$ such that 
$\alpha_0\mathfrak{d}_0 + \ldots  + \alpha_r\mathfrak{d}_r = \mathfrak{d}$
holds, one has 
$\Cl(X) = \ZZ^r \times \ZZ / \mathfrak{d}\ZZ$ 
and
$$ 
\qquad
\setlength{\arraycolsep}{1.9pt}
Q 
\ = \ 
{\tiny
\begin{bmatrix}
\mathfrak{d}_2 & -\mathfrak{d}_2 & 0 & 0 &  -\mathfrak{d}_0 & \mathfrak{d}_0  \\
0 & 0 & \mathfrak{d}_2 & -\mathfrak{d}_2 & -\mathfrak{d}_1 & \mathfrak{d}_1 \\
 
  &&&  & -\mathfrak{d}_3 & \mathfrak{d}_3 & \mathfrak{d}_2 & -\mathfrak{d}_2 &&0 & 0 \\
 &  &&& \vdots & \vdots & &&\ddots \\
 &&&  & -\mathfrak{d}_r & \mathfrak{d}_r & 0 & 0 & & \mathfrak{d}_2 & -\mathfrak{d}_2 \\
-\overline{\alpha_0} & \overline{\alpha_0} & -\overline{\alpha_1} & \overline{\alpha_1} & -\overline{\alpha_2} & \overline{\alpha_2} & -\overline{\alpha_3} & \overline{\alpha_3}& \cdots & -\overline{\alpha_r} & \overline{\alpha_r}
 \end{bmatrix}
}.
$$
\end{itemize}
\end{proposition}

\begin{proof}
Let $P^*$ be the transpose of $P$. Then we have 
$\Cl(X) \cong \ZZ^{n+m}/\im(P^*)$ and 
$Q \colon \ZZ^{n+m} \to \ZZ^{n+m}/\im(P^*)$ is 
the projection, 
see Construction~\ref{constr:RAPdown}. 
To describe $\Cl(X)$ and $Q$ explicitly, choose unimodular 
matrices $V$ and $W$ such that 
$S := V \cdot P^* \cdot W$ is in Smith Normal Form,
let $\beta_1,\ldots,\beta_\nu$ denote the elementary 
divisors and $\beta$ the number of zero rows of $S$.
Then 
$$
\Cl(X) \ \cong \ \ZZ^{\beta} 
\oplus 
\ZZ/ \beta_1 \ZZ
\oplus \ldots \oplus
\ZZ/ \beta_\nu \ZZ.
$$
Moreover, the matrix $Q$ is basically the stack of 
the last $\nu+\beta$ rows of $V$.
Now, we elaborate this explicitly for Case~(9d).
Set $\kappa_1:=(k\mu-1)/\zeta_X-\mu$
and $\kappa_2:= \zeta_X-k $. Then
$$
{
\setlength{\arraycolsep}{1.9pt}
V
:=
{\tiny
\begin{bmatrix}
\kappa_1 & 0 & \frac{1-k \mu}{\zeta_X} & 0 &
k\kappa_1 & 0 &k\kappa_1 & 0& \cdots & k\kappa_1 & 0 \\

&&&&1&0 \\
&&&&&&1 & 0  \\
&&&&&&&& \ddots \\
&&&&&&&&& 1 & 0 \\

\kappa_2 & 0 & k & 0 & k\kappa_2 & 0 & k\kappa_2 & 0 & \cdots &  k\kappa_2 & 0 \\

\mathfrak{d}_2 & -\mathfrak{d}_2 & 0 & 0 &  -\mathfrak{d}_0 & \mathfrak{d}_0  \\
0 & 0 & \mathfrak{d}_2 & -\mathfrak{d}_2 & -\mathfrak{d}_1 & \mathfrak{d}_1 \\
 
  &&&  & -\mathfrak{d}_3 & \mathfrak{d}_3 & \mathfrak{d}_2 & -\mathfrak{d}_2 &&0 & 0 \\
 &  &&& \vdots & \vdots & &&\ddots \\
 &&&  & -\mathfrak{d}_r & \mathfrak{d}_r & 0 & 0 & & \mathfrak{d}_2 & -\mathfrak{d}_2 \\
-{\alpha_0} & {\alpha_0} & -{\alpha_1} & {\alpha_1} & -{\alpha_2} & {\alpha_2} & -{\alpha_3} & {\alpha_3}& \cdots & -{\alpha_r} & {\alpha_r}
 \end{bmatrix} 
}
},
\qquad
W
:=
\begin{bmatrix}
E_r & 0 & 0 \\
0 & 0 & 1 \\
0 & 1 & 0
\end{bmatrix}
$$
are both unimodular matrices and turn the matrix $P^*$ into Smith Normal Form:
we have
$$
V \cdot P^* \cdot W = 
\left[
\begin{array}{cc}
E_{r+1} & 0 
\\
0 & \mathfrak{d}
\\
0 & 0  
\end{array}
\right].
$$
This proves the assertion for Case~9(d).
The other cases run similarly and will be presented 
elsewhere.
\end{proof}

\begin{proposition}
\label{prop:CLXQ}
Consider $X = X(A,P)$ with the defining 
matrix $P$ and the parameters therein as in 
Propositions~\ref{prop:toriccDV}, \ref{prop:Qfac},
\ref{prop:nonQfac} and~\ref{prop:zeta>1},
except Case~\ref{prop:zeta>1}~(9).
Then the divisor class group $\Cl(X)$ 
and the degree matrix $Q$ of the Cox ring 
$\mathcal{R}(X)$ are given as follows.

\renewcommand{\arraystretch}{1.1} 
\setlength{\tabcolsep}{1mm}
\setlength{\arraycolsep}{1.9pt}

\begin{longtable}{c|c|c}
P
&
$\Cl(X)$
& 
$Q$
\\
\hline
\ref{prop:toriccDV}~(1)
&
$\ZZ/k\ZZ$
&
$
\begin{bmatrix}
\overline{1} & \overline{0} & \overline{k-1} 
\end{bmatrix}
$
\\
\hline
\ref{prop:toriccDV}~(2)
&
$\ZZ \times \ZZ/k\ZZ$
&
$
\begin{bmatrix}
 k_2 & -k_1 & -k_2 & k_1 \\
-\overline{\alpha_1} 
& -\overline{\alpha_2} 
& \overline{\alpha_1} 
& \overline{\alpha_2}
\end{bmatrix}$
\\
&
$k := \gcd(k_1,k_2)$
&
$ 
\alpha_1k_1+\alpha_2k_2=k
$
\\
\hline
\ref{prop:toriccDV}~(3)
&
$\ZZ/2\ZZ \times \ZZ/2\ZZ$
&
$
\begin{bmatrix}
\overline{1} & \overline{0} & \overline{1} 
\\
\overline{0} & \overline{1} & \overline{1}
\end{bmatrix}
$
\\
\hline
$
\begin{array}{c}
\text{\ref{prop:Qfac}~(4)} 
\\
k \text{~even}
\end{array}
$
&
$
\ZZ/2\ZZ \times \ZZ/2\ZZ$
 &
$\begin{bmatrix}
    \overline{1} & \overline{0} & \overline{1} & \overline{0} \\
    \overline{0} & \overline{1} & \overline{1} & \overline{0}
 \end{bmatrix}
$
\\
\hline
$
\begin{array}{c}
\text{\ref{prop:Qfac}~(4)} 
\\
k \text{~odd}
\end{array}
$
&
$\ZZ/4\ZZ$ 
&
$
\begin{bmatrix}
\overline{2} & \overline{1} & \overline{3} & \overline{0}
\end{bmatrix}
$
\\
\hline
\text{\ref{prop:Qfac}~(5-e)}
&
$\ZZ/2\ZZ$
&
$
\begin{bmatrix}
\overline{0} & \overline{k+1} & \overline{k} & \overline{1} 
\end{bmatrix}
$
\\
\hline
\text{\ref{prop:Qfac}~(5-o)}
&
$\ZZ/2\ZZ \times \ZZ/2\ZZ$
&
$
\begin{bmatrix}
\overline{0} & \overline{1} & \overline{0} & \overline{1} 
\\
\overline{0} & \overline{1} & \overline{1} & \overline{0}
\end{bmatrix}
$
\\
\hline
\text{\ref{prop:Qfac}~(6)}
&
$\ZZ/3\ZZ$
&
$
\begin{bmatrix}
  \overline{1} & \overline{2} & \overline{0} & \overline{0} 
\end{bmatrix}
$
\\
\hline
\text{\ref{prop:Qfac}~(7)}
&
$\ZZ/2\ZZ$
&
$
\begin{bmatrix}
  \overline{1} & \overline{0} & \overline{1} & \overline{0} 
 \end{bmatrix}
$
\\
\hline
\text{\ref{prop:Qfac}~(8)}
&
$\{0\}$
& ---
\\
\hline
\text{\ref{prop:nonQfac}~(10-e)}
&
$\ZZ$
&
$
\begin{bmatrix}
  0 & 1 & -2 & -1 & 2
\end{bmatrix}
$
   \\
\hline
\text{\ref{prop:Qfac}~(10-o)}
&
$\{0\}$
&
\text{---}
\\
\hline
\text{\ref{prop:Qfac}~(11)}
&
$\ZZ/2\ZZ$
&
$
\begin{bmatrix}
  \overline{0} & \overline{1} & \overline{0} & \overline{1} 
 \end{bmatrix}
$
\\
\hline
$\begin{array}{c}
\text{\ref{prop:Qfac}~(12-e-e)} \\
k_1 \text{~even}
\end{array}
$
&
$\ZZ/2\ZZ \times \ZZ/2\ZZ$
&
$
\begin{bmatrix}
  \overline{1} & \overline{0} & \overline{1} & \overline{0} \\
  \overline{0} & \overline{0} & \overline{1} & \overline{1}
 \end{bmatrix}
$
 \\
\hline
$\begin{array}{c}
\text{\ref{prop:Qfac}~(12-e-e)} \\
k_1 \text{~odd}
\end{array}
$
&
$\ZZ/4\ZZ$
&
$
\begin{bmatrix}
  \overline{2} & \overline{0} & \overline{3} & \overline{1} 
 \end{bmatrix}
$
  \\
\hline
\ref{prop:Qfac}~(12-o-e/o)
&
 $\ZZ/2\ZZ$
&
$
\begin{bmatrix}
  \overline{1} & \overline{0} & \overline{0} & \overline{1} 
 \end{bmatrix}
$
  \\
\hline
\text{\ref{prop:zeta>1}~(13-e)}
 &
 $ \ZZ \times \ZZ/2\ZZ $
 &
 $
\begin{bmatrix}
 0 & 1 &  -1 & -1 & 1 \\
  \overline{1} & \overline{1} & \overline{0} & \overline{1} & \overline{0} 
 \end{bmatrix}
$
\\
\hline
\text{\ref{prop:zeta>1}~(13-o)}
 &
 $
 \ZZ/2\ZZ
 $
 &
 $
\begin{bmatrix}
   \overline{1} & \overline{1} & \overline{0} & \overline{0} 
 \end{bmatrix}
$
\\
\hline
\text{\ref{prop:zeta>1}~(14)}
 &
 $
 \ZZ/3\ZZ
 $
 &
 $
\begin{bmatrix}
   \overline{1} & \overline{2} & \overline{0} & \overline{0} 
 \end{bmatrix}
$
\\
\hline
\text{\ref{prop:Qfac}~(15, 17, 18)}
 &
$\{0\}$
&
\text{---}
\\
\hline
\text{\ref{prop:Qfac}~(16)}
 &
$\ZZ/2\ZZ$
&
$
\begin{bmatrix}
   \overline{1} & \overline{0} & \overline{0} & \overline{1} 
 \end{bmatrix}
$
\end{longtable}
\end{proposition}

\begin{proof}
The arguing is the same as for Proposition~\ref{prop:CLXQ9}
and will be explicitly presented elsewhere.
Note that the cases without parameters can easily be 
settled by computer, e.g.~using~\cite{HaKe}. 
\end{proof}

\begin{proof}[Proof of Theorem~\ref{thm:cdv-class-intro}]
Propositions~\ref{prop:toriccDV}, \ref{prop:Qfac},
\ref{prop:nonQfac} and \ref{prop:zeta>1} provide us 
with the defining matrices $P$ of the compound du 
Val threefold singularities $X$ of complexity one. 
This gives in particular their Cox rings 
$\mathcal{R}(X) = R(A,P)$.
The grading of the Cox ring by $\Cl(X) = K$ is given 
by the degree matrices $Q$ provided in 
Propositions~\ref{prop:CLXQ9} and~\ref{prop:CLXQ}.
We have $X = \Spec \, R(A,P)_0$ and will obtain the 
describing equation for $X \subseteq \CC^4$ from 
a suitable presentation of the degree zero part 
$R(A,P)_0$ of $R(A,P)$ by generators and relations.

We exemplarily carry this procedure out for the 
case~(9d) from Proposition~\ref{prop:CLXQ9}.
A glance at the degree matrix~$Q$  
given in Proposition~\ref{prop:CLXQ9}~(9d)
shows that the following monomials are of 
$K$-degree zero:
$$
x_i \ := \  T_{i1}T_{i2}, \ i = 0,\ldots,r,
\qquad
x_{r+1}  \ := \ T_{01}^{\mathfrak{d}_0} \cdots T_{r1}^{\mathfrak{d}_r},
\qquad 
x_{r+2}  \ := \ T_{02}^{\mathfrak{d}_0} \cdots T_{r2}^{\mathfrak{d}_r}.
$$ 
Obviously, any monomial $h$ in the $T_{ij}$ is 
a product of powers of $x_0, \ldots, x_{r+2}$ and 
a monomial $h'$ depending of at most one variable 
$T_{ij}$ per $i$ and at most on $r-1$ variables in 
total. 
By the shape of $Q$, such a monomial $h'$ is 
of degree zero if and only if it is constant.
We conclude that $x_0, \ldots, x_{r+2}$ generate
$R(A,P)_0$. 
Now consider the morphism
$$ 
\pi \colon \CC^{n} \ \to \ \CC^{r+3},
\qquad
z \ \mapsto \ (x_0(z), \ldots, x_{r+2}(z)). 
$$
Then $X = \Spec \, R(A,P)_0$ is the image of 
$\b{X} = \Spec \, R(A,P)$ under $\pi$. 
We claim that $X = \pi(\b{X}) \subseteq \CC^{r+3}$ 
is contained in the zero set  
of the polynomials
$$ 
x_{r+1}x_{r+2}-x_0^{\mathfrak{d}_0} \cdots x_r^{\mathfrak{d}_r},
$$
$$
x_0^k + x_1^{\zeta-k} +x_2, 
\quad 
x_1^{\zeta-k} + 2x_2 +x_3,
$$
$$
x_2 + 3x_3 + x_4, 
\quad \ldots, \quad 
x_{r-2}+(r-1)x_{r-1} + x_r.
$$  
Indeed, the first polynomial is an obvious relation 
between the $x_i$ and the remaining ones pull back
via $\pi$ to the Cox ring relations given by the 
matrix $P$.
The above relations allow elimination of variables
$x_3, \ldots, x_r$: starting with the last 
relation, we successively plug these into the first 
one and arrive at
$$
x_{r+1}x_{r+2}
-
x_0^{\mathfrak{d}_0}x_1^{\mathfrak{d}_1}
\prod_{i=2}^{r}(a_ix_0^k+b_ix_1^{\zeta-k})^{\mathfrak{d}_i},
$$
where we can by a suitable coordinate change achieve 
that $a_i=i-1$ and $b_i=2i-3$ hold for all $i=2,\ldots,r$.
Thus, $X$ can be realized as a closed subset inside the 
hypersurface $X' \subseteq \CC^4$ defined by the above 
polynomial. 
As the latter is irreducible, we conclude $X = X'$,
which proves the assertion in case (9d).

For the other non-factorial compound du Val singularities,
one argues analogously; we present this elsewhere. 
In the cases~(9a), (9b) and 
(9c), we obtain the following invariants $x_i$ and 
relations among them:

\begin{longtable}{c|c|c}
Case
&
$\begin{array}{c}
\text{Invariant monomials} \\
x_0,\ldots,x_\nu
\end{array}
$
& Relations among $x_i$
\\
\hline
(9a)
&
$T_{21}^{\mathfrak{d}},T_{22}^{\mathfrak{d}},T_{21}T_{22},T_{01},T_{11}$
&
$\begin{array}{c}
x_0x_1-x_2^{\mathfrak{d}} \\
x_3^k+x_4^{\zeta-k}+x_2
\end{array}
$
\\
\hline
(9b)
&
$ \begin{array}{c}
T_{01},T_{11},T_{21}T_{22},\ldots,T_{r1}T_{r2}, \\
T_{21}^{\mathfrak{d}_2} \cdots T_{r1}^{\mathfrak{d}_r}, T_{22}^{\mathfrak{d}_2} \cdots T_{r2}^{\mathfrak{d}_r}
\end{array}
$
&
$\begin{array}{c}
x_{r+1}x_{r+2}-x_2^{\mathfrak{d}_2} \cdots x_r^{\mathfrak{d}_r}  \\
x_{0}^k+x_{1}^{\zeta-k}+x_2 \\
\vdots \\
x_{r-2}+(r-1)x_{r-1}+x_r
\end{array}
$
\\
\hline
(9c)
&
$ \begin{array}{c}
T_{01},T_{11}T_{12},T_{21}T_{22},\ldots,T_{r1}T_{r2}, \\
T_{11}^{\mathfrak{d}_1} \cdots T_{r1}^{\mathfrak{d}_r}, T_{12}^{\mathfrak{d}_1} \cdots T_{r2}^{\mathfrak{d}_r}
\end{array}
$
&
$\begin{array}{c}
x_{r+1}x_{r+2}-x_1^{\mathfrak{d}_1} \cdots x_r^{\mathfrak{d}_r}  \\
x_{0}^k+x_{1}^{\zeta-k}+x_2 \\
\vdots \\
x_{r-2}+(r-1)x_{r-1}+x_r
\end{array}
$

\end{longtable}

For the cases different from 9, we list the relevant data in 
the following table; note that the cases without parameters
can be settled by computer, e.g., using~\cite{HaKe}:

\renewcommand{\arraystretch}{1.2} 
\setlength{\tabcolsep}{.2mm}

\begin{longtable}{c|c|c|c}
Case
&
Relations in $\mathcal{R}(X)$
&
$\begin{array}{c}
\text{Invariant monomials} \\
x_1,\ldots,x_\nu
\end{array}
$
& Relations among $x_i$
\\
\hline
(1)
&

&
$T_1^k,T_3^k,T_1T_3,T_2$
  &
$x_1x_2-x_3^k
 $

\\
\hline
(2)
&

&
$\begin{array}{c}

T_1^{k_1}T_2^{k_2},T_3^{k_1}T_4^{k_2}, \\
T_1T_3, T_2T_4

\end{array}
$
&
$
x_1x_2-x_3^{k_1}x_4^{k_2}
$
\\
\hline
(3)
&

&
$T_1T_2T_3,T_1^2,T_2^2,T_3^2$
&
$x_1^2-x_2x_3x_4$
\\
\hline
(4)
&
$\begin{array}{c}
T_1^k+T_2^2+T_3^2 \\
k \text{~even}
\end{array}$
&
$T_1T_2T_3,T_1^2,T_2^2,T_3^2,T_4$
 &
$\begin{array}{c}
x_1^2-x_2x_3x_4 \\
x_2^{k/2}+x_3+x_4
\end{array}
$
\\
\hline
(4)
&
$\begin{array}{c}
T_1^k+T_2^2+T_3^2 \\
k \text{~odd}
\end{array}$
&
$\begin{array}{c}
T_2^4,T_3^4,T_1T_2^2,T_1T_3^2, \\
T_2T_3,T_1^2,T_4
\end{array}$
&
$\begin{array}{c}
x_1x_2-x_5^4 \\
x_6^{{\frac{k-1}{2}}}x_3 + x_2 + x_5^2 \\
x_6^{{\frac{k-1}{2}}}x_4 + x_1 + x_4^2 \\
x_6^{{\frac{k+1}{2}}} + x_3 + x_4
\end{array}
 $
\\
\hline
(5-e)
&
$T_1^{2k+1}+T_2^2+T_3^2$
 &
 $ 
T_4^2,T_{2}^2,T_{2}T_4,T_1, T_3$
 
 &
 $
\begin{array}{c}
x_1x_2-x_3^2 \\
x_4^{2k+1}+x_5^2+x_2
\end{array}

$
  \\
\hline
(5-o)
&
$T_1^k+T_2^2+T_3^2$
 &
 $
 T_2T_3T_4, T_2^2, T_3^2, T_4^2, T_1 
 $
 &
 $\begin{array}{c}
 x_1^2-x_2x_3x_4 \\
 x_5^k+x_2+x_3
 \end{array}

$
  \\
\hline
(6)
&
$T_1^3+T_2^3+T_3^2$
 &
 $
 T_1^3,T_2^3,T_1T_2,T_3,T_4
 $
 &
 $\begin{array}{c}
x_1x_2-x_3^3 \\
x_1+x_2+x_4^2
\end{array}
$
  \\
\hline
(7)
&
$T_1^4+T_2^3+T_3^2$
 &
 $
T_1^2,T_3^2,T_1T_3,T_2,T_4
 $
 &
 $
\begin{array}{c}
x_1x_2-x_3^2 \\
x_1^2+x_4^3+x_2
\end{array}

$
  \\
\hline
(10-e)
&
$T_1^k+T_2^2T_3+T_4^2T_5$
&
 $\begin{array}{c}

 T_2^2T_3,T_4^2T_5, \\
 T_2T_4,T_3T_5,T_1
 
 \end{array} 
 $
 &
 $
\begin{array}{c}
x_1x_2-x_3^2x_4 \\
x_5^k+x_1+x_2
\end{array}
 $
   \\
\hline
(11)
&
$T_1^k+T_2^2T_3+T_4^2$
&
$T_2^2,T_4^2,T_2T_4,T_1,T_3$
&
$
\begin{array}{c}
x_1x_2-x_3^2 \\
x_4^k+x_1x_5+x_2
\end{array}
$
  \\
\hline
(12-e-e)
&
$\begin{array}{c}
T_1^{k_1}T_2^{k_2}+T_3^2+T_4^2 \\
k_1 \text{~even}
\end{array}

$
&
$T_1T_3T_4,T_1^2,T_3^2,T_4^2,T_2$
&
$
\begin{array}{c}
x_1^2-x_2x_3x_4 \\
x_2^{\frac{k_1}{2}}x_5^{k_2}+x_3+x_4
\end{array}
$
 \\
\hline
(12-e-e)
&
$\begin{array}{c}
T_1^{k_1}T_2^{k_2}+T_3^2+T_4^2 \\
k_1 \text{~odd}
\end{array}
$
&
$\begin{array}{c}

T_3^4,T_4^4,T_3^2T_1,T_4^2T_1, \\
T_3T_4,T_1^2,T_2

\end{array}$
&
$\begin{array}{c}
x_1x_2-x_5^4 \\
x_6^{\frac{k-1}{2}}x_7^{k_2}x_3 + x_2 + x_5^2 \\
x_6^{\frac{k-1}{2}}x_7^{k_2}x_4 + x_1 + x_4^2 \\
x_6^{\frac{k+1}{2}}x_7^{k_2} + x_3 + x_4
\end{array}
 $
  \\
\hline
(12-o-e/o)
&
$T_1^{2k_1}T_2^{2k_2+1}+T_3^2+T_4^2$
 &
 $T_1^2,T_4^2,T_1T_4,T_2,T_3$
&
$
\begin{array}{c}
x_1x_2-x_3^2 \\
x_1^{k_1}x_4^{2k_2+1}+x_5^2+x_2
\end{array}
$
  \\
\hline
(13-e)
&
$T_1^{2\zeta-1}+T_2T_3+T_4T_5$
 &
 $\begin{array}{c}
  T_2^2T_3^2,T_4^2T_5^2,T_1T_2T_3, \\
 T_1T_4T_5,T_1^2,T_2T_4,T_3T_5
 \end{array}$
 &
 $
\begin{array}{c}
x_3x_4-x_5x_6x_7 \\
x_1x_2-x_6^2x_7^2 \\
x_3^2-x_5x_1\\
x_4^2-x_5x_2 \\
x_5^{\zeta}+x_3+x_4 \\
x_5^{\zeta-1}x_3+x_1+x_6x_7 \\
x_5^{\zeta-1}x_4+x_2+x_6x_7
\end{array}
$
\\
\hline
(13-o)
&
$T_1^{2\zeta-2}+T_2^2+T_3T_4$
 &
 $
 T_1^2,T_2^2,T_1T_2,T_3,T_4
 $
 &
 $
\begin{array}{c}
x_1x_2-x_3^2 \\
x_1^{\zeta-1}+x_2+x_4x_5
\end{array}
$
\\
\hline
(14)
&
$T_1^3+T_2^3+T_3T_4$
 &
 $
 T_1^3,T_2^3,T_1T_2,T_3,T_4
 $
 &
 $
\begin{array}{c}
x_1x_2-x_3^3 \\
x_4x_5+x_1+x_2
\end{array}
$
\\
\hline
(16)
&
$T_1^4+T_2^2T_3+T_4^2$
 &
 $
 T_1^2,T_4^2,T_1T_4,T_2,T_3
 $
 &
 $
\begin{array}{c}
x_1x_2-x_3^2 \\
x_1^2+x_4^2x_5+x_2
\end{array}
$

\end{longtable}

\end{proof}

\begin{proof}[Proof of Theorem~\ref{thm:cdv-graph-intro}]
Theorem~\ref{thm:cdv-class-intro} gives us all 
compound du Val singularities of complexity one. 
The respective Cox rings finally can be 
computed using Remark~\ref{rem:ithelp}. 
\end{proof}

\end{document}